\documentclass[12pt]{amsart}
\usepackage{amsmath}
\usepackage{fullpage}
\usepackage{amsthm}
\usepackage{amssymb}
\usepackage{amsfonts}
\usepackage{tikz} 
\usepackage{yfonts}
\usepackage{mathtools}
\usepackage{setspace}
\usepackage{cancel}
\usepackage{chemarrow}
\usetikzlibrary{matrix,arrows} 
\usepackage[parfill]{parskip}    % Activate to begin paragraphs with an empty line rather than an indent
\usepackage{amscd}
\numberwithin{equation}{section}
\newtheorem{prop}{Proposition}[section]

\newtheorem{lemma}[prop]{Lemma}

\newtheorem{theorem}[prop]{Theorem}

\theoremstyle{definition}

\newtheorem{defn}[prop]{Definition}
\newtheorem{example}[prop]{Example}

\newtheorem{assumption}[prop]{Assumption}
\newtheorem{remark}[prop]{Remark}
\begin{document}

%\department{Mathematics}
%\degreeyear{2011}
%\committee{Professor Steven Bradlow, Chair\\ Professor Sheldon Katz, Director of research\\ Associate Professor Tom Nevins, Director of reserach\\ Associate Professor Hal Schenk\\ }
\title{Weighted Euler characteristic of the moduli space of  higher rank Joyce-Song  pairs}
\author{Artan Sheshmani}
%\address{Department of Mathematics at University of Illinois at Urbana Champaign, 1409 W. Green Street, Urbana, IL, 61801}
%\email{sheshman@math.uiuc.edu}
%\date{\today}                                   
%\maketitle
%\mainmatter
\maketitle

\pagenumbering{arabic}
%\doublespacing
\begin{abstract}
The invariants of rank 2 Joyce-Song semistable pairs over a Calabi-Yau threefold were computed in \cite{a41}, using the wall-crossing formula of Joyce-Song \cite{a30} and Kontsevich-Soibelman \cite{a42}. Such wallcrossing computations often depend on the combinatorial properties of certain elements of a Hall-algebra (these are the stack functions defined by Joyce \cite{a56}). These combinatorial computations become immediately complicated and hard to carry out, when studying higher rank stable pairs with rank$>2$. The main purpose of this article is to introduce an independent approach to computation of rank 2 stable pair invariants, without applying the wallcrossing formula and rather by stratifying their corresponding moduli space and directly computing the weighted Euler characteristics of the strata. This approach may similarly be used to avoid complex combinatorial wallcrossing calculations in rank$>2$ cases. 
\end{abstract}
%\tableofcontents
%\mainmatter
\section{Introduction}
The Donaldson-Thomas theory (DT in short) of a Calabi--Yau threefold $X$ is defined in \cite{a24} and \cite{a20} via integration against the virtual fundamental class of the moduli space of ideal sheaves. In \cite{a18} and \cite{a17} Pandharipande and Thomas introduced objects given by pairs $(F,s)$ where $F$ is a pure sheaf with one dimensional support together with a fixed Hilbert polynomial and $s\in H^{0}(X,F)$ is given as a section of $F$. The authors computed the invariants of stable pairs, using deformation theory and virtual fundamental classes.

Following their work, Joyce and Song defined a similar notion of a (\textit{twisted}) stable pair, given by a sheaf $F$ and section map $s:\mathcal{O}(-n)\rightarrow F$ where $n\gg 0$ was chosen to be a sufficiently large integer so that the cohomology vanishing condition $H^{1}(F(n))=0$ is satisfied. These stable pairs were equipped with a stability condition rather different than the one used in \cite{a17}:

\begin{defn}\label{newstabl} 
(Joyce-Song pair stability) Given a coherent sheaf $F$ let $p_{F}$ denote the reduced Hilbert polynomial of $F$ with respect to the ample line bundle $\mathcal{O}_{X}(1)$. A pair $\phi:\mathcal{O}(-n)\rightarrow F$ is called stable if the following conditions are satisfied:
\begin{enumerate}
\item $p_{F'}\leq p_{F}$ for all proper subsheaves $F'$ of $F$ such that $F'\neq0$.
\item If $\phi$ factors through $F'$ ($F'$ a proper subsheaf of $F$), then $p_{F'}<p_{F}$.
\end{enumerate}
\end{defn}
In this article we refer to this stability as $\Hat{\tau}$-stability. For more on Joyce-Song stability look at \cite[Definition 12.2]{a30}.

One advantage in defining $\hat{\tau}$-stability, is that it enables one to compute the ``\textit{Generalized Donaldson-Thomas invariants}" with respect to the invariants of $\hat{\tau}$-stable pairs. The generalized Donaldson-Thomas invariants could not be calculated using the machinery developed by Thomas in \cite{a20}, since they were given by invariants of semistable sheaves (not just the stable ones!). After work of Joyce and Song the interesting question was to whether one is able to study and compute the invariants of objects composed of a sheaf $F$ and multiple sections given by the morphism $s_{1}\cdots s_{r}:\mathcal{O}^{\oplus r}(-n)\rightarrow F$ for $r>1$ (we will later denote these by rank $r$ stable pairs for short). 
In \cite{a39} the author introduced the notion of highly frozen triples (same as rank $r$ stable pairs), and used the virtual localization technique introduced by Graber-Pandharipande \cite{a25}, to compute their invariants over local Calabi-Yau threefolds (such as local $\mathbb{P}^1$). The objects studied in \cite{a39} (because of the stability condition chosen) were reminiscent of the higher rank analog of (a twisted version of) the Pandharipande-Thomas (PT in short) stable pairs \cite{a17}. In \cite{a41} the author studied the same higher rank objects, but equipped with $\hat{\tau}$-stability condition, and computed their invariants using the wallcrossing technique.

 In this article we would like to introduce a direct method of calculation of such invariants, which involves first stratifying the moduli space of higher rank semistable pairs into disjoint components, where each stratum contains the stable rank 1 pairs and then, computing the weighted Euler characteristic of the moduli space of higher rank pairs with respect to the Euler characteristics of the rank 1 strata. In doing so, we need to first define an auxiliary category $\mathcal{B}_{p}$ (which was originally introduced by Joyce-Song \cite[Section 13.1]{a30}). The objects in $\mathcal{B}_{p}$ are defined similar to the higher rank Joyce-Song pairs and they are classified based on their numerical class $(\beta,r)$. Here, $\beta$ denotes the Chern character of $F$ and $r$ denotes the number of sections of $F$ being considered in the construction. The definition of the category $\mathcal{B}_{p}$ allows one to define ``\textit{weak}" stability conditions  on $\mathcal{B}_{p}$ (look at Definition \ref{weak}).

As we have shown in \cite[Theorem 5.1]{a41}, the moduli stack of weak semistable objects (we denote this by $\tilde{\tau}$-semistable)  in $\mathcal{B}_{p}$ is closely related to the parameterizing moduli stack of higher rank $\hat{\tau}$-semistable pairs, which enables us to obtain the following identity:

\begin{equation}\label{N_B}\textbf{N}^{\beta,r}_{stp}(\hat{\tau})=(-1)^{r^2}\textbf{B}_{p}^{ss}(X,\beta, r, \tilde{\tau}),
\end{equation}
The left hand side of Equation \eqref{N_B} stands for invariants of $\hat{\tau}$-semistable pairs and the right hand side stands for invariants of $\tilde{\tau}$-semistable objects in $\mathcal{B}_p$ which are, roughly speaking, defined as the weighted Euler characteristic of their corresponding moduli stack. Therefore, using this identity, we aim at calculating the right hand side of Equation \eqref{N_B}, using the stratification method mentioned above.

We show in this article that the result of our calculation agrees with the results obtained in \cite{a41}. In particular, we restrict our computations to the rank 2 pairs ($r=2$), and very explicitly calculate their invariants in some examples. As we will see below, even though the computation of such invariants requires a detailed study of the strata involved in the moduli space, the advantage of the strategy used in here is that; it is much more geometric and it avoids complicated combinatorics involved in the method of wallcrossing.  Moreover,  we suspect that the methods introduced in this article may be used to prove the integrality conjectures for the partition functions of the higher rank Joyce-Song invariants in special cases. Toda in \cite{a37} has used a similar stratification technique and  provided an evidence of such integrality property, proposed by Kontsevich-Soibelman \cite[Conjecture 6]{a42}. 
\section{Acknowledgment}\label{wallcross}
I would like to thank Sheldon Katz for suggesting me the computation in Section \ref{compute-example}. Thanks toYukinobu Toda for many valuable discussions and for explaining to me his paper on rank$=2$ Donaldson-Thomas invariants. Thanks to Arend Bayer for kindly answering some of my questions related to this work in spring 2011. I am grateful to Richard Thomas for his help and support and for providing me the opportunity for being a member at Isaac Newton Institute for Mathematical sciences during 2010- 2011. I also thank Newton Institute for hospitality. I acknowledge partial support from NSF grants DMS 0244412, DMS 0555678 and DMS 08-38434 EMSW21-MCTP (R.E.G.S) during the time that the first drafts of project were being completed. I would also like to thank Kavli IPMU and MIT for their kind and wonderful hospitality. My work at IPMU was supported by the World Premier International
Research Center Initiative (WPI Initiative), MEXT, Japan.

\section*{The auxiliary category $\mathcal{B}_{p}$}\label{wallcross}
\begin{defn}\label{Ap} 
Let $X$ be a nonsingular projective Calabi--Yau threefold equipped with an ample line bundle $\mathcal{O}_{X}(1)$. Let $\tau$ denote the Gieseker stability condition on the abelian category of coherent sheaves on $X$. Define $\mathcal{A}_{p}$ to be the sub-category of coherent sheaves whose objects are zero sheaves and non-zero $\tau$-semistable sheaves with fixed reduced Hilbert polynomial $p$ \footnote {Look at \cite[Definition 13.1]{a30} for more detail}. 
\end{defn}
\begin{defn}\label{Bp}
Fix an integer $n$. Now define the category $\mathcal{B}_{p}$ to be the category whose objects are triples $(F,V,\phi)$, where $F\in Obj(\mathcal{A}_{p})$, $V$ is a finite \text{\text{dim}}ensional $\mathbb{C}$-vector space, and $\phi: V\rightarrow \operatorname{Hom}(\mathcal{O}_{X}(-n),F)$ is a $\mathbb{C}$-linear map. Given $(F,V,\phi)$ and $(F',V',\phi')$ in $\mathcal{B}_{p}$ define morphisms $(F,V,\phi)\rightarrow (F',V',\phi')$ in $\mathcal{B}_{p}$ to be pairs of morphisms $(f,g)$ where $f:F\rightarrow F'$ is a morphism in $\mathcal{A}_{p}$ and $g:V\rightarrow V'$ is a $\mathbb{C}$-linear map, such that the following diagram commutes:
\begin{equation*}
\begin{tikzpicture}
back line/.style={densely dotted}, 
cross line/.style={preaction={draw=white, -, 
line width=6pt}}] 
\matrix (m) [matrix of math nodes, 
row sep=2em, column sep=3.25em, 
text height=1.5ex, 
text depth=0.25ex]{ 
V&\operatorname{Hom}(\mathcal{O}_{X}(-n),F)\\
V'&\operatorname{Hom}(\mathcal{O}_{X}(-n),F')\\};
\path[->]
(m-1-1) edge node [above] {$\phi$} (m-1-2)
(m-1-1) edge node [left] {$g$} (m-2-1)
(m-1-2) edge node [right] {$f$}(m-2-2)
(m-2-1) edge node [above] {$\phi'$} (m-2-2);
\end{tikzpicture}
\end{equation*} 
\end{defn}
Our definition of the category $\mathcal{B}_{p}$ is compatible with that of \cite[Definition 13.1]{a30}. \\

Now we define the numerical class of objects in $\mathcal{B}_{p}$ based on \cite[Section 3.1]{a30}. 
\begin{defn}
Define the Grothendieck group $K(\mathcal{B}_{p})=K(\mathcal{A}_{p})\oplus \mathbb{Z}$ where $K(\mathcal{A}_{p})$ is given by the image of $K_{0}(\mathcal{A}_{p})$ in $K(\operatorname{Coh}(X)):=K^{num}(\operatorname{Coh}(X))$. Let  $\mathcal{C}(\mathcal{A}_{p})$ denote the positive cone of $\mathcal{A}_{p}$ defined as$$\mathcal{C}(\mathcal{A}_{p})=\{E\in K^{num}(\mathcal{A}_{p}):0\neq E\in \mathcal{A}_{p}\}.$$Now given $(F,V,\phi)\in \mathcal{B}_{p}$, we write $[(F,V,\phi)]=([F],\text{\text{dim}}(V))$ and define the positive cone of $\mathcal{B}_{p}$ by:
\begin{center}
$\mathcal{C}(\mathcal{B}_{p})=\{(\beta,d)| \beta\in \mathcal{C}(\mathcal{A}_{p})$ and $d\geq 0$ or $\beta=0$ and $d>0\},$
\end{center}
\end{defn}
We state the following results by Joyce and Song without proof:
\begin{lemma}
\cite[Lemma 13.2]{a30}. The category $\mathcal{B}_{p}$ is abelian and $\mathcal{B}_{p}$ satisfies the condition that if $[F]=0\in K(\mathcal{A}_{p})$ then $F\cong 0$. Moreover, $\mathcal{B}_{p}$ is noetherian and artinian and the moduli stacks $\mathfrak{M}^{(\beta,d)}_{\mathcal{B}_{p}}$ are of finite type for all $(\beta,d)\in C(\mathcal{B}_{p})$. 
\end{lemma}
\begin{remark}
The category $\mathcal{A}_{p}$ embeds as a full and faithful sub-category in $\mathcal{B}_{p}$ by $F\rightarrow (F,0,0)$. Moreover, it is shown in \cite[Equation (13.3)]{a30} that every object $(F,V,\phi)$ sits in a short exact sequence.
\begin{equation}\label{coreexact}
0\rightarrow (F,0,0)\rightarrow (F,V,\phi)\rightarrow (0,V,0)\rightarrow 0
\end{equation}
\end{remark}
Next we recall the definition of ``\textit{weak}" (semi)stability for a general abelian category $\mathcal{A}$. 
\begin{defn}\label{weakjoyce}
\footnote{For more detail on Definition \ref{weakjoyce} look at \cite[Definition 3.5]{a30}.} 
Let $\mathcal{A}$ be an abelian category. Let $K(\mathcal{A})$ be the quotient of $K_{0}(\mathcal{A})$ by some fixed group. Let $C(\mathcal{A})$ be the positive cone of $\mathcal{A}$. Suppose $(T,\leq)$ is a totally ordered set and $\tau: C(\mathcal{A})\rightarrow T$ a map. We call $(\tau,T,\leq)$ a stability condition on $\mathcal{A}$ if whenever $\alpha,\beta,\gamma\in C(\mathcal{A})$ with $\beta=\alpha+\gamma$ then either $$\tau(\alpha)<\tau(\beta)<\tau(\gamma)$$ or $$\tau(\alpha)>\tau(\beta)>\tau(\gamma)$$ or $\tau(\alpha)=\tau(\beta)=\tau(\gamma)$. We call $(\tau,T,\leq)$ a weak stability condition on $\mathcal{A}$ if whenever $\alpha,\beta,\gamma\in C(\mathcal{A})$ with $\beta=\alpha+\gamma$ then either $\tau(\alpha)\leq\tau(\beta)\leq\tau(\gamma)$ or $\tau(\alpha)\geq\tau(\beta)\geq\tau(\gamma)$. For such $(\tau,T,\leq)$, we say that a nonzero object $E$ in $\mathcal{A}$ is
\begin{enumerate}
\item $\tau$-semistable if $\forall S\subset E$ where $S\ncong 0$, we have $\tau([S])\leq \tau([E/S])$
\item $\tau$-stable if $\forall S\subset E$ where $S\ncong 0$, we have $\tau([S])< \tau([E/S])$
\item $\tau$-unstable if it is not $\tau$-semistable.
\end{enumerate}
\end{defn}
Now we apply the definition of weak stability conditions to the category $\mathcal{B}_{p}$: 
\begin{defn}\label{weak}
Define the weak stability condition $(\tilde{\tau},\tilde{T}, \leq)$ on $\mathcal{B}_{p}$  by $\tilde{T}=\{0, 1\}$ with the natural order $0<1$, and $\tilde{\tau}(\beta,d)=0$ if $d=0$ and $\tilde{\tau}(\beta,d)=1$ if $d>0$.
\end{defn}
Definition \ref{weak} is compatible with that of \cite[Definition. 13.5]{a30}.
\section*{Moduli stack of objects in $\mathcal{B}_{p}$}

Before constructing the moduli stack of objects in $\mathcal{B}_{p}$, we would like to provide a different description of objects in $\mathcal{B}_{p}$, as complexes in the derived category. This makes it easier to understand the strategy to construct their moduli spaces; By \cite[Lemma 13.2]{a30} there exists a natural embedding functor $\mathfrak{F}:\mathcal{B}_{p}\rightarrow D(X)$ which takes $(F,V,\phi_{V})\in \mathcal{B}_{p}$ to an object in the derived category given by $\cdots\rightarrow 0\rightarrow V\otimes \mathcal{O}_{X}(-n)\rightarrow F\rightarrow 0\rightarrow \cdots$ where $V\otimes \mathcal{O}_{X}(-n)$ and $F$ sit in degree $-1$ and $0$. Assume that $\operatorname{dim}(V)=r$. In that case $V\otimes \mathcal{O}_{X}(-n)\cong \mathcal{O}_{X}(-n)^{\oplus r}$. Therefore, one may view an object  $(F,V,\phi_{V})\in \mathcal{B}_{p}$ as an object in an abelian subcategory of the derived category, given as $[\mathcal{O}_{X}(-n)^{\oplus r}\rightarrow F]$. Now let us use this prespective and define the notion of flat families for objects in $\mathcal{B}_{p}$.
\begin{defn}\label{Bp-sflat}
Fix a parameterizing scheme of finite type $S$. Let $\pi_{X}:X\times S\rightarrow X$ and $\pi_{S}:X\times S\rightarrow S$ denote the natural projections. Use the natural embedding functor $\mathfrak{F}:\mathcal{B}_{p}\rightarrow D(X)$ in \cite[Lemma 13.2]{a30}. Define the $S$-flat family of objects in $\mathcal{B}_{p}$ of type $(\beta,r)$ as a complex $$\pi_{S}^{*}M\otimes \pi_{X}^{*}\mathcal{O}_{X}(-n)\xrightarrow{\psi_{S}} \mathcal{F}$$ sitting in degree $-1$ and $0$, such that $\mathcal{F}$ is given by an $S$-flat family of semistable sheaves with fixed reduced Hilbert polynomial $p$ with $\operatorname{ch}(F)=\beta$ and $M$ is a vector bundle of rank $r$ over $S$. A morphism between two such $S$-flat families is given by a morphism between the complexes $\pi_{S}^{*}M\otimes \pi_{X}^{*}\mathcal{O}_{X}(-n)\xrightarrow{\psi_{S}} \mathcal{F}$ and $\pi_{S}^{*}M'\otimes \pi_{X}^{*}\mathcal{O}_{X}(-n)\xrightarrow{\psi'_{S}} \mathcal{F}'$:
\begin{center}
\begin{tikzpicture}
back line/.style={densely dotted}, 
cross line/.style={preaction={draw=white, -, 
line width=6pt}}] 
\matrix (m) [matrix of math nodes, 
row sep=2em, column sep=3.5em, 
text height=1.5ex, 
text depth=0.25ex]{ 
\pi_{S}^{*}M\otimes \pi_{X}^{*}\mathcal{O}_{X}(-n)&\mathcal{F}\\
\pi_{S}^{*}M'\otimes \pi_{X}^{*}\mathcal{O}_{X}(-n)&\mathcal{F}'.\\};
\path[->]
(m-1-1) edge node [above] {$\psi_{S}$} (m-1-2)
(m-1-1) edge (m-2-1)
(m-1-2) edge (m-2-2)
(m-2-1) edge node [above] {$\psi'_{S}$} (m-2-2);
\end{tikzpicture}
\end{center}
Moreover an isomorphism between two such $S$-flat families in $\mathcal{B}_{p}$ is given by an isomorphism between the associated complexes $\pi_{S}^{*}M\otimes \pi_{X}^{*}\mathcal{O}_{X}(-n)\xrightarrow{\psi_{S}} \mathcal{F}$ and $\pi_{S}^{*}M'\otimes \pi_{X}^{*}\mathcal{O}_{X}(-n)\xrightarrow{\psi'_{S}} \mathcal{F}'$:
\begin{center}
\begin{tikzpicture}
back line/.style={densely dotted}, 
cross line/.style={preaction={draw=white, -, 
line width=6pt}}] 
\matrix (m) [matrix of math nodes, 
row sep=2em, column sep=3.5em, 
text height=1.5ex, 
text depth=0.25ex]{ 
\pi_{S}^{*}M\otimes \pi_{X}^{*}\mathcal{O}_{X}(-n)&\mathcal{F}\\
\pi_{S}^{*}M'\otimes \pi_{X}^{*}\mathcal{O}_{X}(-n)&\mathcal{F}'.\\};
\path[->]
(m-1-1) edge node [above] {$\psi_{S}$} (m-1-2)
(m-1-1) edge node [left] {$\cong$} (m-2-1)
(m-1-2) edge node [right] {$\cong$} (m-2-2)
(m-2-1) edge node [above] {$\psi'_{S}$} (m-2-2);
\end{tikzpicture}
\end{center}
From now on, by objects in $\mathcal{B}_{p}$ we mean the objects which lie in the image of the natural embedding functor $\mathfrak{F}:\mathcal{B}_{p}\rightarrow D(X)$. Moreover, by the $S$-flat family of objects in $\mathcal{B}_{p}$, their morphisms (or isomorphisms) we mean their corresponding definitions as stated in Definition \ref{Bp-sflat}. 
\end{defn}
Now we define the \textit{rigidified} objects in $\mathcal{B}_{p}$. These are \textbf{only} going to provide the means for construction of moduli stack of objects in $\mathcal{B}_{p}$ as a quotient stack.

\subsection{Rigidified objects and their realization in the derived category}
As stated in Definition \ref{Bp-sflat}, the objects in $\mathcal{B}_{p}$ are defined such that the sheaf sitting in degree $-1$ is given by a trivial vector bundle of rank $r$ isomorphic to $\mathcal{O}^{\oplus r}_{X}(-n)$. However we have not fixed any choice of such trivialization. Below we will define the closely related objects, which we denote by rigidified objects in $\mathcal{B}_{p}$, by fixing a choice of the trivialization of $\mathcal{O}^{\oplus r}_{X}(-n)$. These objects are essential for our construction, as their moduli stack forms a $GL_{r}(\mathbb{C})$-torsor over the (to be defined)  moduli stack of objects in $\mathcal{B}_{p}$. Therefore our plan is to essentially  construct the moduli stack of objects in $\mathcal{B}_{p}$ as the stacky quotient of the moduli stack of rigidified objects in $\mathcal{B}_{p}$, where the group we take the quotient with is $GL_{r}(\mathbb{C})$.

\begin{defn}\label{rigid-Bp}
Fix a positive integer $r$ and define the subcategory $\mathcal{B}_{p}^{\textbf{R}}\subset \mathcal{B}_{p}$ to be the category of ``rigidified" objects in $\mathcal{B}_{p}$ of rank $r$ whose objects are defined by tuples $(F,\mathbb{C}^{\oplus r},\rho)$ where $F$ is a coherent sheaf with reduced Hilbert polynomial $p$ and $\operatorname{ch}(F)=\beta$ and $\rho:\mathbb{C}^{r}\rightarrow \operatorname{Hom}(\mathcal{O}_{X}(-n),F)$. Given two rigidified objects of fixed type $(\beta,r)$ as $(F,\mathbb{C}^{\oplus r},\rho)$ and $(F',\mathbb{C}^{\oplus r},\rho')$ in $\mathcal{B}_{p}^{\textbf{R}}$, define morphisms $(F,\mathbb{C}^{\oplus r},\rho)\rightarrow (F',\mathbb{C}^{\oplus r},\rho')$ to be given by a morphism $f:F\rightarrow F'$ in $\mathcal{A}_{p}$ such that the following diagram commutes:
\begin{equation*}
\begin{tikzpicture}
back line/.style={densely dotted}, 
cross line/.style={preaction={draw=white, -, 
line width=6pt}}] 
\matrix (m) [matrix of math nodes, 
row sep=2em, column sep=3.25em, 
text height=1.5ex, 
text depth=0.25ex]{ 
\mathbb{C}^{\oplus r}&\operatorname{Hom}(\mathcal{O}_{X}(-n),F)\\
\mathbb{C}^{\oplus r}&\operatorname{Hom}(\mathcal{O}_{X}(-n),F').\\};
\path[->]
(m-1-1) edge node [above] {$\rho$} (m-1-2) 
(m-2-1) edge node [above] {$\rho'$} (m-2-2)
(m-1-2) edge node [right] {$f$}(m-2-2)
(m-1-1) edge node [left] {$\operatorname{id}$} (m-2-1);
\end{tikzpicture}
\end{equation*} 
\end{defn}
\begin{remark}\label{embed-func-BpR}
Similar to before, there exists a natural embedding functor $\mathfrak{F}^{\textbf{R}}:\mathcal{B}^{\textbf{R}}_{p}\rightarrow D(X)$ which takes $(F,\mathbb{C}^{\oplus r},\rho)\in \mathcal{B}^{\textbf{R}}_{p}$ to an object in the derived category given by $\cdots\rightarrow 0\rightarrow \mathbb{C}^{\oplus r}\otimes \mathcal{O}_{X}(-n)\rightarrow F\rightarrow 0\rightarrow \cdots$ where $\mathbb{C}^{\oplus r}\otimes \mathcal{O}_{X}(-n)$ sits in degree $-1$ and $F$ sits in degree $0$. One may view an object in $\mathcal{B}^{\textbf{R}}_{p}$ as a complex $\phi:\mathcal{O}^{\oplus r}_{X}(-n)\rightarrow F$ such that the choice of trivialization of $\mathcal{O}^{\oplus r}_{X}(-n)$ is fixed.
\end{remark}
\begin{defn}\label{Bp-R-sflat}
Fix a parametrizing scheme of finite type $S$. Use the natural embedding functor $\mathfrak{F}^{\textbf{R}}:\mathcal{B}^{\textbf{R}}_{p}\rightarrow D(X)$ in Remark \ref{embed-func-BpR}. An $S$-flat family of objects of type $(\beta,r)$ in $\mathcal{B}^{\textbf{R}}_{p}$ is given by a complex $$\pi_{S}^{*}\mathcal{O}_{S}^{\oplus r}\otimes \pi_{X}^{*}\mathcal{O}_{X}(-n)\xrightarrow{\psi_{S}} \mathcal{F}$$ sitting in degree $-1$ and $0$ such that $\mathcal{F}$ is given by an $S$-flat family of semistable sheaves with fixed reduced Hilbert polynomial $p$ with $\operatorname{ch}(\mathcal{F}_{s})=\beta$ for all $s\in S$. A morphism between two such $S$-flat families in $\mathcal{B}^{\textbf{R}}_{p}$ is given by a morphism between the complexes $\pi_{S}^{*}\mathcal{O}_{S}^{\oplus r}\otimes \pi_{X}^{*}\mathcal{O}_{X}(-n)\xrightarrow{\psi_{S}} \mathcal{F}$ and $\pi_{S}^{*}\mathcal{O}_{S}^{\oplus r}\otimes \pi_{X}^{*}\mathcal{O}_{X}(-n)\xrightarrow{\psi'_{S}} \mathcal{F}'$:
\begin{center}
\begin{tikzpicture}
back line/.style={densely dotted}, 
cross line/.style={preaction={draw=white, -, 
line width=6pt}}] 
\matrix (m) [matrix of math nodes, 
row sep=2em, column sep=3.5em, 
text height=1.5ex, 
text depth=0.25ex]{ 
\pi_{S}^{*}\mathcal{O}_{S}^{\oplus r}\otimes \pi_{X}^{*}\mathcal{O}_{X}(-n)&\mathcal{F}\\
\pi_{S}^{*}\mathcal{O}_{S}^{\oplus r}\otimes \pi_{X}^{*}\mathcal{O}_{X}(-n)&\mathcal{F}'.\\};
\path[->]
(m-1-1) edge node [above] {$\psi_{S}$} (m-1-2)
(m-1-1) edge node [left] {$\operatorname{id}_{\mathcal{O}_{X\times S}}$}(m-2-1)
(m-1-2) edge (m-2-2)
(m-2-1) edge node [above] {$\psi'_{S}$} (m-2-2);
\end{tikzpicture}
\end{center}
Moreover an isomorphism between two such $S$-flat families in $\mathcal{B}^{\textbf{R}}_{p}$ is given by an isomorphism between the associated complexes $\pi_{S}^{*}\mathcal{O}_{S}^{\oplus r}\otimes \pi_{X}^{*}\mathcal{O}_{X}(-n)\xrightarrow{\psi_{S}} \mathcal{F}$ and $\pi_{S}^{*}\mathcal{O}_{S}^{\oplus r}\otimes \pi_{X}^{*}\mathcal{O}_{X}(-n)\xrightarrow{\psi'_{S}} \mathcal{F}'$:
\begin{center}
\begin{tikzpicture}
back line/.style={densely dotted}, 
cross line/.style={preaction={draw=white, -, 
line width=6pt}}] 
\matrix (m) [matrix of math nodes,  
row sep=2em, column sep=3.5em, 
text height=1.5ex, 
text depth=0.25ex]{ 
\pi_{S}^{*}\mathcal{O}_{S}^{\oplus r}\otimes \pi_{X}^{*}\mathcal{O}_{X}(-n)&\mathcal{F}\\
\pi_{S}^{*}\mathcal{O}_{S}^{\oplus r}\otimes \pi_{X}^{*}\mathcal{O}_{X}(-n)&\mathcal{F}'.\\};
\path[->]
(m-1-1) edge node [above] {$\psi_{S}$} (m-1-2)
(m-1-1) edge node [left] {$\operatorname{id}_{\mathcal{O}_{X\times S}}$} (m-2-1)
(m-1-2) edge node [right] {$\cong$} (m-2-2)
(m-2-1) edge node [above] {$\psi'_{S}$} (m-2-2);
\end{tikzpicture}
\end{center}
Similar to the way that we treated objects in $\mathcal{B}_{p}$, from now on by objects in $\mathcal{B}^{\textbf{R}}_{p}$ we mean the objects which lie in the image of the natural embedding functor $\mathfrak{F}^{\textbf{R}}:\mathcal{B}^{\textbf{R}}_{p}\rightarrow D(X)$ in Remark \ref{embed-func-BpR}. Moreover by the $S$-flat family of objects in $\mathcal{B}^{\textbf{R}}_{p}$, their morphisms (or isomorphisms) we mean the corresponding definitions as stated in Definition \ref{Bp-R-sflat}. 
\end{defn}
\textbf{Notation}: 
In what follows we define $\mathfrak{M}^{(\beta,r)}_{\mathcal{B}_{p}}$ ($\mathfrak{M}^{(\beta,r)}_{\mathcal{B^{\textbf{R}}}_{p}}$ respectively) to be the moduli functors from $Sch/\mathbb{C}\rightarrow \operatorname{Groupoids}$ which send a $\mathbb{C}$-scheme $S$ to the groupoid of $S$-flat family of objects of type $(\beta,r)$ in $\mathcal{B}_{p}$ ($\mathcal{B}^{\textbf{R}}_{p}$ respectively). \\

We will show that these moduli functors (as groupoid valued functors) are equivalent to algebraic quotient stacks. We will also show that the moduli stack $\mathfrak{M}^{(\beta,r)}_{\mathcal{B}_{p}}$ is given by a stacky quotient of $\mathfrak{M}^{(\beta,r)}_{\mathcal{B}^{\textbf{R}}_{p}}$ by $GL_{r}(\mathbb{C})$. 

\subsection{The underlying parameter scheme}\label{scheme-beta}

According to Definition \ref{Bp}, an object in the category $\mathcal{A}_{p}$ consists of semistable sheaves with fixed reduced Hilbert polynomial $p$. Note that having fixed a polynomial (in variable $t$) $p(t)$ as the reduced Hilbert polynomial of $F$ means that the Hilbert polynomial of $F$ can yet be chosen as $P_{F}(t)= \frac{k}{d'!} \cdot p(t)$ for different values of $k$  where $d'$ is the dimension of $F$. However here we make an assumption that there are only finitely many possible values $k = 0,1,\cdots, N$ for which our computation makes sense. We explain the motivation behind this assumption further below;

Our analysis inherits this finiteness property directly from applying \cite[Proposition 13.7]{a30} where the authors show that there are only finitely many nontrivial contributions to their wallcrossing computation which are induced by objects, whose underlying sheaves could only have finitely many fixed Hilbert polynomials. In other words, according to  \cite[Proposition 13.7]{a30}, it suffices to consider Hilbert polynomials $P_{F}(t)= \frac{k}{d'!} \cdot p(t)$ induced by $p$ and only finitely many values of $k=1,2,\cdots, N$ for some $N>0$. Similary for us there would only be finitely many such $k$'s for which \eqref{N_B} holds true, which justifies the reason behind our assumption.

On the other hand, as discussed in  \cite[Theorem 3.37]{a8}, the family of Gieseker-semistable sheaves $F$ on $X$ such that $F$ has a fixed Hilbert polynomial is bounded. Therefore, the family of coherent sheaves with finitely many fixed Hilbert polynomials is also bounded.  We will use this boundedness property in our construction of parameterizing moduli stacks; 

Now fix the Hilbert polynomial $P_{F}(t)=P$ as above, and use the fact that given a bounded family $\mathcal{F}$ of coherent sheaves with fixed Hilbert polynomial $P$ (here $F$ denotes each member of the family $\mathcal{F}$), there exists an upper bound for their Castelnuvo-Mumford regularity, given by the integer $m$ such that for each member of the family, $F$, the twisted sheaf $F(m')$ is globally generated for all $m'\geq m$. Fix such $m'$ and let $V$ be the complex vector space of \text{\text{dim}}ension $d=P(m')$ given by $V=\operatorname{H}^{0}(F\otimes \mathcal{O}_{X}(m'))$. Twisting the sheaf $F$ by the fixed large enough integer $m'$ would ensure one to get a surjective morphism of coherent sheaves $V\otimes \mathcal{O}_{X}(-m')\rightarrow F$. One can then construct a scheme parametrizing the flat quotients of $V\otimes \mathcal{O}_{X}(-m')$ with fixed given Hilbert polynomial $P$. This by usual arguments provides us with Grothendieck's Quot-scheme. Here to shorten the notation we use $\mathcal{Q}$ to denote $\operatorname{Quot}_{P}(V\otimes \mathcal{O}_{X}(-m'))$. Denote by $\mathcal{Q}^{ss}\subset \mathcal{Q}$ the sublocus of Gieseker-semistable ($\tau$-semistable for short) sheaves $F$ with fixed Hilbert polynomial $P$. 
\begin{defn}\label{scheme-beta-def}
Let $n$ in Definition \ref{Bp-sflat} be given so that $n\gg m'$. Define $\mathcal{P}$ over $\mathcal{Q}^{ss}$ to be a bundle whose fibers parameterize $\operatorname{H}^{0}(F(n))$. The fibers of the bundle $\mathcal{P}^{\oplus r}$ over each point $[F]=\{p\}$, where $p\in \mathcal{Q}^{ss}$, parameterize $\operatorname{H}^{0}(F(n))^{\oplus r}$.
In other words the fibers of $\mathcal{P}^{\oplus r}$ parameterize the maps $\mathcal{O}_{X}^{\oplus r}(-n)\rightarrow F$ (which define the complexes representing the objects in $\mathcal{B}^{\textbf{R}}_{p}$). 
\end{defn} 
There exists a right action of $GL(V)$ (where $V$ is as above) on  the Quot scheme $\mathcal{Q}$ which induces an action on $\mathcal{Q}^{ss}$, after restriction to the open subscheme of $\tau$-semistable sheaves. It is trivially seen that the action of $GL(V)$ on $\mathcal{Q}^{ss}$ induces a right action on $\mathcal{P}^{\oplus r}$. However note that, since we have fixed the trivialization of $\mathcal{O}^{\oplus r}_{X}(-n)$ for the objects in $\mathcal{B}^{\textbf{R}}_{p}$, there exists also an extra action of $GL_{r}(\mathbb{C})$ on $\mathcal{P}^{\oplus r}$ which is described as follows; let $[\mathcal{O}_{X}^{\oplus r}(-n)\xrightarrow{\phi} F]$ be given as a point in $\mathcal{P}^{\oplus r}$. Let $\psi\in GL_{r}(\mathbb{C})$ be the map given by $\psi: \mathcal{O}_{X}(-n)^{\oplus r}\rightarrow \mathcal{O}_{X}(-n)^{\oplus r}$. The action of $GL_{r}(\mathbb{C})$ on $\mathcal{P}^{\oplus r}$ is defined via precomposing the sections of $F$ with $\psi$ as shown in the diagram below:
\begin{equation}\label{action}
\begin{tikzpicture}
back line/.style={densely dotted}, 
cross line/.style={preaction={draw=white, -, 
line width=6pt}}] 
\matrix (m) [matrix of math nodes, 
row sep=2em, column sep=3.25em, 
text height=1.5ex, 
text depth=0.25ex]{ 
\mathcal{O}^{\oplus r}_{X}(-n)&\\
\mathcal{O}^{\oplus r}_{X}(-n)&F.\\};
\path[->]
(m-1-1) edge node [left] {$\psi$} (m-2-1)

(m-2-1)  edge node [above]{$\phi$}(m-2-2);
\end{tikzpicture}
\end{equation}
Note that, by the Grothendieck-Riemann-Roch theorem, fixing a polarization over $X$ and the chern character of sheaves as $\text{ch}(F)=\beta$, induces fixed Hilbert polynomial for such $F$. Therefore this is the reason why we index our parameterizing moduli stacks by $(\beta,r)$ instead of $(P,r)$. 
%\backmatter
\subsection{The Artin stacks $\mathfrak{M}^{(\beta,r)}_{\mathcal{B}_{p}}$ and $\mathfrak{M}^{(\beta,r)}_{\mathcal{B}^{\textbf{R}}_{p}}$.}\label{sec41}
By definitions \ref{Bp-sflat} and \ref{Bp-R-sflat}  the construction of moduli stack of objects  in $\mathcal{B}_{p}$ and $\mathcal{B}^{\textbf{R}}_{p}$ is done similar to \cite[Section 5]{a39}:

\begin{theorem}\label{theorem81}
Let $\mathcal{P}^{\oplus r}$ be as in Definition \ref{scheme-beta-def}. Then the following statements hold true:
\begin{enumerate}
\item Let $\left[\frac{\mathcal{P}^{\oplus r}}{GL(V)}\right]$ be the stack theoretic quotient of $\mathcal{P}^{\oplus r}$ by $GL(V)$. Then, there exists an isomorphism of stacks $$\mathfrak{M}^{(\beta,r)}_{\mathcal{B^{\textbf{R}}}_{p}}\cong \left[\frac{\mathcal{P}^{\oplus r}}{GL(V)}\right].$$In particular, $\mathfrak{M}^{(\beta,r)}_{\mathcal{B^{\textbf{R}}}_{p}}$ is an Artin stack. 

\item The moduli stack, $\mathfrak{M}^{(\beta,r)}_{\mathcal{B}^{\textbf{R}}_{p}}$, is a $GL_{r}(\mathbb{C})$-torsor over $\mathfrak{M}^{(\beta,r)}_{\mathcal{B}_{p}}$. In particular there exists an isomorphism of stacks $$\mathfrak{M}^{(\beta,r)}_{\mathcal{B}_{p}}\cong \left[\frac{\mathfrak{M}^{(\beta,r)}_{\mathcal{B}^{\textbf{R}}_{p}}}{GL_{r}(\mathbb{C})}\right].$$

\item It is true that locally in the flat topology, $\mathfrak{M}^{(\beta,r)}_{\mathcal{B}_{p}}\cong \mathfrak{M}^{(\beta,r)}_{\mathcal{B}^{\textbf{R}}_{p}}\times \left[\frac{\operatorname{Spec}(\mathbb{C})}{GL_{r}(\mathbb{C})}\right]$. This isomorphism does not hold true globally unless $r=1$.
\end{enumerate}
\end{theorem}
\begin{proof} The proofs of parts (1), (2) are essentially the same as \cite[Proposition 5.5, Corollary 6.4, Theorems 6.2 and Theorem 6.5]{a39}. Now we prove part (3) by showing that there exists a forgetful map $\pi:\mathfrak{M}^{(\beta,r)}_{\mathcal{B}^{\textbf{R}}_{p}}\rightarrow \mathfrak{M}^{(\beta,r)}_{\mathcal{B}_{p}}$ which induces a map from $\mathfrak{M}^{(\beta,r)}_{\mathcal{B}^{\textbf{R}}_{p}}\times \left[\frac{\operatorname{Spec}(\mathbb{C})}{GL_{r}(\mathbb{C})}\right] $ to $\mathfrak{M}^{(\beta,r)}_{\mathcal{B}_{p}}$ and show that this map has an inverse locally but not globally unless $r=1$. First we prove the claim for $r=1$; 

For $r=1$, $GL_{1}(\mathbb{C})=\mathbb{G}_{m}$. For a $\mathbb{C}$-scheme $S$, an $S$-point of $\mathfrak{M}^{(\beta,1)}_{\mathcal{B}^{\textbf{R}}_{p}}\times [\frac{\operatorname{Spec}(\mathbb{C})}{\mathbb{G}_{m}}]$ is identified with the data $(\mathcal{O}_{X\times S}(-n)\rightarrow \mathcal{F}, \mathcal{L}_{S})$ where $\mathcal{L}_{S}$ is a $\mathbb{G}_{m}$ line bundle over $S$. Let $\pi_{S}:X\times S\rightarrow S$ be the natural projection onto the second factor. There exists a map that sends this point to an $S$-point $p\in \mathfrak{M}^{(\beta,1)}_{\mathcal{B}_{p}}$ which is obtained by tensoring with $\mathcal{L}_{S}$, i.e $\mathcal{O}_{X}(-n)\boxtimes \mathcal{L}_{S}\xrightarrow{\phi^{\mathcal{L}}} \mathcal{F}\times \pi_{S}^{*}\mathcal{L}_{S}$. Note that tensoring $\mathcal{O}_{X\times S}(-n)$ with $\pi^{*}_{S}\mathcal{L}_{S}$ does not change the fact that $\mathcal{O}_{X\times S}(-n)|_{s\in S}\cong \mathcal{O}_{X}(-n)\boxtimes \mathcal{L}_{S}|_{s\in S}$ fiber by fiber. Moreover, there exists a section map $s: \mathfrak{M}^{(\beta,1)}_{\mathcal{B}_{p}}\rightarrow \mathfrak{M}^{(\beta,1)}_{\mathcal{B}^{\textbf{R}}_{p}}\times [\frac{\operatorname{Spec}(\mathbb{C})}{\mathbb{G}_{m}}]$; Simply take an $S$-point  $[\mathcal{O}_{X}(-n)\boxtimes \mathcal{L}_{S}\rightarrow \mathcal{F}]\in (\mathfrak{M}^{(\beta,1)}_{\mathcal{B}_{p}})(S)$  and send to an $S$-point in $(\mathfrak{M}^{(\beta,1)}_{\mathcal{B}^{\textbf{R}}_{p}}\times [\frac{\operatorname{Spec}(\mathbb{C})}{\mathbb{G}_{m}}])(S)$ by the map $$[\mathcal{O}_{X}(-n)\boxtimes \mathcal{L}_{S}\rightarrow \mathcal{F}]\mapsto([\mathcal{O}_{X\times S}(-n)\rightarrow \mathcal{F}\otimes \pi_{S}^{*}\mathcal{L}_{S}^{-1}], \mathcal{L}_{S}).$$ Note that since $\mathcal{L}_{S}$ is a line bundle over $S$ then it is invertible and hence a section map is always well defined and $\mathfrak{M}^{(\beta,1)}_{\mathcal{B}_{p}}$ is a $\mathbb{G}_{m}$-gerbe over $\mathfrak{M}^{(\beta,1)}_{\mathcal{B}^{\textbf{R}}_{p}}$. Now let $r>1$. It is left to show that there exists a map from $\mathfrak{M}^{(\beta,r)}_{\mathcal{B}^{\textbf{R}}_{p}}\times \left[\frac{\operatorname{Spec}(\mathbb{C})}{GL_{r}(\mathbb{C})}\right] $ to $\mathfrak{M}^{(\beta,r)}_{\mathcal{B}_{p}}$ and this map does not have an inverse (section map) globally. To proceed further, we state the following definition.
\begin{defn}\label{fiber-stack}
Consider a stack $(\mathfrak{Y}, p_{\mathfrak{Y}}:\mathfrak{Y}\rightarrow Sch\slash \mathbb{C})$.  Given Two morphism of stacks $\pi_{1}:\mathfrak{X}\rightarrow \mathfrak{Y}$ and $\pi_{2}:\mathfrak{X}'\rightarrow \mathfrak{Y}$, the fibered product of $\mathfrak{X}$ and $\mathfrak{X}'$ over $\mathfrak{Y}$ is defined by the category whose objects are defined by triples $(x,x',\alpha)$ where $x\in \mathfrak{X}$ and $x'\in \mathfrak{X}'$ respectively and $\alpha:\pi_{1}(x)\rightarrow \pi_{2}(x')$ is an arrow in $\mathfrak{Y}$ such that $p_{\mathfrak{Y}}(\alpha)=\operatorname{id}$. Moreover the morphisms $(x,x',\alpha)\rightarrow (y,y',\beta)$ are defined by the tuple $(\phi: x\rightarrow y, \psi: x'\rightarrow y')$ such that$$\pi_{1}(\psi)\circ \alpha=\beta\circ \pi_{2}(\phi):P_{E}(x)\rightarrow P_{F}(y').$$
\end{defn}
Now observe that for $r>1$, there exists a forgetful map $\pi:\mathfrak{M}^{(\beta,r)}_{\mathcal{B}^{\textbf{R}}_{p}}\rightarrow \mathfrak{M}^{(\beta,r)}_{\mathcal{B}_{p}}$  which over the $S$-points takes $[\mathcal{O}_{X}(-n)\boxtimes \mathcal{O}_{S}^{\oplus r}\to \mathcal{F}]$ (look at Definition \ref{Bp-R-sflat}) to $[M\boxtimes \mathcal{O}_{X}(-n)\to \mathcal{F}]$ where $M$ is defined in Definition \ref{Bp-sflat}. Moreover, there exists a map $g':\mathfrak{M}^{(\beta,r)}_{\mathcal{B}_{p}}\rightarrow \mathcal{B}GL_{r}(\mathbb{C})$ which sends $[M\boxtimes \mathcal{O}_{X}(-n)\to \mathcal{F}]$ to $M$ by forgetting $\mathcal{F}$. Finally there exists the natural projection $i:\operatorname{Spec}(\mathbb{C})\rightarrow \left[\frac{\operatorname{Spec}(\mathbb{C})}{GL_{r}(\mathbb{C})}\right]=\mathcal{B}GL_{r}(\mathbb{C})$. It follows directly from Definition \ref{fiber-stack} that the diagram:
$$
\begin{tikzpicture}
back line/.style={densely dotted}, 
cross line/.style={preaction={draw=white, -, 
line width=6pt}}] 
\matrix (m) [matrix of math nodes, 
row sep=2em, column sep=3.25em, 
text height=1.5ex, 
text depth=0.25ex]{ 
\mathfrak{M}^{(\beta,r)}_{\mathcal{B}^{\textbf{R}}_{p}}&pt=\operatorname{Spec}(\mathbb{C})\\
\mathfrak{M}^{(\beta,r)}_{\mathcal{B}_{p}}&\mathcal{B}GL_{r}(\mathbb{C})=\left[\frac{\operatorname{Spec}(\mathbb{C})}{GL_{r}(\mathbb{C})}\right]\\};
\path[->]
(m-1-1) edge node [above] {$g$} (m-1-2)
(m-1-1) edge node [left] {$\pi$} (m-2-1)
(m-1-2) edge node [right] {$i$}(m-2-2)
(m-2-1) edge node [above] {$\acute{g}$} (m-2-2);
\end{tikzpicture}
$$ 
is a fibered diagram and $\mathfrak{M}^{(\beta,r)}_{\mathcal{B}^{\textbf{R}}_{p}}=\mathfrak{M}^{(\beta,r)}_{\mathcal{B}_{p}}\times_{\mathcal{B}GL_{r}(\mathbb{C})}pt$.
However, one cannot use the same argument used for case of $r=1$ to conclude that there exists a section map $s:\mathfrak{M}^{(\beta,r)}_{\mathcal{B}_{p}}\rightarrow \mathfrak{M}^{(\beta,r)}_{\mathcal{B}^{\textbf{R}}_{p}}\times \left[\frac{\operatorname{Spec}(\mathbb{C})}{GL_{r}(\mathbb{C})}\right]$, since the $S$-point of $\mathcal{B}GL_{r}(\mathbb{C})$ is a $GL_{r}(\mathbb{C})$ bundle over $S$ and this vector bundle is trivializable locally but not globally. Therefore locally one may think of $\mathfrak{M}^{(\beta,r)}_{\mathcal{B}_{p}}$ as isomorphic to $\mathfrak{M}^{(\beta,r)}_{\mathcal{B}^{\textbf{R}}_{p}}\times [\frac{\operatorname{Spec}(\mathbb{C})}{GL_{r}(\mathbb{C})}]$ but not globally.\end{proof}
\begin{defn}\label{stable-locus}
Define $\mathfrak{M}^{(\beta,r)}_{\mathcal{B}_{p},ss}(\tilde{\tau})$ as the substack of $\mathfrak{M}^{(\beta,r)}_{\mathcal{B}_{p}}$ parameterizing $\tilde{\tau}$-semistable objects in $\mathcal{B}_{p}$. Since $\mathfrak{M}^{(\beta,r)}_{\mathcal{B}_{p}}$ is of finite type by \cite[Lemma 13.2]{a30} , then $\mathfrak{M}^{(\beta,r)}_{\mathcal{B}_{p},ss}(\tilde{\tau})$ is of finite type for all $(\beta,r)\in \mathcal{C}(\mathcal{B}_{p})$. 
\end{defn}

\section{Stack function identities in the Ringel-Hall algebra}
We review here some basic facts about the stack functions in the Ringel-Hall algebras. Let $\mathfrak{M}$ be an Artin $\mathbb{C}$-stack with affine geometric stabilizers.  Consider pairs $(\mathfrak{R}, \rho)$ where $\mathfrak{R}$ is given by a finite type Artin $\mathbb{C}$-stack with affine geometric stabilizers and $\rho:=\mathfrak{R}\rightarrow\mathfrak{M}$ is a 1-morphism. Now define an equivalence relation for such pairs where $(\mathfrak{R}, \rho)$ and $(\mathfrak{R}', \rho')$ are called equivalent if there exists a 1-morphism $\iota: \mathfrak{R}\to \mathfrak{R}'$ such that $\rho'\circ \iota$ and $\rho$ are 2-isomorphic 1-morphisms $\mathfrak{R}\to \mathfrak{M}$. Joyce and Song in \cite[Section 2.2]{a30} define the space of stack functions $\underline{\operatorname{S}\overline{\operatorname{F}}}(\mathfrak{M},\chi,\mathbb{Q})$ as the $\mathbb{Q}$-vector space  generated by the above equivalence classes of pairs $[(\mathfrak{R},\rho)]$ such that the following relations are imposed:

\begin{enumerate}
\item Given a closed substack $(\mathfrak{G},\rho|_{\mathfrak{G}})\subset (\mathfrak{R},\rho)$ we have 
\begin{equation}\label{splitt-motiv}
[(\mathfrak{R},\rho)]=[(\mathfrak{G},\rho|_{\mathfrak{G}})]+[(\mathfrak{R}\backslash\mathfrak{G},\rho|_{\mathfrak{R}\backslash\mathfrak{G}})]
\end{equation} 
\item Let $\mathfrak{R}$ be a $\mathbb{C}$-stack of finite type with affine geometric stabilizers and let $\mathcal{U}$ denote a quasi-projective $\mathbb{C}$-variety and $\pi_{\mathfrak{R}}:\mathfrak{R}\times U\rightarrow \mathfrak{R}$ the natural projection and $\rho:\mathfrak{R}\rightarrow \mathfrak{M}$ a 1-morphism. Then 
\begin{equation}\label{property2}
[(\mathfrak{R}\times \mathcal{U}, \rho \circ \pi_{\mathfrak{R}})]=\chi([\mathcal{U}])[(\mathfrak{R},\rho)].
\end{equation}
\item Assume $\mathfrak{R}\cong [X/G]$ where $X$ is a quasiprojective $\mathbb{C}$-variety and $G$ a very special algebraic $\mathbb{C}$-group acting on $X$ with maximal torus $T^{G}$, then we have
\begin{equation*}
[(\mathfrak{R},\rho)]=\sum_{Q\in \mathcal{Q}(G,T^{G})}F(G,T^{G},Q)[([X/Q],\rho\circ \iota^{Q})],
\end{equation*}
where the rational coefficients $F(G,T^{G},Q)$ have a complicated definition explained in \cite[Section 6.2]{a33}. 
Here $\mathcal{Q}(G,T^{G})$ is the set of closed $\mathbb{C}$-subgroups $Q$ of $T^{G}$ such that $Q=T^{G}\cap C_{G}(Q)$ where $C_{G}(Q)=\{g\in G:sg=gs \,\,\text{for all} \,\,s\in Q\}$ and $\iota^{Q}:[X/Q]\rightarrow\mathfrak{R}\cong [X/G]$ is the natural projection 1-morphism. Similarly, one defines $\overline{\operatorname{SF}}(\mathfrak{M},\chi,\mathbb{Q})$ by restricting the 1-morphisms $\rho$ in part 1, 2, 3 to be representable.
\item There exist the notions of multiplication, pullback, pushforward of stack functions in  $\underline{\operatorname{S}\overline{\operatorname{F}}}(\mathfrak{M},\chi,\mathbb{Q})$ and $\overline{\operatorname{SF}}(\mathfrak{M},\chi,\mathbb{Q})$. For further discussions look at (Joyce and Song) \cite[Definitions 2.6, 2.7 and Theorem 2.9]{a30}. 
\end{enumerate}
Joyce and Song in \cite[Section 13.3]{a30} define the notion of characteristic stack functions  $\overline{\delta}^{(\beta,d)}_{ss}(\tilde{\tau})\in \overline{\operatorname{SF}}(\mathfrak{M}_{\mathcal{B}_{p}}(\tilde{\tau}),\chi,\mathbb{Q})$. Moreover, in the instance where the moduli stack contains strictly semistable objects, the authors define the ``\textit{logarithm}" of the moduli stack by the stack function $\overline{\epsilon}^{(\beta,d)}(\tilde{\tau})$ given as an element of the Hall-algebra of stack functions supported over virtual indecomposables, we will include these definitions below:
\begin{defn}\label{decompos}
Define the stack functions $\overline{\delta}^{(\beta,r)}_{ss}(\tilde{\tau}):=\overline{\delta}^{(\beta,r)}_{\mathfrak{M}^{(\beta,r)}_{\mathcal{B}_{p},ss}}(\tilde{\tau})$ in $\operatorname{SF}_{al}(\mathfrak{M}^{(\beta,r)}_{\mathcal{B}_{p},ss}(\tilde{\tau}))$ (for definition of $\operatorname{SF}_{al}$ look at \cite[Definition 3.3]{a30} for $(\beta,r)\in C(\mathcal{B}_{p})$. Now define elements $\overline{\epsilon}^{(\beta,r)}(\tilde{\tau})$ in $\operatorname{SF}_{al}(\mathfrak{M}^{(\beta,r)}_{\mathcal{B}_{p},ss}(\tilde{\tau}))$
\begin{equation}\label{decompos-eq}
\overline{\epsilon}^{(\beta,r)}(\tilde{\tau})=\sum_{\begin{subarray}{1}n\geq 1,(\beta_{1},r_{1}),\cdots, (\beta_{n},r_{n})\in C(\mathcal{B}_{p})\\(\beta_{1},r_{1})+\cdots+(\beta_{n},r_{n})=(\beta,r)\\\tilde{\tau}(\beta_{i},r_{i})=\tilde{\tau}(\beta,r) \forall i\end{subarray}} \frac{(-1)^{n-1}}{n}\overline{\delta}^{(\beta_{1},r_{1})}_{ss}(\tilde{\tau})*\overline{\delta}^{(\beta_{2},r_{2})}_{ss}(\tilde{\tau})*\cdots*\overline{\delta}^{(\beta_{n},r_{n})}_{ss}(\tilde{\tau}),
\end{equation}
where $*$ is the Ringel-Hall multiplication defined in \cite[Definition 3.3]{a30}.
\end{defn}

Our goal in the remainder of this article is to first evaluate the element of the Hall algebra $\overline{\epsilon}^{(\beta,r)}(\tilde{\tau})$ above, by explicitly computing the right hand side of Equation \eqref{decompos-eq} and then calculate the invariants. To do this one needs to apply the Joyce-Song ``\textit{Lie algebra morphism}" to $\overline{\epsilon}^{(\beta,r)}(\tilde{\tau})$. In order to clarify the latter, we need further definitions:

\begin{defn}
\cite[Definition 13.3]{a30}. Define the Euler form in $\mathcal{B}_{p}$ as $\bar{\chi}_{\mathcal{B}_{p}}:K(\mathcal{B}_{p})\times K(\mathcal{B}_{p})\to \mathbb{Z}$ such that
\begin{equation}
\bar{\chi}_{\mathcal{B}_{p}}((\beta,d), (\gamma,e))=\bar{\chi}(\beta,\gamma)-d\bar{\chi}([\mathcal{O}_{X}(-n)], \gamma)+e\bar{\chi}([\mathcal{O}_{X}(-n)],\beta),
\end{equation} 
where $\bar{\chi}()$ is the Euler form on $K(\text{coh}(X))$.
\end{defn}
\begin{defn}\label{13.11}\footnote{Look at \cite[Definition 13.11]{a30}. }
Define $\mathcal{S}$ to be the subset of $(\beta,d)$ in $C(\mathcal{B}_{p})\subset K(\mathcal{B}_{p})$ such that $P_{\beta}(t)=\frac{k}{d!}p(t)$ for $k=0,\cdots, N$ and $d=0$ or 1 or 2. Then $\mathcal{S}$ is a finite set \cite[Theorem 3.37]{a9}. Define a Lie algebra $\tilde{L}(\mathcal{B}_{p})$ to be the $\mathbb{Q}$-vector space with the basis of symbols $\tilde{\lambda}^{(\beta,d)}$ with $(\beta,d)\in \mathcal{S}$ with the Lie bracket
\begin{equation}\label{lambda}
[\tilde{\lambda}^{(\beta,d)},\tilde{\lambda}^{(\gamma,e)}]=(-1)^{\bar{\chi}_{\mathcal{B}_{p}}((\beta,d),(\gamma,e))}\bar{\chi}_{\mathcal{B}_{p}}((\beta,d),(\gamma,e))\tilde{\lambda}^{(\beta+\gamma,d+e)}
\end{equation}
for $(\beta+\gamma,d+e)\in \mathcal{S}$ and $[\tilde{\lambda}^{(\beta,d)},\tilde{\lambda}^{(\gamma,e)}]=0$ otherwise. Here It can be seen that $\bar{\chi}_{\mathcal{B}_{p}}$ is antisymmetric and hence, equation \eqref{lambda} satisfies the Jacobi-identity and that makes $\tilde{L}(\mathcal{B}_{p})$ into a finite-dimensional nilpotent Lie algebra over $\mathbb{Q}$. 
\end{defn}
Now the Joyce-Song Lie algebra morphism is defined as the $\mathbb{Q}$-linear map $\tilde{\Psi}^{\mathcal{B}_{p}}:\textbf{SF}^{ind}_{al}\mathfrak{M}_{\mathcal{B}_{p}}\rightarrow \tilde{L}(\mathcal{B}_{p})$ by applying \cite[Definition 5.13]{a30} to the moduli stack $\mathfrak{M}_{\mathcal{B}_{p}}$ and $\tilde{L}(\mathcal{B}_{p})$\footnote{For further detail look at \cite[Definition 13.11]{a30}}. 
\begin{defn}\label{Bp-ss-defn}
Define the invariant $\textbf{B}^{ss}_{p}(X,\beta,r,\tilde{\tau})$ associated to $\tilde{\tau}$-semistable objects of type $(\beta,r)$ in $\mathcal{B}_{p}$ by$$\tilde{\Psi}^{\mathcal{B}_{p}}(\bar{\epsilon}^{(\beta,r)}(\tilde{\tau}))=\textbf{B}^{ss}_{p}(X,\beta,r,\tilde{\tau})\cdot\tilde{\lambda}^{(\beta,r)},$$
where $\tilde{\Psi}^{\mathcal{B}_{p}}$ is given by the Lie algebra morphism above. 
\end{defn}

From now on, to minimize the unnecessary computational complexity, we restrict our analysis to rank 2 pairs, i.e. $r=2$. First to motivate the intuition behind our computations we study an interesting example.

\section{Direct computation of invariants in an example}\label{stable=semistable}\label{compute-example}
\begin{example}
Computation of $\textbf{B}^{ss}_{p}(X,[\mathbb{P}^{1}],2,\tilde{\tau})$ where $X$ is given by the total space of $\mathcal{O}_{\mathbb{P}^{1}}^{\oplus 2}(-1)\rightarrow \mathbb{P}^{1}$.
\end{example}
We compute the invariant of $\tilde{\tau}$-semistable objects $(F,\mathbb{C}^{2},\phi_{\mathbb{C}^{2}})$ of type $([\mathbb{P}^{1}],2)$ in $\mathcal{B}_{p}$. In this case $F$ has rank 1 over its support and $p_{F}(n)=n+\chi(F)$. Assume that $\chi(F)=k$. In this case by computations in \cite{a52} and \cite{a53} the only semistable sheaf, $F$, with $\operatorname{ch}_{2}(F)=[\mathbb{P}^{1}]$ is given by $\mathcal{O}_{\mathbb{P}^{1}}(k-1)$ which is a stable sheaf. First we give the description of $\mathfrak{M}^{([\mathbb{P}^{1}],2)}_{\mathcal{B}_{p},ss}(\tilde{\tau})$;

 By definition, an object of type $([\mathbb{P}^{1}],2)$ in $\mathcal{B}_{p}$ is identified by a complex $\mathcal{O}_{X}(-n)^{\oplus 2}\rightarrow \iota_{*}\mathcal{O}_{\mathbb{P}^{1}}(k-1)$ where $\iota:\mathbb{P}^{1}\hookrightarrow X$ (from now on we suppress $\iota_{*}$ in our notation).  By the constructions in Section \ref{scheme-beta} the parameter scheme of $\tilde{\tau}$-semistable objects is obtained by choosing two sections $(s_{1},s_{2})$ such that $s_{i}\in \operatorname{H}^{0}((\mathcal{O}_{\mathbb{P}^{1}}(n+k-1))$ for $i=1,2$. More over since $\mathcal{O}_{\mathbb{P}^{1}}(k-1)$ is a stable sheaf, its stabilizer is given by $\mathbb{G}_{m}$. \\
An important point to note is that given a $\tilde{\tau}$-semistable object $(F,\mathbb{C}^{2},\phi_{\mathbb{C}^{2}})$ one is always able to obtain a an exact sequence of the form
\begin{equation}\label{strict-semi}
0\rightarrow (F,\mathbb{C},\phi_{\mathbb{C}})\rightarrow (F,\mathbb{C}^{2},\phi_{\mathbb{C}^{2}})\rightarrow (0,\mathbb{C},0)\rightarrow 0,
\end{equation}for every object in the moduli stack and since $\tilde{\tau}(F,\mathbb{C}^{2},\phi_{\mathbb{C}})=1\leq \tilde{\tau}(0,\mathbb{C},0)=1$ then one concludes that all the objects parametrized by the moduli stack are given by extensions of rank 1 $\tilde{\tau}$-stable objects and hence all objects are $\tilde{\tau}$-strictly-semistable. Moreover, note that giving a $\tilde{\tau}$-semistable object of the form $\mathcal{O}_{X}(-n)^{\oplus 2}\xrightarrow{(s_{1},s_{2})}F$, is equivalent to requiring the condition that $(s_{1},s_{2})\neq (0,0)$, since otherwise, one may be able to obtain a an exact sequence:$$0\rightarrow (\mathbb{C}^{2},0,0)\rightarrow (\mathbb{C}^{2},F,0)\rightarrow (0,F,0)\rightarrow 0$$ such that $\tilde{\tau}(\mathbb{C}^{2},0,0)=1>\tilde{\tau}(0,F,0)=0$, hence $(\mathbb{C}^{2},0,0)$ (weakly) destabilizes $(\mathbb{C}^{2},F,0)$ hence a contradiction!. Now use Theorem \ref{theorem81} and find that\begin{equation}\label{stack-1}
\mathfrak{M}^{([\mathbb{P}^{1}],2)}_{\mathcal{B}_{p},ss}(\tilde{\tau})=\left[\frac{ (\operatorname{H}^{0}((\mathcal{O}_{\mathbb{P}^{1}}(n+k-1)))^{\oplus 2}\backslash \{0\}/\mathbb{G}_{m}}{GL_{2}(\mathbb{C})}\right]\cong \left[\frac{\mathbb{P}( \operatorname{H}^{0}((\mathcal{O}_{\mathbb{P}^{1}}(n+k-1))^{\oplus 2})}{GL_{2}(\mathbb{C})}\right].
\end{equation}
We need to compute the element of the Hall algebra $\bar{\epsilon}^{([\mathbb{P}^{1}],2)}(\tilde{\tau})$. By applying Definition \ref{decompos} to $\mathfrak{M}^{([\mathbb{P}^{1}],2)}_{\mathcal{B}_{p},ss}(\tilde{\tau})$ we obtain:
\begin{equation}\label{compute-it}
\bar{\epsilon}^{([\mathbb{P}^{1}],2)}(\tilde{\tau})=\bar{\delta}^{([\mathbb{P}^{1}],2)}_{ss}(\tilde{\tau})-\frac{1}{2}\sum_{\beta_{k}+\beta_{l}=[\mathbb{P}^{1}]}\bar{\delta}^{(\beta_{k},1)}_{s}(\tilde{\tau})*\bar{\delta}^{(\beta_{l},1)}_{s}(\tilde{\tau}).
\end{equation}
Now we use a stratification strategy in order to decompose $\mathfrak{M}^{([\mathbb{P}^{1}],2)}_{\mathcal{B}_{p},ss}(\tilde{\tau})$ into a disjoint union of strata as follows; Since the objects in the moduli stack are of type $([\mathbb{P}^{1}],2)$, one would immediately see that the only possible decomposition of a $\tilde{\tau}$-semistable object of type $([\mathbb{P}^{1}],2)$ is given by $([\mathbb{P}^{1}],2)=([\mathbb{P}^{1}],1)+(0,1)$. This means that a strictly $\tilde{\tau}$-semistable object of type $([\mathbb{P}^{1}],2)$ is always given by an (split or non-split) extension, involving an object of type $([\mathbb{P}^{1}],1)$ and an object of type $(0,1)$. Note that it is allowable to flip the order of an extension, as long as the original strictly semistable object is re-produced. Our stratification then involves a study of the parametrizing moduli stack for the objects, depending on which extension is used to produce them. This will enable us to decompose $\mathfrak{M}^{([\mathbb{P}^{1}],2)}_{\mathcal{B}_{p},ss}(\tilde{\tau})$ into a disjoint union of split and non-split strata (i.e. induced by split or non-split extensions).\\
 \begin{defn}\label{2-strata}
Define $\mathfrak{M}^{([\mathbb{P}^{1}],2)}_{\mathcal{B}_{p},sp}(\tilde{\tau})\subset \mathfrak{M}^{([\mathbb{P}^{1}],2)}_{\mathcal{B}_{p},ss}(\tilde{\tau})$ to be the locally closed stratum over which an object of type $([\mathbb{P}^{1}],2)$ is given by split extensions involving objects of type $([\mathbb{P}^{1}],1)$ and $(0,1)$.\\
Define $\mathfrak{M}^{([\mathbb{P}^{1}],2)}_{\mathcal{B}_{p},nsp}(\tilde{\tau})\subset \mathfrak{M}^{([\mathbb{P}^{1}],2)}_{\mathcal{B}_{p},ss}(\tilde{\tau})$ to be a locally closed stratum over which an object of type $([\mathbb{P}^{1}],2)$ is given by non-split extensions involving objects of type $([\mathbb{P}^{1}],1)$ and $(0,1)$.
\end{defn}
Now we study the structure of each stratum separately.
\subsection{Stacky structure of $\mathfrak{M}^{([\mathbb{P}^{1}],2)}_{\mathcal{B}_{p},sp}(\tilde{\tau})$}
It is easy to see that any $\tilde{\tau}$-semistable objects of type $([\mathbb{P}^{1}],2)$ given as $$[\mathcal{O}_{X}(-n)^{\oplus 2}\rightarrow \mathcal{O}_{\mathbb{P}^{1}}(k-1)]\cong \left[\mathcal{O}_{X}(-n)^{\oplus 2}\rightarrow \mathcal{O}_{\mathbb{P}^{1}}(k-1)\right]\oplus \left[\mathcal{O}_{X}(-n)\rightarrow 0 \right]$$ has the property that the sections $s_{1},s_{2}$ for this object are linearly dependent on one another. Hence the underlying parameter scheme of $\tilde{\tau}$-semistable objects of this given form, is given by chooing a nonzero section of $\mathcal{O}_{\mathbb{P}^{1}}(n+k-1)$ in other words we obtain $\operatorname{H}^{0}(\mathcal{O}_{\mathbb{P}^{1}}(n+k-1))\backslash \{0\}$. Now we need to take the quotient of this space by the stabilizer group of points. We know that the condition required for a $\tilde{\tau}$-semistable object $\mathcal{O}_{X}(-n)^{\oplus 2}\xrightarrow{(s_{1},s_{2})}F$ to be given by split extensions of rank 1 objects, is that $s_{1}$ and $s_{2}$ are linearly dependent on one another. Now pick such an object given by $\mathcal{O}_{X}(-n)^{\oplus 2}\xrightarrow{(s_{1},0)}F$. The automorphisms of this object are given by the group which makes the following diagram commutative:
\begin{center}
\begin{tikzpicture}
back line/.style={densely dotted}, 
cross line/.style={preaction={draw=white, -, 
line width=6pt}}] 
\matrix (m) [matrix of math nodes, 
row sep=3em, column sep=3.5em, 
text height=1.5ex, 
text depth=0.25ex]{ 
\mathcal{O}_{X}(-n)^{\oplus 2}&\mathcal{O}_{\mathbb{P}^{1}}(k-1)\\
\mathcal{O}_{X}(-n)^{\oplus 2}&\mathcal{O}_{\mathbb{P}^{1}}(k-1)\\};
\path[->]
(m-1-1) edge node [above] {$(s_{1},0)$} (m-1-2)
(m-1-1) edge node [left] {$\cong$} (m-2-1)
(m-1-2) edge node [right] {$\cong$} (m-2-2)
(m-2-1) edge node [above] {$(s_{1},0)$} (m-2-2);
\end{tikzpicture}
\end{center}
Hence it is seen that the left vertical map needs to be given by a subgroup of $GL_{2}(\mathbb{C})$ which preserves $s_{1}$, i.e the Borel subgroup of $GL_{2}(\mathbb{C})$ whose elements are given by $2\times 2$ upper triangular matrices $\left(\begin{array}{cc}  k_{1}&k_{2}\\
 0&k_{3}
 \end{array}\right)$ where $k_{1},k_{3}\in \mathbb{G}_{m}$ and $k_{2}\in \mathbb{A}^{1}$. Having fixed one of the automorphisms via fixing $k_{1},k_{2},k_{3}$, it is seen that, by the commutativity of the square diagram, the right vertical map needs to be given by multiplication by $k_{1}$ which is an element of $\mathbb{G}_{m}$. Note that one needs to take the quotient of the parameter scheme by all  isomorphisms between any two objects in the split stratum, not just the  automorphisms of one fixed representative. In general for an object to live in the split stratum one requires the sections $(s_{1},s_{2})$ to be given by $(s_{1},a\cdot s_{1})$. We observed that fixing a representative for a split object of rank $2$ (such as fixing $(s_{1},s_{2})=(s_{1},0)$ as above) would tell us that its automorphisms are given by $\mathbb{G}_{m}^{2}\rtimes \mathbb{A}^{1}$. Hence, taking into account all possible representatives implies that the stabilizer group of objects in the spit stratum is given by $\mathbb{G}_{m}^{2}\rtimes \mathbb{A}^{1}\times \mathbb{G}_{m}$. Hence we obtain
\begin{equation}\label{tof1}
\mathfrak{M}^{([\mathbb{P}^{1}],2)}_{\mathcal{B}_{p},sp}(\tilde{\tau})=\left[\frac{\operatorname{H}^{0}(\mathcal{O}_{\mathbb{P}^{1}}(n+k-1))\backslash \{0\}}{\mathbb{G}_{m}^{2}\rtimes \mathbb{A}^{1}\times \mathbb{G}_{m}}\right]=\left[\frac{\mathbb{P}(\operatorname{H}^{0}(\mathcal{O}_{\mathbb{P}^{1}}(n+k-1)))}{\mathbb{G}_{m}^{2}\rtimes \mathbb{A}^{1}}\right].
\end{equation}
\subsection{Stacky structure of $\mathfrak{M}^{([\mathbb{P}^{1}],2)}_{\mathcal{B}_{p},nsp}(\tilde{\tau})$}\label{stacky-structure}
In this case, the objects in $\mathfrak{M}^{([\mathbb{P}^{1}],2)}_{\mathcal{B}_{p},nsp}(\tilde{\tau})$ are given by non-split extensions of the form:
\begin{equation}\label{extensions13}
\begin{tikzpicture}
back line/.style={densely dotted}, 
cross line/.style={preaction={draw=white, -, 
line width=6pt}}] 
\matrix (m) [matrix of math nodes, 
row sep=3em, column sep=3.25em, 
text height=1.5ex, 
text depth=0.25ex]{ 
0&\mathcal{O}_{X}(-n)&\mathcal{O}^{\oplus 2}_{X}(-n)&\mathcal{O}_{X}(-n)&0\\
0&\mathcal{O}_{\mathbb{P}^{1}}(k-1)&F&0&0,\\};
\path[->]
(m-1-1) edge (m-1-2)
(m-1-2) edge (m-1-3) edge node [right] {$s_{1}$} (m-2-2)
(m-1-3) edge (m-1-4) edge node [right] {$(s_{1},s_{2})$} (m-2-3)
(m-1-4) edge (m-1-5) edge (m-2-4)
(m-2-1) edge (m-2-2)
(m-2-2) edge node [above] {$\cong$} (m-2-3)
(m-2-3) edge (m-2-4)
(m-2-4) edge (m-2-5);
\end{tikzpicture}
\end{equation} 
Note that switching the place of $\mathcal{O}_{\mathbb{P}^{1}}(k-1)$ and $0$ in the bottom row of diagram \eqref{extensions13} would produce a split extension and so such extensions can not lie in the non-split stratum. Now in order to obtain non-split extensions, one needs to choose two sections $s_{1},s_{2}$ such that $s_{1}$ and $s_{2}$ are linearly independent. The set of all linearly independent choices of $s_1$ and $s_{2}$ spans a two dimensional subspace of $\operatorname{H}^{0}(\mathcal{O}_{\mathbb{P}^{1}}(n+k-1))$ which is given by the Grassmanian:$$\operatorname{G}(2,n+k).$$Now we need to take the quotient of this scheme by the stabilizer group of points in the stratum. We know that the condition required for a $\tilde{\tau}$-semistable object $\mathcal{O}_{X}(-n)^{\oplus 2}\xrightarrow{(s_{1},s_{2})}F$ to be given by nonsplit extensions of rank 1 objects is that $s_{1}$ and $s_{2}$ are linearly independent. Now pick such an object given by $\mathcal{O}_{X}(-n)^{\oplus 2}\xrightarrow{(s_{1},s_{2})}F$. The automorphisms of this object are given by the group which makes the following diagram commutative:
\begin{center}
\begin{tikzpicture}
back line/.style={densely dotted}, 
cross line/.style={preaction={draw=white, -, 
line width=6pt}}] 
\matrix (m) [matrix of math nodes, 
row sep=3em, column sep=3.5em, 
text height=1.5ex, 
text depth=0.25ex]{ 
\mathcal{O}_{X}(-n)^{\oplus 2}&\mathcal{O}_{\mathbb{P}^{1}}(k-1)\\
\mathcal{O}_{X}(-n)^{\oplus 2}&\mathcal{O}_{\mathbb{P}^{1}}(k-1)\\};
\path[->]
(m-1-1) edge node [above] {$(s_{1},s_{2})$} (m-1-2)
(m-1-1) edge node [left] {$\cong$} (m-2-1)
(m-1-2) edge node [right] {$\cong$} (m-2-2)
(m-2-1) edge node [above] {$(s_{1},s_{2})$} (m-2-2);
\end{tikzpicture}
\end{center}
Hence it is seen that the left vertical map needs to be given by a subgroup of $GL_{2}(\mathbb{C})$ whose elements are given by $2\times 2$ diagonal matrices of the form $\left(\begin{array}{cc}  k_{1}&0\\
 0&k_{1}
 \end{array}\right)$ where $k_{1}\in \mathbb{G}_{m}$. Having fixed one of the automorphisms via fixing $k_{1}$, it is seen that by the commutativity of the square diagram, the right vertical map needs to be given by multiplication by $k_{1}$ which is an element of $\mathbb{G}_{m}$. Hence we obtain
\begin{equation}\label{tof2}
\mathfrak{M}^{([\mathbb{P}^{1}],2)}_{\mathcal{B}_{p},nsp}(\tilde{\tau})=\left[\frac{\operatorname{G}(2,n+k)}{\mathbb{G}_{m}}\right].
\end{equation} 
Now we are ready to compute $\sum_{\beta_{k}+\beta_{l}=[\mathbb{P}^{1}]}\bar{\delta}^{(\beta_{k},1)}_{s}(\tilde{\tau})*\bar{\delta}^{(\beta_{l},1)}_{s}(\tilde{\tau})$ appearing on the right hand side of \eqref{compute-it}. We use the fact that
\begin{equation}\label{compute-it-2}
\sum_{\beta_{k}+\beta_{l}=[\mathbb{P}^{1}]}\bar{\delta}^{(\beta_{k},1)}_{s}(\tilde{\tau})*\bar{\delta}^{(\beta_{l},1)}_{s}(\tilde{\tau})=\bar{\delta}^{([\mathbb{P}^{1}],1)}_{s}(\tilde{\tau})*\bar{\delta}^{(0,1)}_{s}(\tilde{\tau})+\bar{\delta}^{(0,1)}_{s}(\tilde{\tau})*\bar{\delta}^{([\mathbb{P}^{1}],1)}_{s}(\tilde{\tau})
\end{equation} 
and compute each term on the right hand side of \eqref{compute-it-2} separately.
\begin{remark}
As we described above, there exists an action of $GL_{2}(\mathbb{C})$ on $\textbf{S}:=\mathbb{P}(\operatorname{H}^{0}(\mathcal{O}_{\mathbb{P}^{1}}(n+k-1))^{\oplus 2}))$. This action induces an action of the corresponding Lie algebra on the tangent space of $\textbf{S}$ given by the map:
\begin{align}\label{lie-tangent1}
\mathcal{O}_{\textbf{S}}\otimes \mathfrak{g}l_{2}(\mathbb{C})\rightarrow T_{\textbf{S}},
\end{align}
where $\mathfrak{g}l_{2}(\mathbb{C})$ denotes the Lie algebra associated to the group $GL_{2}(\mathbb{C})$. The dimension of the automorphism group of objects representing the elements of $\textbf{S}$ is given by the dimension of the stabilizer (in $GL_{2}(\mathbb{C})$) group of these elements, which itself is given by the dimension of the kernel of the map in \eqref{lie-tangent1}. On the other hand, the dimension of the kernel of the map in \eqref{lie-tangent1} is an upper-semicontinious function. Therefore by the usual arguments, we obtain a stratification of $\textbf{S}$ which induces a stratification of $\left[\frac{\textbf{S}}{GL_{2}(\mathbb{C})}\right]$ into locally closed strata, such that over each stratum the dimension of the stabilizer group is constant as we vary over points inside that stratum. Here in Definition \ref{2-strata} we stated without proof that the defined strata are locally closed in $\mathfrak{M}^{([\mathbb{P}^{1}],2)}_{\mathcal{B}_{p},ss}(\tilde{\tau})$, this is discussed in detail and for more general cases in the Appendix .
\end{remark}
\subsection{Computation of $\bar{\delta}^{([\mathbb{P}^{1}],1)}_{s}(\tilde{\tau})*\bar{\delta}^{(0,1)}_{s}(\tilde{\tau})$}\label{first-d1d1}
Given $\bar{\delta}^{([\mathbb{P}^{1}],1)}_{s}(\tilde{\tau})=(\left[\frac{\mathbb{P}(\operatorname{H}^{0}(\mathcal{O}_{\mathbb{P}^{1}}(n+k-1)))}{\mathbb{G}_{m}}\right],\rho_{1})$ and $\bar{\delta}^{(0,1)}_{s}(\tilde{\tau})=(\left[\frac{\operatorname{Spec}(\mathbb{C})}{\mathbb{G}_{m}}\right],\rho_{2})$ ($\rho_{1},\rho_{2}$ are natural embedding maps to $\mathfrak{M}_{\mathcal{B}_{p}}(\tilde{\tau})$), consider the diagram:
\begin{equation}\label{embedding2}
\begin{tikzpicture}
back line/.style={densely dotted}, 
cross line/.style={preaction={draw=white, -, 
line width=6pt}}] 
\matrix (m) [matrix of math nodes, 
row sep=3em, column sep=3em, 
text height=1.5ex, 
text depth=0.25ex]{ 
\mathcal{Z}'_{12}&\mathfrak{Exact}_{\mathcal{B}_{p}}&\left[\frac{\mathbb{P}( \operatorname{H}^{0}((\mathcal{O}_{\mathbb{P}^{1}}(n+k-1))^{\oplus 2})}{GL_{2}(\mathbb{C})}\right]\\
\left[\frac{\mathbb{P}(\operatorname{H}^{0}(\mathcal{O}_{\mathbb{P}^{1}}(n+k-1)))}{\mathbb{G}_{m}}\right]\times \left[\frac{\operatorname{Spec}(\mathbb{C})}{\mathbb{G}_{m}}\right]&\mathfrak{M}_{\mathcal{B}_{p}}(\tilde{\tau})\times \mathfrak{M}_{\mathcal{B}_{p}}(\tilde{\tau})&\\};
\path[->]
(m-1-1) edge node [above] {$\Phi$} (m-1-2)
(m-1-2) edge node [above] {$\pi_{2}$} (m-1-3)
(m-1-1) edge (m-2-1)
(m-1-2) edge node [right] {$\pi_{1}\times \pi_{3}$}(m-2-2)
(m-2-1) edge node [above] {$\rho_{1}\times \rho_{2}$} (m-2-2);
\end{tikzpicture}
\end{equation} 
where $\mathcal{Z}'_{12}$ is given by the scheme parametrizing the set of commutative diagrams: 
\begin{equation}\label{extensions13-1}
\begin{tikzpicture}
back line/.style={densely dotted}, 
cross line/.style={preaction={draw=white, -, 
line width=6pt}}] 
\matrix (m) [matrix of math nodes, 
row sep=3em, column sep=3.25em, 
text height=1.5ex, 
text depth=0.25ex]{ 
0&\mathcal{O}_{X}(-n)&\mathcal{O}^{\oplus 2}_{X}(-n)&\mathcal{O}_{X}(-n)&0\\
0&\mathcal{O}_{\mathbb{P}^{1}}(k-1)&F&0&0,\\};
\path[->]
(m-1-1) edge (m-1-2)
(m-1-2) edge (m-1-3) edge node [right] {$s_{1}$} (m-2-2)
(m-1-3) edge (m-1-4) edge node [right] {$(s_{1},s_{2})$} (m-2-3)
(m-1-4) edge (m-1-5) edge (m-2-4)
(m-2-1) edge (m-2-2)
(m-2-2) edge node [above] {$\cong$} (m-2-3)
(m-2-3) edge (m-2-4)
(m-2-4) edge (m-2-5);
\end{tikzpicture}
\end{equation} 
Since that the extensions in \eqref{extensions13-1} have the possibility of being split or non-split, therefore, we consider each case separately and define the multiplication $$\bar{\delta}^{([\mathbb{P}^{1}],1)}_{s}(\tilde{\tau})*\bar{\delta}^{(0,1)}_{s}(\tilde{\tau})$$ in each case separately as follows below\\
\begin{defn}\label{11-split}
Let $[\bar{\delta}^{([\mathbb{P}^{1}],1)}_{s}(\tilde{\tau})*\bar{\delta}^{(0,1)}_{s}(\tilde{\tau})]_{sp}$ denote the stratum of $(\pi_{2}\circ\Phi)_{*}\mathcal{Z}'_{12}$ over which the points are represented by split extensions given by the commutative diagram in  \eqref{extensions13-1}.  
\end{defn}
\begin{defn}\label{11-nonsplit}
Let $[\bar{\delta}^{([\mathbb{P}^{1}],1)}_{s}(\tilde{\tau})*\bar{\delta}^{(0,1)}_{s}(\tilde{\tau})]_{nsp}$ denote the stratum of $(\pi_{2}\circ\Phi)_{*}\mathcal{Z}'_{12}$ over which the points are represented by the non-split extensions given by coomutative diagram in  \eqref{extensions13-1}.  
\end{defn}
Therefore $$\bar{\delta}^{([\mathbb{P}^{1}],1)}_{s}(\tilde{\tau})*\bar{\delta}^{(0,1)}_{s}(\tilde{\tau})=[\bar{\delta}^{([\mathbb{P}^{1}],1)}_{s}(\tilde{\tau})*\bar{\delta}^{(0,1)}_{s}(\tilde{\tau})]_{sp}+[\bar{\delta}^{([\mathbb{P}^{1}],1)}_{s}(\tilde{\tau})*\bar{\delta}^{(0,1)}_{s}(\tilde{\tau})]_{nsp}$$
Now we compute $[\bar{\delta}^{([\mathbb{P}^{1}],1)}_{s}(\tilde{\tau})*\bar{\delta}^{(0,1)}_{s}(\tilde{\tau})]_{sp}$. This amounts to choosing sections $s_{1},s_{2}$ so that $s_{1}$ and $s_{2}$ are linearly depending on one another. The scheme parametrizing the nonzero sections $s_{1}$ is given by $\mathbb{P}(\operatorname{H}^{0}(\mathcal{O}_{\mathbb{P}^{1}}(n+k-1)))$. Now we take the quotient of this scheme by the stabilizer group of points. Similar to arguments in \cite[Lemma 12.1]{a41} given any point in $\mathbb{P}(\operatorname{H}^{0}(\mathcal{O}_{\mathbb{P}^{1}}(n+k-1)))$ represented by the extension in diagram \eqref{extensions13-1}, its stabilizer group is given by the semi-direct product $\mathbb{G}_{m}^{2}\rtimes \operatorname{Hom}(E_{3},E_{1})$ where each factor of $\mathbb{G}_{m}$ amounts to the stabilizer group of objects given as $E_{3}:=\mathcal{O}_{X}(-n)\rightarrow 0$ and $E_{1}:=\mathcal{O}_{X}(-n)\rightarrow \mathcal{O}_{\mathbb{P}^{1}}(k-1)$ respectively. Note that the extra factor of $\mathbb{A}^{1}$ will not appear as a part of the stabilizer group since, by the given description of $E_{1}$ and $E_{3}$ we know that $\operatorname{Hom}(E_{3},E_{1})=0$ for every such $E_{1}$ and $E_{3}$. We obtain the following conclusion:
\begin{align}\label{deldel-sp-2}
&
[\bar{\delta}^{([\mathbb{P}^{1}],1)}_{s}(\tilde{\tau})*\bar{\delta}^{(0,1)}_{s}(\tilde{\tau})]_{sp}=\left[\frac{\mathbb{P}(\operatorname{H}^{0}(\mathcal{O}_{\mathbb{P}^{1}}(n+k-1)))}{\mathbb{G}_{m}^{2}}\right].
\end{align}
Now we compute $[\bar{\delta}^{([\mathbb{P}^{1}],1)}_{s}(\tilde{\tau})*\bar{\delta}^{(0,1)}_{s}(\tilde{\tau})]_{nsp}$. This amounts to choosing $s_{1},s_{2}$ so that $s_{1}$ and $s_{2}$ are linearly independent and the extension in diagram \eqref{extensions13-1} becomes non-split. Note that for any fixed value of $s_{1}$, one has $\mathbb{P}^{1}$ worth of choices for $s_{2}$. Now we need to consider all possible choices of $s_{1}$ and in doing so, we require the sections $s_{1},s_{2}$ to remain linearly independent. This gives the flag variety $\operatorname{F}(1,2,n+k)$. Hence we obtain:
\begin{equation}\label{nsp-deldel}
[\bar{\delta}^{([\mathbb{P}^{1}],1)}_{s}(\tilde{\tau})*\bar{\delta}^{(0,1)}_{s}(\tilde{\tau})]_{nsp}=\left[\frac{\operatorname{F}(1,2,n+k)}{\mathbb{G}_{m}}\right].
\end{equation}
Note that the factor of $\mathbb{G}_{m}$ in the denominator of \eqref{nsp-deldel} is due to the fact that we have used one of the $\mathbb{G}_{m}$ factors in projectivising the bundle of $s_{2}$-choices over the Grassmanian. 
We finish this section by summarizing our computation. By \eqref{deldel-sp-2} and \eqref{nsp-deldel} one obtains:
\begin{equation}\label{summary1}
\bar{\delta}^{([\mathbb{P}^{1}],1)}_{s}(\tilde{\tau})*\bar{\delta}^{(0,1)}_{s}(\tilde{\tau})=\left[\frac{\mathbb{P}(\operatorname{H}^{0}(\mathcal{O}_{\mathbb{P}^{1}}(n+k-1)))}{\mathbb{G}_{m}^{2}}\right]+\left[\frac{\operatorname{F}(1,2,n+k)}{\mathbb{G}_{m}}\right]
\end{equation}
\subsection{Computation of $\bar{\delta}^{(0,1)}_{s}(\tilde{\tau})*\bar{\delta}^{([\mathbb{P}^{1}],1)}_{s}(\tilde{\tau})$}\label{second-d1d1}
Now change the order of $\bar{\delta}^{(0,1)}_{s}(\tilde{\tau})$ and $\bar{\delta}^{([\mathbb{P}^{1}],1)}_{s}(\tilde{\tau})$ and obtain a diagram
\begin{equation}\label{embedding3}
\begin{tikzpicture}
back line/.style={densely dotted}, 
cross line/.style={preaction={draw=white, -, 
line width=6pt}}] 
\matrix (m) [matrix of math nodes, 
row sep=3em, column sep=3em, 
text height=1.5ex, 
text depth=0.25ex]{ 
\mathcal{Z}'_{21}&\mathfrak{Exact}_{\mathcal{B}_{p}}&\left[\frac{\mathbb{P}( \operatorname{H}^{0}((\mathcal{O}_{\mathbb{P}^{1}}(n+k-1))^{\oplus 2})}{GL_{2}(\mathbb{C})}\right]\\
\left[\frac{\operatorname{Spec}(\mathbb{C})}{\mathbb{G}_{m}}\right]\times\left[\frac{\mathbb{P}(\operatorname{H}^{0}(\mathcal{O}_{\mathbb{P}^{1}}(n+k-1)))}{\mathbb{G}_{m}}\right]&\mathfrak{M}_{\mathcal{B}_{p}}(\tilde{\tau})\times \mathfrak{M}_{\mathcal{B}_{p}}(\tilde{\tau})&\\};
\path[->]
(m-1-1) edge node [above] {$\Phi$} (m-1-2)
(m-1-2) edge node [above] {$\pi_{2}$} (m-1-3)
(m-1-1) edge (m-2-1)
(m-1-2) edge node [right] {$\pi_{1}\times \pi_{3}$}(m-2-2)
(m-2-1) edge node [above] {$\rho_{2}\times \rho_{1}$} (m-2-2);
\end{tikzpicture}
\end{equation} 
Here $\mathcal{Z}'_{21}$ is given by the scheme parametrizing the set of commutative diagrams: 
\begin{equation}\label{extensions13-2}
\begin{tikzpicture}
back line/.style={densely dotted}, 
cross line/.style={preaction={draw=white, -, 
line width=6pt}}] 
\matrix (m) [matrix of math nodes, 
row sep=3em, column sep=3.25em, 
text height=1.5ex, 
text depth=0.25ex]{ 
0&\mathcal{O}_{X}(-n)&\mathcal{O}^{\oplus 2}_{X}(-n)&\mathcal{O}_{X}(-n)&0\\
0&0&F&\mathcal{O}_{\mathbb{P}^{1}}(k-1)&0,\\};
\path[->]
(m-1-1) edge (m-1-2)
(m-1-2) edge (m-1-3) edge node [right] {$0$} (m-2-2)
(m-1-3) edge (m-1-4) edge node [right] {$(0,s_{2})$} (m-2-3)
(m-1-4) edge (m-1-5) edge (m-2-4)
(m-2-1) edge (m-2-2)
(m-2-2) edge node [above] {$\cong$} (m-2-3)
(m-2-3) edge (m-2-4)
(m-2-4) edge (m-2-5);
\end{tikzpicture}
\end{equation} 
Note that the computation in this case is easier since the only possible extensions of the form given in \eqref{extensions13-2} are the split extensions. The computation in this case is similar to computation in \eqref{deldel-sp-2} except that one needs to take into account that over any point represented by an extension (as in diagram \eqref{extensions13-2}) of $E_{1}:=\mathcal{O}_{X}(-n)\rightarrow 0$ and $E_{3}:=\mathcal{O}_{X}(-n)\rightarrow \mathcal{O}_{\mathbb{P}^{1}}(k-1)$ we have  $\operatorname{Hom}(E_{3},E_{1})\cong \mathbb{A}^{1}$. Hence by similar discussions we obtain :
\begin{align}\label{deldel-sp-3}
&
\bar{\delta}^{([\mathbb{P}^{1}],1)}_{s}(\tilde{\tau})*\bar{\delta}^{(0,1)}_{s}(\tilde{\tau})=\left[\frac{\mathbb{P}(\operatorname{H}^{0}(\mathcal{O}_{\mathbb{P}^{1}}(n+k-1)))}{\mathbb{G}_{m}^{2}\rtimes \mathbb{A}^{1}}\right].
\end{align}  
\subsection{Computation of $\bar{\epsilon}^{(\beta,2)}(\tilde{\tau})$} 
By \eqref{compute-it}, \eqref{tof1}, \eqref{tof2}, \eqref{compute-it-2}, \eqref{summary1} and \eqref{deldel-sp-3} we obtain:
\begin{align}\label{sagdast}
&
\bar{\epsilon}^{(\beta,2)}(\tilde{\tau})=\left[\frac{\mathbb{P}(\operatorname{H}^{0}(\mathcal{O}_{\mathbb{P}^{1}}(n+k-1)))}{\mathbb{G}_{m}^{2}\rtimes \mathbb{A}^{1}}\right]+\left[\frac{\operatorname{G}(2,n+k)}{\mathbb{G}_{m}}\right]\notag\\
&
-\frac{1}{2}\cdot \left[\frac{\mathbb{P}(\operatorname{H}^{0}(\mathcal{O}_{\mathbb{P}^{1}}(n+k-1)))}{\mathbb{G}_{m}^{2}}\right]-\frac{1}{2}\cdot \left[\frac{\operatorname{F}(1,2,n+k)}{\mathbb{G}_{m}}\right]\notag\\
&
-\frac{1}{2}\cdot\left[\frac{\mathbb{P}(\operatorname{H}^{0}(\mathcal{O}_{\mathbb{P}^{1}}(n+k-1)))}{\mathbb{G}_{m}^{2}\rtimes \mathbb{A}^{1}}\right].
\end{align}
Now use the decomposition used by  Joyce and Song in \cite{a30} (page 158) and write 
\begin{align}
&
\left[\frac{\mathbb{P}( \operatorname{H}^{0}(\mathcal{O}_{\mathbb{P}^{1}}(n+k-1)))}{\mathbb{G}_{m}^{2}\rtimes \mathbb{A}^{1}}\right]=\notag\\
&
F(G,\mathbb{G}_{m}^{2},\mathbb{G}_{m}^{2})\cdot \left[\frac{\mathbb{P}(\operatorname{H}^{0}(\mathcal{O}_{\mathbb{P}^{1}}(n+k-1)))}{\mathbb{G}_{m}^{2}}\right]+F(G,\mathbb{G}_{m}^{2},\mathbb{G}_{m})\cdot \left[\frac{\mathbb{P}(\operatorname{H}^{0}(\mathcal{O}_{\mathbb{P}^{1}}(n+k-1)))}{\mathbb{G}_{m}}\right],
\end{align}
where $F(G,\mathbb{G}_{m}^{2},\mathbb{G}_{m}^{2})=1$ and $F(G,\mathbb{G}_{m}^{2},\mathbb{G}_{m})=-1$. Equation \eqref{sagdast} simplifies as follows:
\begin{align}\label{cancel-it} 
&
\bar{\epsilon}^{(\beta,2)}(\tilde{\tau})=\cancel{\left[\frac{\mathbb{P}(\operatorname{H}^{0}(\mathcal{O}_{\mathbb{P}^{1}}(n+k-1)))}{\mathbb{G}_{m}^{2}}\right]}-\left[\frac{\mathbb{P}(\operatorname{H}^{0}(\mathcal{O}_{\mathbb{P}^{1}}(n+k-1)))}{\mathbb{G}_{m}}\right]\notag\\
&
+\left[\frac{\operatorname{G}(2,n+k)}{\mathbb{G}_{m}}\right]-\cancel{\frac{1}{2}\cdot \left[\frac{\mathbb{P}(\operatorname{H}^{0}(\mathcal{O}_{\mathbb{P}^{1}}(n+k-1)))}{\mathbb{G}_{m}^{2}}\right]}\notag\\
&
\frac{1}{2}\cdot \left[\frac{\operatorname{F}(1,2,n+k)}{\mathbb{G}_{m}}\right]-\cancel{\frac{1}{2}\cdot \left[\frac{\mathbb{P}(\operatorname{H}^{0}(\mathcal{O}_{\mathbb{P}^{1}}(n+k-1)))}{\mathbb{G}_{m}^{2}}\right]}\notag\\
&
+\frac{1}{2}\cdot \left[\frac{\mathbb{P}(\operatorname{H}^{0}(\mathcal{O}_{\mathbb{P}^{1}}(n+k-1)))}{\mathbb{G}_{m}}\right]
\end{align}
Now use Definition \ref{chichi-migi} and write:
\begin{align}\label{sagkosh}
&
\left[\frac{\operatorname{G}(2,n+k)}{\mathbb{G}_{m}}\right]=\chi(\operatorname{G}(2,n+k))\cdot \left[\frac{\operatorname{Spec}(\mathbb{C})}{\mathbb{G}_{m}}\right]
\end{align}
and
\begin{align}\label{the-factor}
&
\left[\frac{\operatorname{F}(1,2,n+k)}{\mathbb{G}_{m}}\right]=\chi(\operatorname{F}(1,2,n+k))\cdot \left[\frac{\operatorname{Spec}(\mathbb{C})}{\mathbb{G}_{m}}\right]\notag\\
&
=\chi(\mathbb{P}^{1})\cdot\chi(\operatorname{G}(2,n+k))\cdot \left[\frac{\operatorname{Spec}(\mathbb{C})}{\mathbb{G}_{m}}\right]=2\cdot \chi(\operatorname{G}(2,n+k))\cdot \left[\frac{\operatorname{Spec}(\mathbb{C})}{\mathbb{G}_{m}}\right]
\end{align}
where the second equality is due to the fact that the topological Euler characteristic of a vector bundle over a base variety is equal to the Euler characteristic of its fibers times the Euler characteristic of the base. By \eqref{sagdast} and \eqref{the-factor} we obtain:
\begin{align}
&
-\frac{1}{2}\cdot \left[\frac{\mathbb{P}(\operatorname{H}^{0}(\mathcal{O}_{\mathbb{P}^{1}}(n+k-1)))}{\mathbb{G}_{m}}\right]+\chi(\operatorname{G}(2,n+k))\cdot \left[\frac{\operatorname{Spec}(\mathbb{C})}{\mathbb{G}_{m}}\right]\notag\\
&
-2\cdot \frac{1}{2}\cdot \chi(\operatorname{G}(2,n+k))\cdot \left[\frac{\operatorname{Spec}(\mathbb{C})}{\mathbb{G}_{m}}\right]\notag\\
&
=-\frac{1}{2}\cdot \left[\frac{\mathbb{P}(\operatorname{H}^{0}(\mathcal{O}_{\mathbb{P}^{1}}(n+k-1)))}{\mathbb{G}_{m}}\right]=-\frac{1}{2}\chi(\mathbb{P}(\operatorname{H}^{0}(\mathcal{O}_{\mathbb{P}^{1}}(n+k-1))))\cdot \left[\frac{\operatorname{Spec}(\mathbb{C})}{\mathbb{G}_{m}}\right]\notag\\
&
=-\frac{1}{2}(n+k)\cdot\left[\frac{\operatorname{Spec}(\mathbb{C})}{\mathbb{G}_{m}}\right].
\end{align}
\subsection{Computation of the invariant}\label{factor-1}
Now apply the Lie algebra morphism $\tilde{\Psi}^{\mathcal{B}_{p}}$ in Definition \ref{13.11} to $\bar{\epsilon}^{([\mathbb{P}^{1}],2)}(\tilde{\tau})$. By definition:
\begin{align}
&
\tilde{\Psi}^{\mathcal{B}_{p}}(\bar{\epsilon}^{([\mathbb{P}^{1}],2)}(\tilde{\tau}))=\chi^{na}(-\frac{1}{2}(n+k)\cdot\left[\frac{\operatorname{Spec}(\mathbb{C})}{\mathbb{G}_{m}}\right],(\mu\circ i_{2})^{*}\nu_{\mathfrak{M}^{(0,2)}_{\mathcal{B}_{p}}})\tilde{\lambda}^{([\mathbb{P}^{1}],2)}.\notag\\
\end{align} 
Note that by Equation \eqref{stack-1}, $\mathfrak{M}^{([\mathbb{P}^{1}],2)}_{\mathcal{B}_{p},ss}(\tilde{\tau})=\left[\frac{\mathbb{P}( \operatorname{H}^{0}((\mathcal{O}_{\mathbb{P}^{1}}(n+k-1))^{\oplus 2})}{GL_{2}(\mathbb{C})}\right]$ and hence $\left[\frac{\operatorname{Spec}(\mathbb{C})}{\mathbb{G}_{m}}\right]$ has relative dimension $-1-(2n+2k-5)=4-2n-2k$ over $\mathfrak{M}^{([\mathbb{P}^{1}],2)}_{\mathcal{B}_{p},ss}(\tilde{\tau})$. Moreover, $\left[\frac{\operatorname{Spec}(\mathbb{C})}{\mathbb{G}_{m}}\right]$ is given by  a single point with Behrend's multiplicity $-1$ and $$(\mu\circ i_{2})^{*}\nu_{\mathfrak{M}^{(0,2)}_{\mathcal{B}_{p}}})\tilde{\lambda}^{([\mathbb{P}^{1}],2)}=(-1)^{4-2n-2k}\cdot\nu_{\left[\frac{\operatorname{Spec}(\mathbb{C})}{\mathbb{G}_{m}}\right]}=\nu_{\left[\frac{\operatorname{Spec}(\mathbb{C})}{\mathbb{G}_{m}}\right]},$$therefore: 
\begin{align}
&
\tilde{\Psi}^{\mathcal{B}_{p}}(\bar{\epsilon}^{([\mathbb{P}^{1}],2)}(\tilde{\tau}))=\chi^{na}\left(-\frac{1}{2}(n+k)\cdot\left[\frac{\operatorname{Spec}(\mathbb{C})}{\mathbb{G}_{m}}\right],\nu_{\left[\frac{\operatorname{Spec}(\mathbb{C})}{\mathbb{G}_{m}}\right]} \right)\tilde{\lambda}^{([\mathbb{P}^{1}],2)}=(-1)^{1}\cdot \frac{-1}{2}(n+k)\cdot\tilde{\lambda}^{([\mathbb{P}^{1}],2)}.\notag\\
\end{align}
Finally by Definition \ref{Bp-ss-defn} we obtain:
\begin{equation}\label{first-answer}
\textbf{B}^{ss}_{p}(X,\beta,2,\tilde{\tau})=\frac{1}{2}(n+k).
\end{equation}
Note that a simple calculation shows that substituting $([\mathbb{P}^{1}],2)$ for $(\beta,2)$ in \cite[Equation 5.2]{a41} would give the same answer as in \eqref{first-answer}. Hence our result is compatible with wallcrossing calculations.

To summarize, in section \ref{compute-example} we introduced our strategy of direct computations over an example where $r=2$ and $\beta$ was given by the irreducible class $[\mathbb{P}^{1}]$. Our strategy involved computing the weighted Euler characteristic of the right hand side of Equation \eqref{compute-it}. After carefully analyzing the stacky structure of the moduli space of objects of type $([\mathbb{P}^{1}],2)$ (Section \ref{stacky-structure}), first we computed the second summand on the right hand side of Equation \eqref{compute-it} (sections \ref{first-d1d1} and \ref{second-d1d1}). Then, by the stratification of the original moduli stack, we showed that the first summand on the right hand side of Equation \eqref{compute-it} is written as a sum of the characteristic stack functions of the strata in a suitable way, such that essnetially the difference of the first and second summands on the right hand side of \eqref{compute-it} turned out to be given as a sum of stack functions supported over virtual indecomposables. This enabled us to apply the Lie algebra morphism defined in Definition \ref{13.11} to both sides of Equation \eqref{compute-it} and obtain the final answer in Equation \eqref{first-answer}. 
\section{Direct calculations in general cases} \label{sec-wall}
In this section we extend the computational strategy in section \ref{compute-example} to one level of generalization further, i.e. to the case where $r=2$, however $\beta$ can have the possibility of being given by a reducible class. Eventually the final generalization, which is to consider semistable pairs with rank $>2$, follows the same strategy as will be described in detail below. 

\begin{remark}
As was seen above, a key condition to hold for our strategy to work, is that the strata involved in our stratification are locally closed and disjoint from one another. In order to show that such stratification is possible for general cases, we start below by assuming that an object $E_{2}:\mathcal{O}^{2}_{X}(-n)\to F_{2}$, with $F_{2}$ being a semistable sheaf with reducible Chern character, is decomposable into rank 1 objects $E_{1}:=\mathcal{O}_{X}(-n)\to F_{1}$ and  $E_{3}:=\mathcal{O}_{X}(-n)\to F_{3}$, where $F_{1},F_{3}$ are stable sheaves. We then show that the moduli space parameterizing $E_{2}$ can be decomposed into a disjoint union of locally closed strata, depending on how $E_{2}$ is produced by extensions of $E_{1}, E_{3}$. In other words, we will show below that the extensions of the form $0\to E_{1}\to E_{2}\to E_{3}\to 0$ (depending on being split or non-split), or the ones with the order of the extension flipped as in $0\to E_{3}\to E_{2}\to E_{1}\to 0$ are all locally closed and disjoint from one another. 
\end{remark}

\begin{assumption}\label{big-assumption} 
Throughout this section we assume that a $\tilde{\tau}$-semistable object $(F,V,\phi_{V})\in \mathcal{B}_{p}$ of type $(\beta,2)$ has the property that $\beta$ is either indecomposable or  it satisfies the condition that if $\beta=\beta_{k}+\beta_{l}$ (i.e if $\beta$ is decomposable) then $\beta_{k}=[F_{k}]$ and $\beta_{l}=[F_{l}]$ such that $F_{k}$ and $F_{l}$ are $\tau$-stable sheaves with fixed reduced Hilbert polynomial $p$. In other words, $\beta$ \textbf{can not} be decomposed into smaller classes whose associated sheaves are not $\tau$-stable.
\end{assumption}
\begin{lemma}\label{hom-eval}
Let $\mathfrak{S}_{s}^{(\beta_{k},1)}$ and $\mathfrak{S}_{s}^{(\beta_{l},1)}$ (for some $\beta_{k}$ and $\beta_{l}$) denote the underlying schemes, as in Section \ref{scheme-beta} given as a bundle $\mathcal{P}$ over the $\tau$-stable locus  of the Quot scheme, $\mathcal{Q}^{s}\subset \mathcal{Q}$, parameterizing maps $\mathcal{O}_{X}(-n)\to F$ such that the Chern character of $F$ is given by $\beta_{k}$ and $\beta_{l}$ respectively, satisfying the Assumption \ref{big-assumption}. Then, given a tuple of objects, $(E_{1},E_{3})\in \mathfrak{S}_{s}^{(\beta_{k},1)}\times \mathfrak{S}_{s}^{(\beta_{l},1)}$, the following is true:$$\operatorname{Hom}(E_{3},E_{1})=\mathbb{A}^{1},$$if $E_{1}\cong E_{3}$ and$$\operatorname{Hom}(E_{3},E_{1})=0,$$if $E_{1}\ncong E_{3}$.
\end{lemma}
\begin{proof} This is mainly due to the assumption on the stability of the sheaves involved; Fix $E_{1}:=[\mathcal{O}_{X}(-n)\rightarrow F_{1}]\in \mathfrak{S}_{s}^{(\beta_{k},1)}$ and $E_{3}:=[\mathcal{O}_{X}(-n)\rightarrow F_{3}]\in \mathfrak{S}_{s}^{(\beta_{l},1)}$.  Consider a map $\psi:E_{1}\rightarrow E_{3}$. By definition, the morphism between $E_{1}$ and $E_{3}$ is defined by a morphism $\psi_{F}: F_{1}\rightarrow F_{3}$ which makes the following diagram commutative:
\begin{equation}\label{A1-0}
\begin{tikzpicture}
back line/.style={densely dotted}, 
cross line/.style={preaction={draw=white, -, 
line width=6pt}}] 
\matrix (m) [matrix of math nodes, 
row sep=3em, column sep=3.5em, 
text height=1.5ex, 
text depth=0.25ex]{ 
O_{X}(-n)&F_{1}\\
O_{X}(-n)&F_{3}\\};
\path[->]
(m-1-1) edge  (m-1-2)
(m-1-1) edge node [left] {$id_{\mathcal{O}_{X}}$}(m-2-1)
(m-1-2) edge node [right] {$\psi_{F}$} (m-2-2)
(m-2-1) edge (m-2-2);
\end{tikzpicture}
\end{equation}
By assumption $F_{1}$ and $F_{3}$ are given as stable sheaves with fixed reduced Hilbert polynomial $p$. Hence any nontrivial sheaf homomorphism from $F_{1}$ to $F_{3}$ is an isomorphism. Moreover by simplicity of stable sheaves any such nontrivial isomorphism is identified with $\mathbb{A}^{1}$. Now use the commutativity of the diagram in \eqref{A1-0}. 
\end{proof}
Given a reducible class $\beta$ there might be multiple ways to decompose it into underlying smaller classes $\beta_{k}, \beta_{l}$. We will now show that having fixed a class $\beta$ and moving from one decomposition type, say $\beta=\beta_{k}+\beta_{l}$, to another given by $\beta=\beta_{k'}+\beta_{l'}$, will correspond to moving from one stratum to another in the ambient moduli space of $E_{2}$, such that the corresponding strata are disjoint from one another; 
\begin{lemma}\label{non-isom}
Fix $\beta_{k}$ and $\beta_{l}$ such that $\beta_{k}+\beta_{l}=\beta$ and $\beta$, $\beta_{k}$ and $\beta_{l}$ satisfy the condition in Assumption \ref{big-assumption}. Consider an object $E_{2}$ given as an element of $\mathfrak{S}^{(\beta,2)}_{ss}$ (by this notation we mean that an object $E_{2}$ may consist of a semistable sheaf with reducible class $\beta$) which fits into a non-split extension of objects$$0\rightarrow E_{1}\rightarrow E_{2}\rightarrow E_{3}\rightarrow 0,$$such that $E_{1}$ and $E_{3}$ are given by the elements of $\mathfrak{S}_{s}^{(\beta_{k},1)}$ and $\mathfrak{S}_{s}^{(\beta_{l},1)}$ respectively. Now suppose $E'_{1}$ and $E'_{3}$ are objects with classes $(\beta_{k'},1)$ and $(\beta_{l'},1)$ respectively such that $\beta_{k'}+\beta_{l'}=\beta $,  $\beta_{k'}\neq \beta_{k}$, $\beta_{l'}\neq \beta_{l}$ and furthermore, $\beta$, $\beta'_{k}$ and $\beta'_{l}$ also satisfy the condition in Assumption \ref{big-assumption}. Then it is true that $E_{2}$ \textbf{can not} be given as an extension$$0\rightarrow E'_{1}\rightarrow E_{2}\rightarrow E'_{3}\rightarrow 0.$$
\end{lemma}
\begin{proof} If $E_{2}$ is given by both extensions then we obtain a map between the two short exact sequences: 
\begin{equation}
\begin{tikzpicture}
back line/.style={densely dotted}, 
cross line/.style={preaction={draw=white, -, 
line width=6pt}}] 
\matrix (m) [matrix of math nodes, 
row sep=1em, column sep=2.em, 
text height=1.5ex, 
text depth=0.25ex]{ 
0&E_{1}&E_{2}&E_{3}&0\\
0&E'_{1}&E_{2}&E'_{3}&0\\};
\path[->]
(m-1-1) edge  (m-1-2)
(m-1-2) edge node [above] {$\iota$} (m-1-3)
(m-1-3) edge (m-1-4) edge node [right] {$\cong$} (m-2-3)
(m-1-4) edge (m-1-5)
(m-2-1) edge (m-2-2)
(m-2-2) edge (m-2-3)
(m-2-3) edge node [below] {$p$} (m-2-4)
(m-2-4) edge (m-2-5);
\end{tikzpicture}.
\end{equation}  
Hence we obtain a map $p\circ \iota:E_{1}\rightarrow E'_{3}$. Since by assumption $\beta_{k}\neq \beta_{k'}$ (this means $\beta_{l}\neq \beta_{l'}$ because $\beta_{k}+\beta_{l}=\beta_{k'}+\beta_{l'}=\beta$), by Lemma \ref{hom-eval} we conclude that $p\circ \iota$ is the zero map. Hence $p\circ \iota$ factors through the map $\iota'\circ g$ in the following diagram 
\begin{equation}\label{khol}
\begin{tikzpicture}
back line/.style={densely dotted}, 
cross line/.style={preaction={draw=white, -, 
line width=6pt}}] 
\matrix (m) [matrix of math nodes, 
row sep=1em, column sep=2.5em, 
text height=1.5ex, 
text depth=0.25ex]{ 
0&E_{1}&E_{2}&E_{3}&0\\
0&E'_{1}&E_{2}&E'_{3}&0\\};
\path[->]
(m-1-1) edge  (m-1-2) 
(m-1-2) edge node [above] {$\iota$} (m-1-3) edge node [right] {$g$} (m-2-2)
(m-1-3) edge (m-1-4) edge node [right] {$\cong$} (m-2-3)
(m-1-4) edge (m-1-5)
(m-2-1) edge (m-2-2)
(m-2-2) edge node [above] {$\iota'$} (m-2-3)
(m-2-3) edge node [below] {$p$} (m-2-4)
(m-2-4) edge (m-2-5);
\end{tikzpicture}.
\end{equation}  
Since $E_{1}\ncong E'_{1}$, by Lemma \ref{hom-eval}, $g$ is the zero map. By considering the left commutative square in \eqref{khol} we obtain a contradiction, since the image of $E_{1}$ in $E_{2}$ can not always be zero. 
\end{proof}
Now we show that flipping the order of non-split extensions will again induce disjoint parameterizing stratum in the ambient moduli space; 
\begin{lemma}\label{isom-nonisom}
Fix $\beta_{k}$ and $\beta_{l}$ such that $\beta_{k}+\beta_{l}=\beta$ as in Assumption \ref{big-assumption}. Now Consider $E_{2}\in \mathfrak{S}^{(\beta,2)}_{ss}$ which fits into a non-split extension$$0\rightarrow E_{1}\rightarrow E_{2}\rightarrow E_{3}\rightarrow 0,$$where $E_{1}$ and $E_{3}$ are given by the elements of $\mathfrak{S}_{s}^{(\beta_{k},1)}$ and $\mathfrak{S}_{s}^{(\beta_{l},1)}$ respectively. Then, the object $E_{2}$ can not be given as an extension$$0\rightarrow E'_{1}\rightarrow E_{2}\rightarrow E'_{3}\rightarrow 0,$$where $[E'_{1}]=(\beta_{l},1)$ and $[E'_{3}]=(\beta_{k},1)$. 
\end{lemma}
\begin{proof}. We prove by contradiction. Assume $E_{2}$ fits in both exact sequences. We obtain a map between the two sequences
 \begin{equation}
\begin{tikzpicture}
back line/.style={densely dotted}, 
cross line/.style={preaction={draw=white, -, 
line width=6pt}}] 
\matrix (m) [matrix of math nodes, 
row sep=1em, column sep=2.em, 
text height=1.5ex, 
text depth=0.25ex]{ 
0&E_{1}&E_{2}&E_{3}&0\\
0&E'_{1}&E_{2}&E'_{3}&0\\};
\path[->]
(m-1-1) edge  (m-1-2)
(m-1-2) edge node [above] {$\iota$} (m-1-3)
(m-1-3) edge (m-1-4) edge node [right] {$\cong$} (m-2-3)
(m-1-4) edge (m-1-5)
(m-2-1) edge (m-2-2)
(m-2-2) edge (m-2-3)
(m-2-3) edge node [below] {$p$} (m-2-4)
(m-2-4) edge (m-2-5);
\end{tikzpicture}.
\end{equation}  
Similar to before, we obtain a map $p\circ \iota:E_{1}\rightarrow E'_{3}$. Since $E_{1}$ and $E'_{3}$ have equal classes, we need to consider two possibilities. First when $E_{1}\cong E'_{3}$ and second when $E_{1}\ncong E'_{3}$.\\
If $E_{1}\cong E'_{3}$, then the image of the map $p\circ \iota$ is either multiple of identity over $E_{1}$ or the zero map. If the former case happens it means that the exact sequence on the first row is split, contradicting the assumption that $E_{2}$ fits in a non-split exact sequence. If the map is given by the zero map, then we can apply the argument in Lemma \ref{non-isom} and obtain a contradiction. If $E_{1}\ncong E'_{3}$ then the map $p\circ \iota$ is the zero map and the proof similarly reduces to argument in proof of Lemma \ref{non-isom}.\end{proof}
\section{$GL_{2}(\mathbb{C})$-invariant stratification}
Now we are ready to introduce a $GL_{2}(\mathbb{C})$-invariant stratification of $\mathfrak{M}^{(\beta,2)}_{\mathcal{B}_{p},ss}(\tilde{\tau})$ for $\beta$ satisfying the condition in Assumption \ref{big-assumption}. Note that by Theorem \ref{theorem81} and Definition \ref{stable-locus}$$\mathfrak{M}^{(\beta,2)}_{ss,\mathcal{B}^{\textbf{R}}_{p}}=\left[\frac{\mathfrak{S}_{ss}^{(\beta,2)}}{GL(V)}\right]\,\,\,\, \text{and}\,\,\,\,\mathfrak{M}^{(\beta,2)}_{\mathcal{B}_{p},ss}(\tilde{\tau})\subset \mathfrak{M}^{(\beta,2)}_{\mathcal{B}_{p},ss}=\bigg[\frac{\mathfrak{M}^{(\beta,2)}_{ss,\mathcal{B}^{\textbf{R}}_{p}}}{GL_{2}(\mathbb{C})}\bigg].$$which shows us that a suitable stratification of $\mathfrak{S}_{ss}^{(\beta,2)}$, after passing to the subsequent quotient stacks and restricting to the $\tilde{\tau}$-semistable locus, induces a $GL_{2}(\mathbb{C})$-invariant stratification of $\mathfrak{M}^{(\beta,2)}_{\mathcal{B}_{p},ss}(\tilde{\tau})$. The cartoon below explains the intuition behind our strategy:

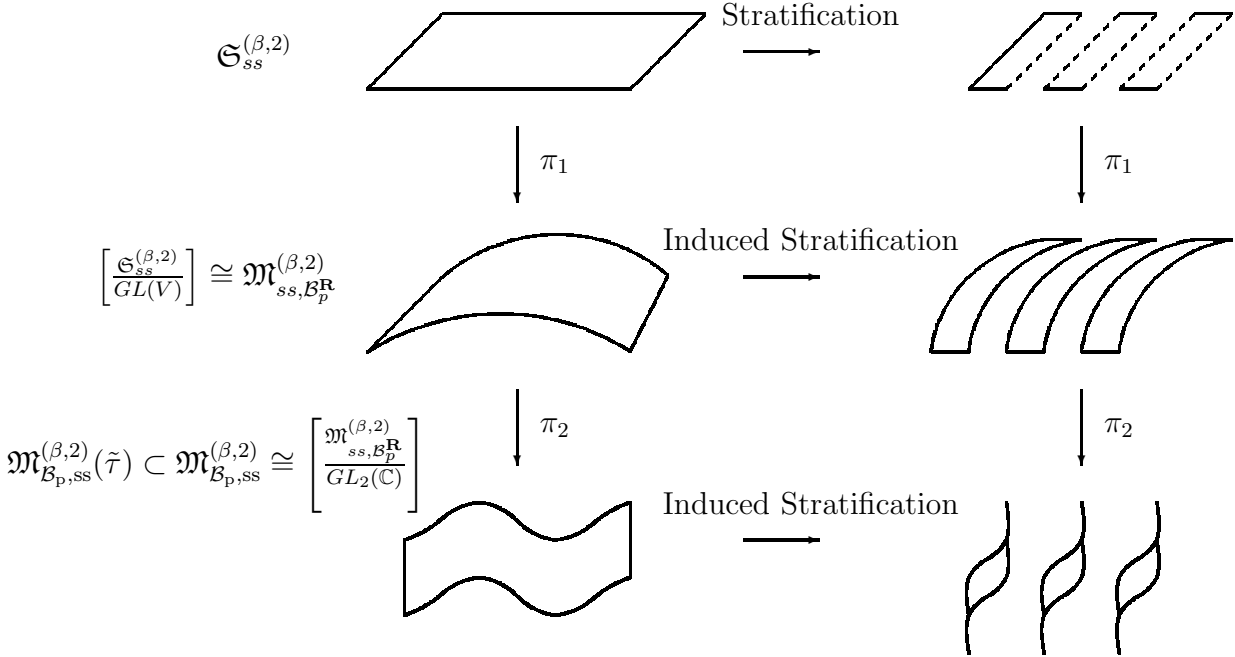
\ifx\JPicScale\undefined\def\JPicScale{1}\fi
\unitlength \JPicScale mm
\begin{picture}(165.18,85)(-10,0)
\linethickness{0.3mm}
\multiput(30,75)(0.12,0.12){83}{\line(1,0){0.12}}
\linethickness{0.3mm}
\put(30,75){\line(1,0){35}}
\linethickness{0.3mm}
\multiput(65,75)(0.12,0.12){83}{\line(1,0){0.12}}
\linethickness{0.3mm}
\put(40,85){\line(1,0){35}}
\linethickness{0.3mm}
\put(80,80){\line(1,0){10}}
\put(90,80){\vector(1,0){0.12}}
\linethickness{0.3mm}
\multiput(110,75)(0.12,0.12){83}{\line(1,0){0.12}}
\linethickness{0.3mm}
\put(120,85){\line(1,0){5}}
\linethickness{0.3mm}
\multiput(115,75)(1.33,1.33){8}{\multiput(0,0)(0.11,0.11){6}{\line(1,0){0.11}}}
\linethickness{0.3mm}
\put(110,75){\line(1,0){5}}
\linethickness{0.3mm}
\multiput(120,75)(1.33,1.33){8}{\multiput(0,0)(0.11,0.11){6}{\line(1,0){0.11}}}
\linethickness{0.3mm}
\put(130,85){\line(1,0){5}}
\linethickness{0.3mm}
\multiput(125,75)(1.33,1.33){8}{\multiput(0,0)(0.11,0.11){6}{\line(1,0){0.11}}}
\linethickness{0.3mm}
\put(120,75){\line(1,0){5}}
\linethickness{0.3mm}
\multiput(130,75)(1.33,1.33){8}{\multiput(0,0)(0.11,0.11){6}{\line(1,0){0.11}}}
\linethickness{0.3mm}
\put(140,85){\line(1,0){5}}
\linethickness{0.3mm}
\multiput(135,75)(1.33,1.33){8}{\multiput(0,0)(0.11,0.11){6}{\line(1,0){0.11}}}
\linethickness{0.3mm}
\put(130,75){\line(1,0){5}}
\linethickness{0.3mm}
\put(50,60){\line(0,1){10}}
\put(50,60){\vector(0,-1){0.12}}
\linethickness{0.3mm}
\multiput(64.58,40.26)(0.21,-0.13){2}{\line(1,0){0.21}}
\multiput(64.15,40.52)(0.21,-0.13){2}{\line(1,0){0.21}}
\multiput(63.72,40.77)(0.22,-0.13){2}{\line(1,0){0.22}}
\multiput(63.28,41.02)(0.22,-0.12){2}{\line(1,0){0.22}}
\multiput(62.84,41.26)(0.22,-0.12){2}{\line(1,0){0.22}}
\multiput(62.4,41.49)(0.22,-0.12){2}{\line(1,0){0.22}}
\multiput(61.96,41.71)(0.22,-0.11){2}{\line(1,0){0.22}}
\multiput(61.51,41.93)(0.22,-0.11){2}{\line(1,0){0.22}}
\multiput(61.06,42.14)(0.23,-0.11){2}{\line(1,0){0.23}}
\multiput(60.6,42.35)(0.23,-0.1){2}{\line(1,0){0.23}}
\multiput(60.14,42.55)(0.23,-0.1){2}{\line(1,0){0.23}}
\multiput(59.68,42.74)(0.23,-0.1){2}{\line(1,0){0.23}}
\multiput(59.21,42.92)(0.23,-0.09){2}{\line(1,0){0.23}}
\multiput(58.75,43.09)(0.47,-0.18){1}{\line(1,0){0.47}}
\multiput(58.28,43.26)(0.47,-0.17){1}{\line(1,0){0.47}}
\multiput(57.8,43.42)(0.47,-0.16){1}{\line(1,0){0.47}}
\multiput(57.33,43.58)(0.47,-0.15){1}{\line(1,0){0.47}}
\multiput(56.85,43.73)(0.48,-0.15){1}{\line(1,0){0.48}}
\multiput(56.37,43.86)(0.48,-0.14){1}{\line(1,0){0.48}}
\multiput(55.89,44)(0.48,-0.13){1}{\line(1,0){0.48}}
\multiput(55.41,44.12)(0.48,-0.12){1}{\line(1,0){0.48}}
\multiput(54.92,44.24)(0.49,-0.12){1}{\line(1,0){0.49}}
\multiput(54.44,44.35)(0.49,-0.11){1}{\line(1,0){0.49}}
\multiput(53.95,44.45)(0.49,-0.1){1}{\line(1,0){0.49}}
\multiput(53.46,44.55)(0.49,-0.1){1}{\line(1,0){0.49}}
\multiput(52.97,44.63)(0.49,-0.09){1}{\line(1,0){0.49}}
\multiput(52.47,44.71)(0.49,-0.08){1}{\line(1,0){0.49}}
\multiput(51.98,44.79)(0.49,-0.07){1}{\line(1,0){0.49}}
\multiput(51.48,44.85)(0.5,-0.06){1}{\line(1,0){0.5}}
\multiput(50.99,44.91)(0.5,-0.06){1}{\line(1,0){0.5}}
\multiput(50.49,44.96)(0.5,-0.05){1}{\line(1,0){0.5}}
\multiput(49.99,45)(0.5,-0.04){1}{\line(1,0){0.5}}
\multiput(49.5,45.03)(0.5,-0.03){1}{\line(1,0){0.5}}
\multiput(49,45.06)(0.5,-0.03){1}{\line(1,0){0.5}}
\multiput(48.5,45.08)(0.5,-0.02){1}{\line(1,0){0.5}}
\multiput(48,45.09)(0.5,-0.01){1}{\line(1,0){0.5}}
\multiput(47.5,45.1)(0.5,-0){1}{\line(1,0){0.5}}
\multiput(47,45.09)(0.5,0){1}{\line(1,0){0.5}}
\multiput(46.5,45.08)(0.5,0.01){1}{\line(1,0){0.5}}
\multiput(46,45.06)(0.5,0.02){1}{\line(1,0){0.5}}
\multiput(45.5,45.03)(0.5,0.03){1}{\line(1,0){0.5}}
\multiput(45.01,45)(0.5,0.03){1}{\line(1,0){0.5}}
\multiput(44.51,44.96)(0.5,0.04){1}{\line(1,0){0.5}}
\multiput(44.01,44.91)(0.5,0.05){1}{\line(1,0){0.5}}
\multiput(43.52,44.85)(0.5,0.06){1}{\line(1,0){0.5}}
\multiput(43.02,44.79)(0.5,0.06){1}{\line(1,0){0.5}}
\multiput(42.53,44.71)(0.49,0.07){1}{\line(1,0){0.49}}
\multiput(42.03,44.63)(0.49,0.08){1}{\line(1,0){0.49}}
\multiput(41.54,44.55)(0.49,0.09){1}{\line(1,0){0.49}}
\multiput(41.05,44.45)(0.49,0.1){1}{\line(1,0){0.49}}
\multiput(40.56,44.35)(0.49,0.1){1}{\line(1,0){0.49}}
\multiput(40.08,44.24)(0.49,0.11){1}{\line(1,0){0.49}}
\multiput(39.59,44.12)(0.49,0.12){1}{\line(1,0){0.49}}
\multiput(39.11,44)(0.48,0.12){1}{\line(1,0){0.48}}
\multiput(38.63,43.86)(0.48,0.13){1}{\line(1,0){0.48}}
\multiput(38.15,43.73)(0.48,0.14){1}{\line(1,0){0.48}}
\multiput(37.67,43.58)(0.48,0.15){1}{\line(1,0){0.48}}
\multiput(37.2,43.42)(0.47,0.15){1}{\line(1,0){0.47}}
\multiput(36.72,43.26)(0.47,0.16){1}{\line(1,0){0.47}}
\multiput(36.25,43.09)(0.47,0.17){1}{\line(1,0){0.47}}
\multiput(35.79,42.92)(0.47,0.18){1}{\line(1,0){0.47}}
\multiput(35.32,42.74)(0.23,0.09){2}{\line(1,0){0.23}}
\multiput(34.86,42.55)(0.23,0.1){2}{\line(1,0){0.23}}
\multiput(34.4,42.35)(0.23,0.1){2}{\line(1,0){0.23}}
\multiput(33.94,42.14)(0.23,0.1){2}{\line(1,0){0.23}}
\multiput(33.49,41.93)(0.23,0.11){2}{\line(1,0){0.23}}
\multiput(33.04,41.71)(0.22,0.11){2}{\line(1,0){0.22}}
\multiput(32.6,41.49)(0.22,0.11){2}{\line(1,0){0.22}}
\multiput(32.16,41.26)(0.22,0.12){2}{\line(1,0){0.22}}
\multiput(31.72,41.02)(0.22,0.12){2}{\line(1,0){0.22}}
\multiput(31.28,40.77)(0.22,0.12){2}{\line(1,0){0.22}}
\multiput(30.85,40.52)(0.22,0.13){2}{\line(1,0){0.22}}
\multiput(30.42,40.26)(0.21,0.13){2}{\line(1,0){0.21}}
\multiput(30,40)(0.21,0.13){2}{\line(1,0){0.21}}

\linethickness{0.3mm}

\linethickness{0.3mm}
\multiput(69.72,50.42)(0.13,-0.11){3}{\line(1,0){0.13}}
\multiput(69.33,50.73)(0.13,-0.1){3}{\line(1,0){0.13}}
\multiput(68.93,51.03)(0.13,-0.1){3}{\line(1,0){0.13}}
\multiput(68.53,51.32)(0.2,-0.15){2}{\line(1,0){0.2}}
\multiput(68.13,51.61)(0.2,-0.14){2}{\line(1,0){0.2}}
\multiput(67.71,51.89)(0.21,-0.14){2}{\line(1,0){0.21}}
\multiput(67.3,52.15)(0.21,-0.13){2}{\line(1,0){0.21}}
\multiput(66.87,52.41)(0.21,-0.13){2}{\line(1,0){0.21}}
\multiput(66.44,52.66)(0.21,-0.12){2}{\line(1,0){0.21}}
\multiput(66.01,52.9)(0.22,-0.12){2}{\line(1,0){0.22}}
\multiput(65.57,53.13)(0.22,-0.12){2}{\line(1,0){0.22}}
\multiput(65.12,53.35)(0.22,-0.11){2}{\line(1,0){0.22}}
\multiput(64.67,53.57)(0.22,-0.11){2}{\line(1,0){0.22}}
\multiput(64.22,53.77)(0.23,-0.1){2}{\line(1,0){0.23}}
\multiput(63.76,53.96)(0.23,-0.1){2}{\line(1,0){0.23}}
\multiput(63.3,54.14)(0.23,-0.09){2}{\line(1,0){0.23}}
\multiput(62.83,54.31)(0.47,-0.17){1}{\line(1,0){0.47}}
\multiput(62.36,54.47)(0.47,-0.16){1}{\line(1,0){0.47}}
\multiput(61.89,54.63)(0.47,-0.15){1}{\line(1,0){0.47}}
\multiput(61.41,54.77)(0.48,-0.14){1}{\line(1,0){0.48}}
\multiput(60.93,54.9)(0.48,-0.13){1}{\line(1,0){0.48}}
\multiput(60.45,55.02)(0.48,-0.12){1}{\line(1,0){0.48}}
\multiput(59.96,55.13)(0.48,-0.11){1}{\line(1,0){0.48}}
\multiput(59.48,55.23)(0.49,-0.1){1}{\line(1,0){0.49}}
\multiput(58.99,55.32)(0.49,-0.09){1}{\line(1,0){0.49}}
\multiput(58.5,55.4)(0.49,-0.08){1}{\line(1,0){0.49}}
\multiput(58,55.46)(0.49,-0.07){1}{\line(1,0){0.49}}
\multiput(57.51,55.52)(0.49,-0.06){1}{\line(1,0){0.49}}
\multiput(57.02,55.57)(0.49,-0.05){1}{\line(1,0){0.49}}
\multiput(56.52,55.6)(0.5,-0.04){1}{\line(1,0){0.5}}
\multiput(56.02,55.63)(0.5,-0.03){1}{\line(1,0){0.5}}
\multiput(55.53,55.64)(0.5,-0.01){1}{\line(1,0){0.5}}
\multiput(55.03,55.65)(0.5,-0){1}{\line(1,0){0.5}}
\multiput(54.53,55.64)(0.5,0.01){1}{\line(1,0){0.5}}
\multiput(54.04,55.62)(0.5,0.02){1}{\line(1,0){0.5}}
\multiput(53.54,55.59)(0.5,0.03){1}{\line(1,0){0.5}}
\multiput(53.04,55.56)(0.5,0.04){1}{\line(1,0){0.5}}
\multiput(52.55,55.51)(0.49,0.05){1}{\line(1,0){0.49}}
\multiput(52.06,55.44)(0.49,0.06){1}{\line(1,0){0.49}}
\multiput(51.57,55.37)(0.49,0.07){1}{\line(1,0){0.49}}
\multiput(51.08,55.29)(0.49,0.08){1}{\line(1,0){0.49}}
\multiput(50.59,55.2)(0.49,0.09){1}{\line(1,0){0.49}}
\multiput(50.1,55.1)(0.49,0.1){1}{\line(1,0){0.49}}
\multiput(49.62,54.98)(0.48,0.11){1}{\line(1,0){0.48}}
\multiput(49.14,54.86)(0.48,0.12){1}{\line(1,0){0.48}}
\multiput(48.66,54.73)(0.48,0.13){1}{\line(1,0){0.48}}
\multiput(48.18,54.58)(0.48,0.14){1}{\line(1,0){0.48}}
\multiput(47.71,54.43)(0.47,0.15){1}{\line(1,0){0.47}}
\multiput(47.24,54.26)(0.47,0.16){1}{\line(1,0){0.47}}
\multiput(46.77,54.09)(0.47,0.17){1}{\line(1,0){0.47}}
\multiput(46.31,53.9)(0.23,0.09){2}{\line(1,0){0.23}}
\multiput(45.86,53.71)(0.23,0.1){2}{\line(1,0){0.23}}
\multiput(45.4,53.5)(0.23,0.1){2}{\line(1,0){0.23}}
\multiput(44.95,53.29)(0.22,0.11){2}{\line(1,0){0.22}}
\multiput(44.51,53.07)(0.22,0.11){2}{\line(1,0){0.22}}
\multiput(44.07,52.83)(0.22,0.12){2}{\line(1,0){0.22}}
\multiput(43.64,52.59)(0.22,0.12){2}{\line(1,0){0.22}}
\multiput(43.21,52.34)(0.21,0.13){2}{\line(1,0){0.21}}
\multiput(42.79,52.08)(0.21,0.13){2}{\line(1,0){0.21}}
\multiput(42.37,51.81)(0.21,0.14){2}{\line(1,0){0.21}}
\multiput(41.96,51.53)(0.21,0.14){2}{\line(1,0){0.21}}
\multiput(41.55,51.24)(0.2,0.14){2}{\line(1,0){0.2}}
\multiput(41.16,50.94)(0.2,0.15){2}{\line(1,0){0.2}}
\multiput(40.76,50.64)(0.13,0.1){3}{\line(1,0){0.13}}
\multiput(40.38,50.32)(0.13,0.1){3}{\line(1,0){0.13}}
\multiput(40,50)(0.13,0.11){3}{\line(1,0){0.13}}

\linethickness{0.3mm}

\linethickness{0.3mm}
\put(80,50){\line(1,0){10}}
\put(90,50){\vector(1,0){0.12}}
\linethickness{0.3mm}
\multiput(119.51,54.92)(0.49,0.08){1}{\line(1,0){0.49}}
\multiput(119.03,54.83)(0.49,0.09){1}{\line(1,0){0.49}}
\multiput(118.54,54.73)(0.48,0.1){1}{\line(1,0){0.48}}
\multiput(118.06,54.61)(0.48,0.12){1}{\line(1,0){0.48}}
\multiput(117.58,54.48)(0.48,0.13){1}{\line(1,0){0.48}}
\multiput(117.11,54.34)(0.47,0.14){1}{\line(1,0){0.47}}
\multiput(116.64,54.18)(0.47,0.16){1}{\line(1,0){0.47}}
\multiput(116.18,54.01)(0.46,0.17){1}{\line(1,0){0.46}}
\multiput(115.72,53.83)(0.23,0.09){2}{\line(1,0){0.23}}
\multiput(115.26,53.63)(0.23,0.1){2}{\line(1,0){0.23}}
\multiput(114.82,53.42)(0.22,0.1){2}{\line(1,0){0.22}}
\multiput(114.37,53.2)(0.22,0.11){2}{\line(1,0){0.22}}
\multiput(113.94,52.97)(0.22,0.12){2}{\line(1,0){0.22}}
\multiput(113.51,52.72)(0.21,0.12){2}{\line(1,0){0.21}}
\multiput(113.09,52.46)(0.21,0.13){2}{\line(1,0){0.21}}
\multiput(112.67,52.19)(0.21,0.13){2}{\line(1,0){0.21}}
\multiput(112.26,51.91)(0.2,0.14){2}{\line(1,0){0.2}}
\multiput(111.87,51.62)(0.2,0.15){2}{\line(1,0){0.2}}
\multiput(111.47,51.32)(0.13,0.1){3}{\line(1,0){0.13}}
\multiput(111.09,51)(0.13,0.1){3}{\line(1,0){0.13}}
\multiput(110.72,50.68)(0.12,0.11){3}{\line(1,0){0.12}}
\multiput(110.35,50.34)(0.12,0.11){3}{\line(1,0){0.12}}
\multiput(110,50)(0.12,0.11){3}{\line(1,0){0.12}}
\multiput(109.66,49.65)(0.11,0.12){3}{\line(0,1){0.12}}
\multiput(109.32,49.28)(0.11,0.12){3}{\line(0,1){0.12}}
\multiput(109,48.91)(0.11,0.12){3}{\line(0,1){0.12}}
\multiput(108.68,48.53)(0.1,0.13){3}{\line(0,1){0.13}}
\multiput(108.38,48.13)(0.1,0.13){3}{\line(0,1){0.13}}
\multiput(108.09,47.74)(0.15,0.2){2}{\line(0,1){0.2}}
\multiput(107.81,47.33)(0.14,0.2){2}{\line(0,1){0.2}}
\multiput(107.54,46.91)(0.13,0.21){2}{\line(0,1){0.21}}
\multiput(107.28,46.49)(0.13,0.21){2}{\line(0,1){0.21}}
\multiput(107.03,46.06)(0.12,0.21){2}{\line(0,1){0.21}}
\multiput(106.8,45.63)(0.12,0.22){2}{\line(0,1){0.22}}
\multiput(106.58,45.18)(0.11,0.22){2}{\line(0,1){0.22}}
\multiput(106.37,44.74)(0.1,0.22){2}{\line(0,1){0.22}}
\multiput(106.17,44.28)(0.1,0.23){2}{\line(0,1){0.23}}
\multiput(105.99,43.82)(0.09,0.23){2}{\line(0,1){0.23}}
\multiput(105.82,43.36)(0.17,0.46){1}{\line(0,1){0.46}}
\multiput(105.66,42.89)(0.16,0.47){1}{\line(0,1){0.47}}
\multiput(105.52,42.42)(0.14,0.47){1}{\line(0,1){0.47}}
\multiput(105.39,41.94)(0.13,0.48){1}{\line(0,1){0.48}}
\multiput(105.27,41.46)(0.12,0.48){1}{\line(0,1){0.48}}
\multiput(105.17,40.97)(0.1,0.48){1}{\line(0,1){0.48}}
\multiput(105.08,40.49)(0.09,0.49){1}{\line(0,1){0.49}}
\multiput(105,40)(0.08,0.49){1}{\line(0,1){0.49}}

\linethickness{0.3mm}
\multiput(124.51,54.92)(0.49,0.08){1}{\line(1,0){0.49}}
\multiput(124.03,54.83)(0.49,0.09){1}{\line(1,0){0.49}}
\multiput(123.54,54.73)(0.48,0.1){1}{\line(1,0){0.48}}
\multiput(123.06,54.61)(0.48,0.12){1}{\line(1,0){0.48}}
\multiput(122.58,54.48)(0.48,0.13){1}{\line(1,0){0.48}}
\multiput(122.11,54.34)(0.47,0.14){1}{\line(1,0){0.47}}
\multiput(121.64,54.18)(0.47,0.16){1}{\line(1,0){0.47}}
\multiput(121.18,54.01)(0.46,0.17){1}{\line(1,0){0.46}}
\multiput(120.72,53.83)(0.23,0.09){2}{\line(1,0){0.23}}
\multiput(120.26,53.63)(0.23,0.1){2}{\line(1,0){0.23}}
\multiput(119.82,53.42)(0.22,0.1){2}{\line(1,0){0.22}}
\multiput(119.37,53.2)(0.22,0.11){2}{\line(1,0){0.22}}
\multiput(118.94,52.97)(0.22,0.12){2}{\line(1,0){0.22}}
\multiput(118.51,52.72)(0.21,0.12){2}{\line(1,0){0.21}}
\multiput(118.09,52.46)(0.21,0.13){2}{\line(1,0){0.21}}
\multiput(117.67,52.19)(0.21,0.13){2}{\line(1,0){0.21}}
\multiput(117.26,51.91)(0.2,0.14){2}{\line(1,0){0.2}}
\multiput(116.87,51.62)(0.2,0.15){2}{\line(1,0){0.2}}
\multiput(116.47,51.32)(0.13,0.1){3}{\line(1,0){0.13}}
\multiput(116.09,51)(0.13,0.1){3}{\line(1,0){0.13}}
\multiput(115.72,50.68)(0.12,0.11){3}{\line(1,0){0.12}}
\multiput(115.35,50.34)(0.12,0.11){3}{\line(1,0){0.12}}
\multiput(115,50)(0.12,0.11){3}{\line(1,0){0.12}}
\multiput(114.66,49.65)(0.11,0.12){3}{\line(0,1){0.12}}
\multiput(114.32,49.28)(0.11,0.12){3}{\line(0,1){0.12}}
\multiput(114,48.91)(0.11,0.12){3}{\line(0,1){0.12}}
\multiput(113.68,48.53)(0.1,0.13){3}{\line(0,1){0.13}}
\multiput(113.38,48.13)(0.1,0.13){3}{\line(0,1){0.13}}
\multiput(113.09,47.74)(0.15,0.2){2}{\line(0,1){0.2}}
\multiput(112.81,47.33)(0.14,0.2){2}{\line(0,1){0.2}}
\multiput(112.54,46.91)(0.13,0.21){2}{\line(0,1){0.21}}
\multiput(112.28,46.49)(0.13,0.21){2}{\line(0,1){0.21}}
\multiput(112.03,46.06)(0.12,0.21){2}{\line(0,1){0.21}}
\multiput(111.8,45.63)(0.12,0.22){2}{\line(0,1){0.22}}
\multiput(111.58,45.18)(0.11,0.22){2}{\line(0,1){0.22}}
\multiput(111.37,44.74)(0.1,0.22){2}{\line(0,1){0.22}}
\multiput(111.17,44.28)(0.1,0.23){2}{\line(0,1){0.23}}
\multiput(110.99,43.82)(0.09,0.23){2}{\line(0,1){0.23}}
\multiput(110.82,43.36)(0.17,0.46){1}{\line(0,1){0.46}}
\multiput(110.66,42.89)(0.16,0.47){1}{\line(0,1){0.47}}
\multiput(110.52,42.42)(0.14,0.47){1}{\line(0,1){0.47}}
\multiput(110.39,41.94)(0.13,0.48){1}{\line(0,1){0.48}}
\multiput(110.27,41.46)(0.12,0.48){1}{\line(0,1){0.48}}
\multiput(110.17,40.97)(0.1,0.48){1}{\line(0,1){0.48}}
\multiput(110.08,40.49)(0.09,0.49){1}{\line(0,1){0.49}}
\multiput(110,40)(0.08,0.49){1}{\line(0,1){0.49}}

\linethickness{0.3mm}
\put(105,40){\line(1,0){5}}
\linethickness{0.3mm}
\put(120,55){\line(1,0){5}}
\linethickness{0.3mm}
\multiput(129.51,54.92)(0.49,0.08){1}{\line(1,0){0.49}}
\multiput(129.03,54.83)(0.49,0.09){1}{\line(1,0){0.49}}
\multiput(128.54,54.73)(0.48,0.1){1}{\line(1,0){0.48}}
\multiput(128.06,54.61)(0.48,0.12){1}{\line(1,0){0.48}}
\multiput(127.58,54.48)(0.48,0.13){1}{\line(1,0){0.48}}
\multiput(127.11,54.34)(0.47,0.14){1}{\line(1,0){0.47}}
\multiput(126.64,54.18)(0.47,0.16){1}{\line(1,0){0.47}}
\multiput(126.18,54.01)(0.46,0.17){1}{\line(1,0){0.46}}
\multiput(125.72,53.83)(0.23,0.09){2}{\line(1,0){0.23}}
\multiput(125.26,53.63)(0.23,0.1){2}{\line(1,0){0.23}}
\multiput(124.82,53.42)(0.22,0.1){2}{\line(1,0){0.22}}
\multiput(124.37,53.2)(0.22,0.11){2}{\line(1,0){0.22}}
\multiput(123.94,52.97)(0.22,0.12){2}{\line(1,0){0.22}}
\multiput(123.51,52.72)(0.21,0.12){2}{\line(1,0){0.21}}
\multiput(123.09,52.46)(0.21,0.13){2}{\line(1,0){0.21}}
\multiput(122.67,52.19)(0.21,0.13){2}{\line(1,0){0.21}}
\multiput(122.26,51.91)(0.2,0.14){2}{\line(1,0){0.2}}
\multiput(121.87,51.62)(0.2,0.15){2}{\line(1,0){0.2}}
\multiput(121.47,51.32)(0.13,0.1){3}{\line(1,0){0.13}}
\multiput(121.09,51)(0.13,0.1){3}{\line(1,0){0.13}}
\multiput(120.72,50.68)(0.12,0.11){3}{\line(1,0){0.12}}
\multiput(120.35,50.34)(0.12,0.11){3}{\line(1,0){0.12}}
\multiput(120,50)(0.12,0.11){3}{\line(1,0){0.12}}
\multiput(119.66,49.65)(0.11,0.12){3}{\line(0,1){0.12}}
\multiput(119.32,49.28)(0.11,0.12){3}{\line(0,1){0.12}}
\multiput(119,48.91)(0.11,0.12){3}{\line(0,1){0.12}}
\multiput(118.68,48.53)(0.1,0.13){3}{\line(0,1){0.13}}
\multiput(118.38,48.13)(0.1,0.13){3}{\line(0,1){0.13}}
\multiput(118.09,47.74)(0.15,0.2){2}{\line(0,1){0.2}}
\multiput(117.81,47.33)(0.14,0.2){2}{\line(0,1){0.2}}
\multiput(117.54,46.91)(0.13,0.21){2}{\line(0,1){0.21}}
\multiput(117.28,46.49)(0.13,0.21){2}{\line(0,1){0.21}}
\multiput(117.03,46.06)(0.12,0.21){2}{\line(0,1){0.21}}
\multiput(116.8,45.63)(0.12,0.22){2}{\line(0,1){0.22}}
\multiput(116.58,45.18)(0.11,0.22){2}{\line(0,1){0.22}}
\multiput(116.37,44.74)(0.1,0.22){2}{\line(0,1){0.22}}
\multiput(116.17,44.28)(0.1,0.23){2}{\line(0,1){0.23}}
\multiput(115.99,43.82)(0.09,0.23){2}{\line(0,1){0.23}}
\multiput(115.82,43.36)(0.17,0.46){1}{\line(0,1){0.46}}
\multiput(115.66,42.89)(0.16,0.47){1}{\line(0,1){0.47}}
\multiput(115.52,42.42)(0.14,0.47){1}{\line(0,1){0.47}}
\multiput(115.39,41.94)(0.13,0.48){1}{\line(0,1){0.48}}
\multiput(115.27,41.46)(0.12,0.48){1}{\line(0,1){0.48}}
\multiput(115.17,40.97)(0.1,0.48){1}{\line(0,1){0.48}}
\multiput(115.08,40.49)(0.09,0.49){1}{\line(0,1){0.49}}
\multiput(115,40)(0.08,0.49){1}{\line(0,1){0.49}}

\linethickness{0.3mm}
\multiput(134.51,54.92)(0.49,0.08){1}{\line(1,0){0.49}}
\multiput(134.03,54.83)(0.49,0.09){1}{\line(1,0){0.49}}
\multiput(133.54,54.73)(0.48,0.1){1}{\line(1,0){0.48}}
\multiput(133.06,54.61)(0.48,0.12){1}{\line(1,0){0.48}}
\multiput(132.58,54.48)(0.48,0.13){1}{\line(1,0){0.48}}
\multiput(132.11,54.34)(0.47,0.14){1}{\line(1,0){0.47}}
\multiput(131.64,54.18)(0.47,0.16){1}{\line(1,0){0.47}}
\multiput(131.18,54.01)(0.46,0.17){1}{\line(1,0){0.46}}
\multiput(130.72,53.83)(0.23,0.09){2}{\line(1,0){0.23}}
\multiput(130.26,53.63)(0.23,0.1){2}{\line(1,0){0.23}}
\multiput(129.82,53.42)(0.22,0.1){2}{\line(1,0){0.22}}
\multiput(129.37,53.2)(0.22,0.11){2}{\line(1,0){0.22}}
\multiput(128.94,52.97)(0.22,0.12){2}{\line(1,0){0.22}}
\multiput(128.51,52.72)(0.21,0.12){2}{\line(1,0){0.21}}
\multiput(128.09,52.46)(0.21,0.13){2}{\line(1,0){0.21}}
\multiput(127.67,52.19)(0.21,0.13){2}{\line(1,0){0.21}}
\multiput(127.26,51.91)(0.2,0.14){2}{\line(1,0){0.2}}
\multiput(126.87,51.62)(0.2,0.15){2}{\line(1,0){0.2}}
\multiput(126.47,51.32)(0.13,0.1){3}{\line(1,0){0.13}}
\multiput(126.09,51)(0.13,0.1){3}{\line(1,0){0.13}}
\multiput(125.72,50.68)(0.12,0.11){3}{\line(1,0){0.12}}
\multiput(125.35,50.34)(0.12,0.11){3}{\line(1,0){0.12}}
\multiput(125,50)(0.12,0.11){3}{\line(1,0){0.12}}
\multiput(124.66,49.65)(0.11,0.12){3}{\line(0,1){0.12}}
\multiput(124.32,49.28)(0.11,0.12){3}{\line(0,1){0.12}}
\multiput(124,48.91)(0.11,0.12){3}{\line(0,1){0.12}}
\multiput(123.68,48.53)(0.1,0.13){3}{\line(0,1){0.13}}
\multiput(123.38,48.13)(0.1,0.13){3}{\line(0,1){0.13}}
\multiput(123.09,47.74)(0.15,0.2){2}{\line(0,1){0.2}}
\multiput(122.81,47.33)(0.14,0.2){2}{\line(0,1){0.2}}
\multiput(122.54,46.91)(0.13,0.21){2}{\line(0,1){0.21}}
\multiput(122.28,46.49)(0.13,0.21){2}{\line(0,1){0.21}}
\multiput(122.03,46.06)(0.12,0.21){2}{\line(0,1){0.21}}
\multiput(121.8,45.63)(0.12,0.22){2}{\line(0,1){0.22}}
\multiput(121.58,45.18)(0.11,0.22){2}{\line(0,1){0.22}}
\multiput(121.37,44.74)(0.1,0.22){2}{\line(0,1){0.22}}
\multiput(121.17,44.28)(0.1,0.23){2}{\line(0,1){0.23}}
\multiput(120.99,43.82)(0.09,0.23){2}{\line(0,1){0.23}}
\multiput(120.82,43.36)(0.17,0.46){1}{\line(0,1){0.46}}
\multiput(120.66,42.89)(0.16,0.47){1}{\line(0,1){0.47}}
\multiput(120.52,42.42)(0.14,0.47){1}{\line(0,1){0.47}}
\multiput(120.39,41.94)(0.13,0.48){1}{\line(0,1){0.48}}
\multiput(120.27,41.46)(0.12,0.48){1}{\line(0,1){0.48}}
\multiput(120.17,40.97)(0.1,0.48){1}{\line(0,1){0.48}}
\multiput(120.08,40.49)(0.09,0.49){1}{\line(0,1){0.49}}
\multiput(120,40)(0.08,0.49){1}{\line(0,1){0.49}}

\linethickness{0.3mm}
\put(115,40){\line(1,0){5}}
\linethickness{0.3mm}
\put(130,55){\line(1,0){5}}
\linethickness{0.3mm}
\multiput(139.51,54.92)(0.49,0.08){1}{\line(1,0){0.49}}
\multiput(139.03,54.83)(0.49,0.09){1}{\line(1,0){0.49}}
\multiput(138.54,54.73)(0.48,0.1){1}{\line(1,0){0.48}}
\multiput(138.06,54.61)(0.48,0.12){1}{\line(1,0){0.48}}
\multiput(137.58,54.48)(0.48,0.13){1}{\line(1,0){0.48}}
\multiput(137.11,54.34)(0.47,0.14){1}{\line(1,0){0.47}}
\multiput(136.64,54.18)(0.47,0.16){1}{\line(1,0){0.47}}
\multiput(136.18,54.01)(0.46,0.17){1}{\line(1,0){0.46}}
\multiput(135.72,53.83)(0.23,0.09){2}{\line(1,0){0.23}}
\multiput(135.26,53.63)(0.23,0.1){2}{\line(1,0){0.23}}
\multiput(134.82,53.42)(0.22,0.1){2}{\line(1,0){0.22}}
\multiput(134.37,53.2)(0.22,0.11){2}{\line(1,0){0.22}}
\multiput(133.94,52.97)(0.22,0.12){2}{\line(1,0){0.22}}
\multiput(133.51,52.72)(0.21,0.12){2}{\line(1,0){0.21}}
\multiput(133.09,52.46)(0.21,0.13){2}{\line(1,0){0.21}}
\multiput(132.67,52.19)(0.21,0.13){2}{\line(1,0){0.21}}
\multiput(132.26,51.91)(0.2,0.14){2}{\line(1,0){0.2}}
\multiput(131.87,51.62)(0.2,0.15){2}{\line(1,0){0.2}}
\multiput(131.47,51.32)(0.13,0.1){3}{\line(1,0){0.13}}
\multiput(131.09,51)(0.13,0.1){3}{\line(1,0){0.13}}
\multiput(130.72,50.68)(0.12,0.11){3}{\line(1,0){0.12}}
\multiput(130.35,50.34)(0.12,0.11){3}{\line(1,0){0.12}}
\multiput(130,50)(0.12,0.11){3}{\line(1,0){0.12}}
\multiput(129.66,49.65)(0.11,0.12){3}{\line(0,1){0.12}}
\multiput(129.32,49.28)(0.11,0.12){3}{\line(0,1){0.12}}
\multiput(129,48.91)(0.11,0.12){3}{\line(0,1){0.12}}
\multiput(128.68,48.53)(0.1,0.13){3}{\line(0,1){0.13}}
\multiput(128.38,48.13)(0.1,0.13){3}{\line(0,1){0.13}}
\multiput(128.09,47.74)(0.15,0.2){2}{\line(0,1){0.2}}
\multiput(127.81,47.33)(0.14,0.2){2}{\line(0,1){0.2}}
\multiput(127.54,46.91)(0.13,0.21){2}{\line(0,1){0.21}}
\multiput(127.28,46.49)(0.13,0.21){2}{\line(0,1){0.21}}
\multiput(127.03,46.06)(0.12,0.21){2}{\line(0,1){0.21}}
\multiput(126.8,45.63)(0.12,0.22){2}{\line(0,1){0.22}}
\multiput(126.58,45.18)(0.11,0.22){2}{\line(0,1){0.22}}
\multiput(126.37,44.74)(0.1,0.22){2}{\line(0,1){0.22}}
\multiput(126.17,44.28)(0.1,0.23){2}{\line(0,1){0.23}}
\multiput(125.99,43.82)(0.09,0.23){2}{\line(0,1){0.23}}
\multiput(125.82,43.36)(0.17,0.46){1}{\line(0,1){0.46}}
\multiput(125.66,42.89)(0.16,0.47){1}{\line(0,1){0.47}}
\multiput(125.52,42.42)(0.14,0.47){1}{\line(0,1){0.47}}
\multiput(125.39,41.94)(0.13,0.48){1}{\line(0,1){0.48}}
\multiput(125.27,41.46)(0.12,0.48){1}{\line(0,1){0.48}}
\multiput(125.17,40.97)(0.1,0.48){1}{\line(0,1){0.48}}
\multiput(125.08,40.49)(0.09,0.49){1}{\line(0,1){0.49}}
\multiput(125,40)(0.08,0.49){1}{\line(0,1){0.49}}

\linethickness{0.3mm}
\multiput(144.51,54.92)(0.49,0.08){1}{\line(1,0){0.49}}
\multiput(144.03,54.83)(0.49,0.09){1}{\line(1,0){0.49}}
\multiput(143.54,54.73)(0.48,0.1){1}{\line(1,0){0.48}}
\multiput(143.06,54.61)(0.48,0.12){1}{\line(1,0){0.48}}
\multiput(142.58,54.48)(0.48,0.13){1}{\line(1,0){0.48}}
\multiput(142.11,54.34)(0.47,0.14){1}{\line(1,0){0.47}}
\multiput(141.64,54.18)(0.47,0.16){1}{\line(1,0){0.47}}
\multiput(141.18,54.01)(0.46,0.17){1}{\line(1,0){0.46}}
\multiput(140.72,53.83)(0.23,0.09){2}{\line(1,0){0.23}}
\multiput(140.26,53.63)(0.23,0.1){2}{\line(1,0){0.23}}
\multiput(139.82,53.42)(0.22,0.1){2}{\line(1,0){0.22}}
\multiput(139.37,53.2)(0.22,0.11){2}{\line(1,0){0.22}}
\multiput(138.94,52.97)(0.22,0.12){2}{\line(1,0){0.22}}
\multiput(138.51,52.72)(0.21,0.12){2}{\line(1,0){0.21}}
\multiput(138.09,52.46)(0.21,0.13){2}{\line(1,0){0.21}}
\multiput(137.67,52.19)(0.21,0.13){2}{\line(1,0){0.21}}
\multiput(137.26,51.91)(0.2,0.14){2}{\line(1,0){0.2}}
\multiput(136.87,51.62)(0.2,0.15){2}{\line(1,0){0.2}}
\multiput(136.47,51.32)(0.13,0.1){3}{\line(1,0){0.13}}
\multiput(136.09,51)(0.13,0.1){3}{\line(1,0){0.13}}
\multiput(135.72,50.68)(0.12,0.11){3}{\line(1,0){0.12}}
\multiput(135.35,50.34)(0.12,0.11){3}{\line(1,0){0.12}}
\multiput(135,50)(0.12,0.11){3}{\line(1,0){0.12}}
\multiput(134.66,49.65)(0.11,0.12){3}{\line(0,1){0.12}}
\multiput(134.32,49.28)(0.11,0.12){3}{\line(0,1){0.12}}
\multiput(134,48.91)(0.11,0.12){3}{\line(0,1){0.12}}
\multiput(133.68,48.53)(0.1,0.13){3}{\line(0,1){0.13}}
\multiput(133.38,48.13)(0.1,0.13){3}{\line(0,1){0.13}}
\multiput(133.09,47.74)(0.15,0.2){2}{\line(0,1){0.2}}
\multiput(132.81,47.33)(0.14,0.2){2}{\line(0,1){0.2}}
\multiput(132.54,46.91)(0.13,0.21){2}{\line(0,1){0.21}}
\multiput(132.28,46.49)(0.13,0.21){2}{\line(0,1){0.21}}
\multiput(132.03,46.06)(0.12,0.21){2}{\line(0,1){0.21}}
\multiput(131.8,45.63)(0.12,0.22){2}{\line(0,1){0.22}}
\multiput(131.58,45.18)(0.11,0.22){2}{\line(0,1){0.22}}
\multiput(131.37,44.74)(0.1,0.22){2}{\line(0,1){0.22}}
\multiput(131.17,44.28)(0.1,0.23){2}{\line(0,1){0.23}}
\multiput(130.99,43.82)(0.09,0.23){2}{\line(0,1){0.23}}
\multiput(130.82,43.36)(0.17,0.46){1}{\line(0,1){0.46}}
\multiput(130.66,42.89)(0.16,0.47){1}{\line(0,1){0.47}}
\multiput(130.52,42.42)(0.14,0.47){1}{\line(0,1){0.47}}
\multiput(130.39,41.94)(0.13,0.48){1}{\line(0,1){0.48}}
\multiput(130.27,41.46)(0.12,0.48){1}{\line(0,1){0.48}}
\multiput(130.17,40.97)(0.1,0.48){1}{\line(0,1){0.48}}
\multiput(130.08,40.49)(0.09,0.49){1}{\line(0,1){0.49}}
\multiput(130,40)(0.08,0.49){1}{\line(0,1){0.49}}

\linethickness{0.3mm}
\put(125,40){\line(1,0){5}}
\linethickness{0.3mm}
\put(140,55){\line(1,0){5}}
\linethickness{0.3mm}
\put(50,25){\line(0,1){10}}
\put(50,25){\vector(0,-1){0.12}}
\linethickness{0.3mm}
\qbezier(35,5)(37.72,5.89)(39.97,7.83)
\qbezier(39.97,7.83)(42.22,9.77)(45,10)
\qbezier(45,10)(47.77,9.69)(50,7.5)
\qbezier(50,7.5)(52.23,5.31)(55,5)
\qbezier(55,5)(57.78,5.23)(60.03,7.17)
\qbezier(60.03,7.17)(62.28,9.11)(65,10)
\linethickness{0.3mm}
\qbezier(35,15)(37.72,15.89)(39.97,17.83)
\qbezier(39.97,17.83)(42.22,19.77)(45,20)
\qbezier(45,20)(47.77,19.69)(50,17.5)
\qbezier(50,17.5)(52.23,15.31)(55,15)
\qbezier(55,15)(57.78,15.23)(60.03,17.17)
\qbezier(60.03,17.17)(62.28,19.11)(65,20)
\linethickness{0.3mm}
\put(35,5){\line(0,1){10}}
\linethickness{0.3mm}
\put(65,10){\line(0,1){10}}
\linethickness{0.3mm}
\put(125,60){\line(0,1){10}}
\put(125,60){\vector(0,-1){0.12}}
\linethickness{0.3mm}
\put(125,25){\line(0,1){10}}
\put(125,25){\vector(0,-1){0.12}}
\linethickness{0.3mm}
\qbezier(110,5)(109.94,5.16)(109.75,6.91)
\qbezier(109.75,6.91)(109.57,8.65)(110,10)
\qbezier(110,10)(110.83,11.57)(112.5,12.5)
\qbezier(112.5,12.5)(114.17,13.43)(115,15)
\qbezier(115,15)(115.43,16.35)(115.25,18.09)
\qbezier(115.25,18.09)(115.06,19.84)(115,20)
%\qbezier(115,20)(115,20)(115,20)
%\qbezier(115,20)(115,20)(115,20)
\linethickness{0.3mm}
\qbezier(110,0)(109.94,0.16)(109.75,1.91)
\qbezier(109.75,1.91)(109.57,3.65)(110,5)
\qbezier(110,5)(110.83,6.57)(112.5,7.5)
\qbezier(112.5,7.5)(114.17,8.43)(115,10)
\qbezier(115,10)(115.43,11.35)(115.25,13.09)
\qbezier(115.25,13.09)(115.06,14.84)(115,15)
%\qbezier(115,15)(115,15)(115,15)
%\qbezier(115,15)(115,15)(115,15)
\linethickness{0.3mm}
\qbezier(120,5)(119.94,5.16)(119.75,6.91)
\qbezier(119.75,6.91)(119.57,8.65)(120,10)
\qbezier(120,10)(120.83,11.57)(122.5,12.5)
\qbezier(122.5,12.5)(124.17,13.43)(125,15)
\qbezier(125,15)(125.43,16.35)(125.25,18.09)
\qbezier(125.25,18.09)(125.06,19.84)(125,20)
%\qbezier(125,20)(125,20)(125,20)
%\qbezier(125,20)(125,20)(125,20)
\linethickness{0.3mm}
\qbezier(120,0)(119.94,0.16)(119.75,1.91)
\qbezier(119.75,1.91)(119.57,3.65)(120,5)
\qbezier(120,5)(120.83,6.57)(122.5,7.5)
\qbezier(122.5,7.5)(124.17,8.43)(125,10)
\qbezier(125,10)(125.43,11.35)(125.25,13.09)
\qbezier(125.25,13.09)(125.06,14.84)(125,15)
%\qbezier(125,15)(125,15)(125,15)
%\qbezier(125,15)(125,15)(125,15)
\linethickness{0.3mm}
\qbezier(130,5)(129.94,5.16)(129.75,6.91)
\qbezier(129.75,6.91)(129.57,8.65)(130,10)
\qbezier(130,10)(130.83,11.57)(132.5,12.5)
\qbezier(132.5,12.5)(134.17,13.43)(135,15)
\qbezier(135,15)(135.43,16.35)(135.25,18.09)
\qbezier(135.25,18.09)(135.06,19.84)(135,20)
%\qbezier(135,20)(135,20)(135,20)
%\qbezier(135,20)(135,20)(135,20)
\linethickness{0.3mm}
\qbezier(130,0)(129.94,0.16)(129.75,1.91)
\qbezier(129.75,1.91)(129.57,3.65)(130,5)
\qbezier(130,5)(130.83,6.57)(132.5,7.5)
\qbezier(132.5,7.5)(134.17,8.43)(135,10)
\qbezier(135,10)(135.43,11.35)(135.25,13.09)
\qbezier(135.25,13.09)(135.06,14.84)(135,15)
%\qbezier(135,15)(135,15)(135,15)
%\qbezier(135,15)(135,15)(135,15)
\put(15,80){\makebox(0,0)[cc]{$\mathfrak{S}^{(\beta,2)}_{ss}$}}

\put(10,50){\makebox(0,0)[cc]{$\left[\frac{\mathfrak{S}^{(\beta,2)}_{ss}}{GL(V)}\right]\cong \mathfrak{M}^{(\beta,2)}_{ss,\mathcal{B}^{\textbf{R}}_{p}}$}}

\put(10,25){\makebox(0,0)[cc]{$\mathfrak{M}^{(\beta,2)}_{\operatorname{\mathcal{B}_{p},ss}}(\tilde{\tau})\subset \mathfrak{M}^{(\beta,2)}_{\operatorname{\mathcal{B}_{p},ss}}\cong \left[\frac{\mathfrak{M}^{(\beta,2)}_{ss,\mathcal{B}^{\textbf{R}}_{p}}}{GL_{2}(\mathbb{C})}\right]$}}

\linethickness{0.3mm}
\multiput(65,40)(0.12,0.24){42}{\line(0,1){0.24}}
\linethickness{0.3mm}
\multiput(30,40)(0.12,0.12){83}{\line(1,0){0.12}}
\linethickness{0.3mm}
\put(80,15){\line(1,0){10}}
\put(90,15){\vector(1,0){0.12}}
\put(89,85){\makebox(0,0)[cc]{Stratification}}

\put(89,55){\makebox(0,0)[cc]{Induced Stratification}}

\put(89,20){\makebox(0,0)[cc]{Induced Stratification}}

\put(55,65){\makebox(0,0)[cc]{$\pi_{1}$}}

\put(55,30){\makebox(0,0)[cc]{$\pi_{2}$}}

\put(130,65){\makebox(0,0)[cc]{$\pi_{1}$}}

\put(130,30){\makebox(0,0)[cc]{$\pi_{2}$}}

\end{picture}
\section*{Strictly semistable objects in $\mathfrak{M}^{(\beta,2)}_{\operatorname{\mathcal{B}_{p},ss}}(\tilde{\tau})$ and further stratifications}
Now we pass to the subsequent quotients of $\mathfrak{S}^{(\beta,2)}_{ss}$ (the picture above) and study the stratification of $\tilde{\tau}$-semistable loci in $\mathfrak{M}^{(\beta,2)}_{\operatorname{\mathcal{B}_{p},ss}}$. These strata will be induced by the decomposition types of the objects involved. Note that in this section we provide only the description of the disjoint strata involved in our calculation and we assume that they are locally closed. The proof of locally closedness property of these strata is somewhat technical and so the interested reader may find the details in the Appendix \ref{appA}. 
\begin{defn}
Define $\mathfrak{M}^{(\beta,2)}_{st-ss}(\tilde{\tau})$ to be the parametrizing scheme of strictly $\tilde{\tau}$-semistable objects in $\mathcal{B}_{p}$ of type $(\beta,2)$ (for $\beta$ as in Assumption \ref{big-assumption}) which are obtained as an extension of two $\tilde{\tau}$-stable objects of rank 1. In other words an object $E_{2}\in \mathfrak{M}^{(\beta,2)}_{st-ss}(\tilde{\tau})$ with class $(\beta,2)$ fits in an exact sequence
\begin{equation}\label{e1e3}
0\rightarrow E_{1}\rightarrow E_{2}\rightarrow E_{3}\rightarrow 0,
\end{equation}
where $E_{1}$ and $E_{3}$ are $\tilde{\tau}$-stable objects with classes $(\beta_{k},1)$ and $(\beta_{l},1)$ respectively for some $\beta_{k}$ and $\beta_{l}$ such that $\beta_{k}+\beta_{l}=\beta$.
\end{defn}
\begin{remark}\label{equal-schemes}
Note that the existence of exact sequence \eqref{strict-semi} and the discussion in Section \ref{compute-example} shows that all objects in $\mathcal{B}_{p}$ of type $(\beta,2)$ are strictly $\tilde{\tau}$-semistable. Hence $\mathfrak{M}^{(\beta,2)}_{st-ss}(\tilde{\tau})\cong \mathfrak{M}^{(\beta,2)}_{\operatorname{\mathcal{B}_{p},ss}}(\tilde{\tau})$.
\end{remark}
If the extension in \eqref{e1e3} is non-split, then its automorphism group is obtained by $\operatorname{Hom}(E_{3},E_{1})\rtimes \mathbb{G}_{m}$ and if split, the automorphism group is obtained by $\operatorname{Hom}(E_{3},E_{1})\rtimes \mathbb{G}^{2}_{m}$ \cite[page 33]{a30}. We need these automorphism groups in order to compute the product ($*$) of the elements of the Ringe-Hall algebra. These elements are given as stack functions which parametrize objects of a given type (such as $(\beta_{k},1)$ or $(\beta_{l},1)$). Now assume that the exact sequence in \eqref{e1e3} is non-split and moreover $E_{1}\cong E_{3}$. In this case since a semistable rank 1 object in $\mathfrak{M}^{(\beta,1)}_{st-ss}(\tilde{\tau})$ is also stable, by the property of $\beta$ in Assumption \ref{big-assumption} and Lemma \ref{hom-eval}, $\operatorname{Hom}(E_{3},E_{1})\cong \mathbb{A}^{1}$ and the automorphism group of extension \eqref{e1e3} is obtained by $\mathbb{A}^{1}\rtimes \mathbb{G}_{m}$. Moreover if in \eqref{e1e3} $E_{1}\ncong E_{3}$, then by Lemma \ref{hom-eval}, $\operatorname{Hom}(E_{3},E_{1})=0$ and the automorphism group of extension \eqref{e1e3} is obtained by $\mathbb{G}_{m}$.

Following similar argument for case of split extensions, we find that the automorphism group  of the split extension is $\mathbb{A}^{1}\rtimes \mathbb{G}^{2}_{m}$ when $E_{1}\cong E_{3}$ and $\mathbb{G}^{2}_{m}$ when $E_{1}\ncong E_{3}$. Therefore, first we decompose $\mathfrak{M}^{(\beta,2)}_{st-ss}(\tilde{\tau})$ into two disjoint strata:$$\mathfrak{M}^{(\beta,2)}_{st-ss}(\tilde{\tau})=\mathfrak{M}^{(\beta,2)}_{n-sp}(\tilde{\tau})\bigsqcup\mathfrak{M}^{(\beta,2)}_{sp}(\tilde{\tau}).$$Here $\mathfrak{M}^{(\beta,2)}_{n-sp}(\tilde{\tau})$ and $\mathfrak{M}^{(\beta,2)}_{sp}(\tilde{\tau})$ stand for the strata over which the objects representing the elements in $\mathfrak{M}^{(\beta,2)}_{st-ss}(\tilde{\tau})$ are given by non-split and split extensions respectively. These strata are disjoint since an element $E_{2}\in \mathfrak{M}^{(\beta,2)}_{st-ss}(\tilde{\tau})$ can not be given by both split and non-split extensions. Now by lemmas \ref{non-isom} and \ref{isom-nonisom}, $\mathfrak{M}^{(\beta,2)}_{n-sp}(\tilde{\tau})$ can  further be stratified into a disjoint union over all possible $\beta_{k}$ and $\beta_{l}$:$$\mathfrak{M}^{(\beta,2)}_{n-sp}(\tilde{\tau})=\bigsqcup_{\beta_{k}+\beta_{l}=\beta}\mathfrak{M}^{(\beta_{k},\beta_{l},2)}_{n-sp}(\tilde{\tau}),$$where $\mathfrak{M}^{(\beta_{k},\beta_{l},2)}_{n-sp}(\tilde{\tau})$ stands for stratum over which $E_{2}$ is obtained by a non-split extension of $\tilde{\tau}$-stable objects $E_{1}$ by $E_{3}$ with fixed classes $(\beta_{k},1)$ and $(\beta_{l},1)$ respectively. Now take one of these strata $\mathfrak{M}^{(\beta_{k},\beta_{l},2)}_{n-sp}(\tilde{\tau})$ by fixing $\beta_{k}$ and $\beta_{l}$. We claim that one may decompose this stratum further into two disjoint strata, depending on the values of $\beta_{k},\beta_{l}$ respectively.
\begin{enumerate}
\item Define $\mathfrak{M}^{(\beta_{k},\beta_{l},2)}_{1}(\tilde{\tau})$ to be the parametrizing stack of objects $E_{2}\in \mathfrak{M}^{(\beta_{k},\beta_{l},2)}_{n-sp}(\tilde{\tau})$ such that there exists a non-split exact sequence 
$$0\rightarrow E_{1}\rightarrow E_{2}\rightarrow E_{3}\rightarrow 0,$$where $E_{1} \in \mathfrak{M}_{\mathcal{B}_{p},ss}^{(\beta_{k},1)}(\tilde{\tau})$ and $E_{3}\in \mathfrak{M}_{\mathcal{B}_{p},ss}^{(\beta_{l},1)}(\tilde{\tau})$ and $E_{1}\ncong E_{3}$.
\item Define $\mathfrak{M}^{(\beta_{k},\beta_{l},2)}_{2}(\tilde{\tau})$ to be the parametrizing stack of objects $E_{2}\in \mathfrak{M}^{(\beta_{k},\beta_{l},2)}_{n-sp}(\tilde{\tau})$ such that there exists a non-split exact sequence$$0\rightarrow E_{1}\rightarrow E_{2}\rightarrow E_{3}\rightarrow 0,$$where $E_{1},E_{3}\in \mathfrak{M}^{(\frac{\beta}{2},1)}_{\mathcal{B}_{p},ss}(\tilde{\tau})$ and $E_{1}\cong E_{3}$.
\end{enumerate}
\begin{lemma}\label{nonsp-ext}
There exists a stratification of $\mathfrak{M}^{(\beta_{k},\beta_{l},2)}_{n-sp}(\tilde{\tau})$:
\begin{align}\label{strata}
&
\mathfrak{M}^{(\beta_{k},\beta_{l},2)}_{n-sp}(\tilde{\tau})=\mathfrak{M}^{(\beta_{k},\beta_{l},2)}_{1}(\tilde{\tau})\bigsqcup\mathfrak{M}^{(\beta_{k},\beta_{l},2)}_{2}(\tilde{\tau}).\notag\\
\end{align}
\end{lemma}
\begin{proof}: We show that $\mathfrak{M}^{(\beta_{k},\beta_{l},2)}_{1}(\tilde{\tau})$ and $\mathfrak{M}^{(\beta_{k},\beta_{l},2)}_{2}(\tilde{\tau})$ are disjoint. Assume that $$\mathfrak{M}^{(\beta_{k},\beta_{l},2)}_{1}(\tilde{\tau})\bigcap \mathfrak{M}^{(\beta_{k},\beta_{l},2)}_{2}(\tilde{\tau})\neq \varnothing$$ Therefore, there exist an object $E_{2}\in\mathfrak{M}^{(\beta_{k},\beta_{l},2)}_{n-sp}(\tilde{\tau})$ which fits in exact sequences:$$0\rightarrow E_{1}\rightarrow E_{2}\rightarrow E_{3}\rightarrow 0,$$and$$0\rightarrow E'_{1}\rightarrow E_{2}\rightarrow E'_{3}\rightarrow 0,$$such that $E_{1}\ncong E_{3}$ and $E'_{1}\cong E'_{3}$. Hence, one obtains a map $p\circ \iota:E_{1}\rightarrow E'_{3}$ via the following diagram: 
\begin{equation}
\begin{tikzpicture}
back line/.style={densely dotted}, 
cross line/.style={preaction={draw=white, -, 
line width=6pt}}] 
\matrix (m) [matrix of math nodes, 
row sep=1em, column sep=2.em, 
text height=1.5ex, 
text depth=0.25ex]{ 
0&E_{1}&E_{2}&E_{3}&0\\
0&E'_{1}&E_{2}&E'_{3}&0\\};
\path[->]
(m-1-1) edge  (m-1-2)
(m-1-2) edge node [above] {$\iota$} (m-1-3)
(m-1-3) edge (m-1-4) edge node [right] {$\cong$} (m-2-3)
(m-1-4) edge (m-1-5)
(m-2-1) edge (m-2-2)
(m-2-2) edge (m-2-3)
(m-2-3) edge node [below] {$p$} (m-2-4)
(m-2-4) edge (m-2-5);
\end{tikzpicture}.
\end{equation}  
If $E_{1}\cong E'_{3}$, then the image of $p\circ \iota$ is multiple of identity over $E_{1}$ or the zero map. For the former case we conclude that the first row splits, hence a contradiction. If $p\circ \iota$ is the zero map, then $p\circ \iota$ factors through the map $\iota'\circ g$ in the following diagram:
\begin{equation}\label{khol}
\begin{tikzpicture}
back line/.style={densely dotted}, 
cross line/.style={preaction={draw=white, -, 
line width=6pt}}] 
\matrix (m) [matrix of math nodes, 
row sep=1em, column sep=2.5em, 
text height=1.5ex, 
text depth=0.25ex]{ 
0&E_{1}&E_{2}&E_{3}&0\\
0&E'_{1}&E_{2}&E'_{3}&0\\};
\path[->]
(m-1-1) edge  (m-1-2) 
(m-1-2) edge node [above] {$\iota$} (m-1-3) edge node [right] {$g$} (m-2-2)
(m-1-3) edge (m-1-4) edge node [right] {$\cong$} (m-2-3)
(m-1-4) edge (m-1-5)
(m-2-1) edge (m-2-2)
(m-2-2) edge node [above] {$\iota'$} (m-2-3)
(m-2-3) edge node [below] {$p$} (m-2-4)
(m-2-4) edge (m-2-5);
\end{tikzpicture}.
\end{equation}  
Since $E_{1}\cong E'_{3}$ and $E'_{3}\cong E'_{1}$ then $E_{1}\cong E'_{1}$ and we conclude one of the following possibilities: either $\iota$ and $\iota'$ are both given as zero maps which is over-ruled (since $\iota$ and $\iota'$ are both injections) or the map $g$ is zero which is obviously overruled. Moreover if the map $g$ is an isomorphism then $E_{3}\cong E'_{3}$ and since $E_{1}\cong E'_{3}$ by assumption, then $E_{1}\cong E_{3}$ which is a contradiction. Therefore, we assume $E_{1}\ncong E'_{3}$.  In this case $p\circ \iota\in\operatorname{Hom}(E_{1},E'_{3})$ which is the zero map by Lemma \ref{hom-eval}. Then similarly $p\circ \iota$ factors through the map $\iota'\circ g$ in the diagram \eqref{khol}. Since $E'_{1}\cong E'_{3}$ and $E_{1}\ncong E'_{3}$, then $E_{1}\ncong E'_{1}$ and by Lemma \ref{hom-eval}, $g$ is the zero map. By considering the left commutative square in \eqref{khol} we obtain a contradiction since the image of $E_{1}$ in $E_{2}$ is nonzero. Therefore $\mathfrak{M}^{(\beta_{k},\beta_{l},2)}_{1}(\tilde{\tau})\bigcap \mathfrak{M}^{(\beta_{k},\beta_{l},2)}_{2}(\tilde{\tau})= \varnothing$.
\end{proof}
So far we have proved $$\mathfrak{M}^{(\beta,2)}_{n-sp}(\tilde{\tau})
=\bigsqcup_{\beta_{k}+\beta_{l}=\beta}\bigg( \mathfrak{M}^{(\beta_{k},\beta_{l},2)}_{1}(\tilde{\tau})\bigsqcup\mathfrak{M}^{(\beta_{k},\beta_{l},2)}_{2}(\tilde{\tau})\bigg).$$Now consider the split stratum $\mathfrak{M}^{(\beta,2)}_{sp}(\tilde{\tau})$. We similarly decompose this stratum further, depending on the values of $\beta_{k},\beta_{l}$.
\begin{enumerate}
\item Define $\mathfrak{M}^{(\beta,2)}_{\Delta}(\tilde{\tau})$ to be the parametrizing stack of objects $E_{2}\in \mathfrak{M}^{(\beta,2)}_{sp}(\tilde{\tau})$ such that $$E_{2}\cong E_{1}\oplus E_{3}$$where $E_{1},E_{3}\in \mathfrak{M}^{(\frac{\beta}{2},1)}_{\mathcal{B}_{p},ss}(\tilde{\tau})$ and $E_{1}\cong E_{3}$.
\item Define $\mathfrak{M}^{(\beta_{k},\beta_{l},2)}_{O-\Delta}(\tilde{\tau})$\footnote{The notation $\Delta$ and $O-\Delta$ stand for diagonal and off-diagonal respectively} to be the parametrizing stack of objects $E_{2}\in \mathfrak{M}^{(\beta,2)}_{sp}(\tilde{\tau})$ such that $$E_{2}\cong E_{1}\oplus E_{3}$$where $E_{1}\in  \mathfrak{M}_{\mathcal{B}_{p},ss}^{(\beta_{k},1)}$, $E_{3}\in  \mathfrak{M}_{\mathcal{B}_{p},ss}^{(\beta_{l},1)}$ and $E_{1}\ncong E_{3}$. %In other words we define $\mathfrak{S}^{(\beta,2)}_{O-\Delta}(\tilde{\tau})$ to be given as
%\begin{equation}\label{sym-kl}
%\mathfrak{M}^{(\beta,2)}_{O-\Delta}(\tilde{\tau})=\bigsqcup_{\beta_{k}+\beta_{l}=\beta}\operatorname{Sym}\bigg(\mathfrak{M}^{(\beta_{l},1)}_{s,\mathcal{B}_{p}}(\tilde{\tau})\times \mathfrak{M}^{(\beta_{l},1)}_{s,\mathcal{B}_{p}}(\tilde{\tau})\bigg)\backslash \mathfrak{M}^{(\beta,2)}_{\Delta}(\tilde{\tau}).
%\end{equation}
\end{enumerate}
By the same argument as in Lemma \ref{nonsp-ext} we see that  $$\mathfrak{M}^{(\beta,2)}_{sp}(\tilde{\tau})=\mathfrak{M}^{(\beta,2)}_{\Delta}(\tilde{\tau})\bigsqcup \mathfrak{M}^{(\beta_{k},\beta_{l},2)}_{O-\Delta}(\tilde{\tau}).$$The above arguments enable us to give a stratification of $\mathfrak{M}^{(\beta,2)}_{st-ss}(\tilde{\tau})$:
\begin{align}\label{st1-4}
&
\mathfrak{M}^{(\beta,2)}_{st-ss}(\tilde{\tau})=
&
\bigsqcup^{o}_{\beta_{k}+\beta_{l}=\beta}\bigg( \mathfrak{M}^{(\beta_{k},\beta_{l},2)}_{1}(\tilde{\tau})\bigsqcup\mathfrak{M}^{(\beta_{k},\beta_{l},2)}_{2}(\tilde{\tau})\bigg)\bigsqcup
\bigsqcup^{u-o}_{\beta_{k}+\beta_{l}=\beta} \left(\mathfrak{M}^{(\beta_{k},\beta_{l},2)}_{O-\Delta}(\tilde{\tau})\bigsqcup \mathfrak{M}^{(\beta,2)}_{\Delta}(\tilde{\tau})\right)\notag\\
\end{align}
Here $\displaystyle{\bigsqcup^{o}_{\beta_{k}+\beta_{l}=\beta}}$ and $\displaystyle{\bigsqcup^{u-o}_{\beta_{k}+\beta_{l}=\beta}}$ stand for ordered and un-ordered disjoint unions respectively. The latter notation makes sense because for non-split extensions it is important which of the two classes $\beta_{k}, \beta_{l}$ appear first (flipping the order of a non-split extension would mean obtaining a different stratum), where as for split extensions flipping the order of appearance of $\beta_{k}, \beta_{l}$ will not make any change.

We introduce a new notation. Fix $\beta_{k}$ and $\beta_{l}$ such that $\beta_{k}+\beta_{l}=\beta$. 
Let $\mathfrak{D}^{(\beta_{k},\beta_{l},2)}(\tilde{\tau})$ be the parametrizing stack of objects $E_{2}\cong E_{1}\oplus E_{3}$ such that $(E_{1},E_{3})\in \mathfrak{M}_{\mathcal{B}_{p},ss}^{(\beta_{k},1)}\times \mathfrak{M}_{\mathcal{B}_{p},ss}^{(\beta_{l},1)}$ and $E_{1}\cong E_{3}$.
%\item $\mathfrak{M}^{(\beta_{k},1)}_{s}(\tilde{\tau})\times \mathfrak{M}^{(\beta_{l},1)}_{s}(\tilde{\tau})\backslash \mathfrak{D}^{(\beta_{k},\beta_{l},2)}=\{(E_{1},E_{3})\in \mathfrak{M}^{(\beta_{k},1)}_{s}(\tilde{\tau})\times \mathfrak{M}^{(\beta_{l},1)}_{s}(\tilde{\tau})\mid E_{1}\ncong E_{3}\}$.
 It is obvious from the definition that if $\beta_{k}\neq \beta_{l}$ then $\mathfrak{D}^{(\beta_{k},\beta_{l},2)}(\tilde{\tau})=\varnothing$ and for $\beta_{k}=\beta_{l}=\frac{\beta}{2}$, $\mathfrak{D}^{(\beta_{k},\beta_{l},2)}(\tilde{\tau})=\mathfrak{M}^{(\beta,2)}_{\Delta}(\tilde{\tau})$.
Using this new notation, $$\mathfrak{M}^{(\beta,2)}_{\Delta}(\tilde{\tau})=\displaystyle{\bigsqcup_{\beta_{k}+\beta_{l}=\beta}}\mathfrak{D}^{(\beta_{k},\beta_{l},2)}(\tilde{\tau})$$and by \eqref{st1-4} we obtain:
 \begin{align}\label{sym-kl34}
 &
 \mathfrak{M}^{(\beta,2)}_{st-ss}(\tilde{\tau})=
\bigsqcup^{o}_{\beta_{k}+\beta_{l}=\beta}\bigg( \mathfrak{M}^{(\beta_{k},\beta_{l},2)}_{1}(\tilde{\tau})\bigsqcup\mathfrak{M}^{(\beta_{k},\beta_{l},2)}_{2}(\tilde{\tau})\bigg)\bigsqcup
\bigsqcup^{u-o}_{\beta_{k}+\beta_{l}=\beta}\bigg(\mathfrak{M}^{(\beta_{k},\beta_{l},2)}_{O-\Delta}(\tilde{\tau})\bigsqcup \mathfrak{D}^{(\beta_{k},\beta_{l},2)}(\tilde{\tau})\bigg)
 \end{align}
\section{The element of the Hall algebra of $\mathfrak{M}^{(\beta,2)}_{\mathcal{B}_{p},ss}(\tilde{\tau})$}

\begin{defn}\label{stackfcn-stss}
Define $\bar{\delta}^{(\beta,2)}_{st-ss}(\tilde{\tau}), \bar{\delta}^{(\beta_{k},\beta_{l},2)}_{1}(\tilde{\tau}), \bar{\delta}^{(\beta_{k},\beta_{l},2)}_{2}(\tilde{\tau}), \bar{\delta}^{(\beta_{k},\beta_{l},2)}_{3}(\tilde{\tau})$ and $\bar{\delta}^{(\beta_{k},\beta_{l},2)}_{4}(\tilde{\tau})$ to be the characteristic stack functions of $\mathfrak{M}^{(\beta,2)}_{st-ss}(\tilde{\tau})\cong \mathfrak{M}^{(\beta,2)}_{\mathcal{B}_{p},ss}(\tilde{\tau}), \mathfrak{M}^{(\beta_{k},\beta_{l},2)}_{1}(\tilde{\tau}), \mathfrak{M}^{(\beta_{k},\beta_{l},2)}_{2}(\tilde{\tau}), \mathfrak{M}^{(\beta_{k},\beta_{l},2)}_{O-\Delta}(\tilde{\tau})$ and $\mathfrak{D}^{(\beta_{k},\beta_{l},2)}(\tilde{\tau})$ respectively.
\end{defn}
Then, by definitions \ref{stackfcn-stss} and Equation \eqref{sym-kl34} the following identity is true:
\begin{equation}\label{cooft}
\bar{\delta}^{(\beta,2)}_{ss}(\tilde{\tau})=\sum_{\beta_{k}+\beta_{l}=\beta}^{o}\bigg(\bar{\delta}^{(\beta_{k},\beta_{l},2)}_{1}(\tilde{\tau})+\bar{\delta}^{(\beta_{k},\beta_{l},2)}_{2}(\tilde{\tau})\bigg)+\sum_{\beta_{k}+\beta_{l}=\beta}^{u-o}\bigg(\bar{\delta}^{(\beta_{k},\beta_{l},2)}_{3}(\tilde{\tau})+\bar{\delta}^{(\beta_{k},\beta_{l},2)}_{4}(\tilde{\tau})\bigg),
\end{equation}
where $\displaystyle{\sum_{\beta_{k}+\beta_{l}=\beta}^{o}}$ and $\displaystyle{\sum_{\beta_{k}+\beta_{l}=\beta}^{u-o}}$ denote the ordered and un-ordered sums respectively. 
However by construction $\bar{\delta}^{(\beta_{k},\beta_{l},2)}_{2}(\tilde{\tau})=\bar{\delta}^{(\beta_{k},\beta_{l},2)}_{4}(\tilde{\tau})=\varnothing$ when ever $\beta_{k}\neq\beta_{l}$, hence we obtain :
\begin{align}\label{chrs}
&
\bar{\delta}^{(\beta,2)}_{ss}(\tilde{\tau})=\sum_{\beta_{k}+\beta_{l}=\beta}^{o}\bigg(\bar{\delta}^{(\beta_{k},\beta_{l},2)}_{1}(\tilde{\tau})\bigg)+\bar{\delta}^{(\frac{\beta}{2},\frac{\beta}{2},2)}_{2}(\tilde{\tau})+\sum_{\beta_{k}+\beta_{l}=\beta}^{u-o}\bigg(\bar{\delta}^{(\beta_{k},\beta_{l},2)}_{3}(\tilde{\tau})\bigg)+\bar{\delta}^{(\frac{\beta}{2},\frac{\beta}{2},2)}_{4}(\tilde{\tau}).\notag\\
\end{align}
On the other hand by applying Definition \ref{decompos} to $\mathfrak{M}^{(\beta,2)}_{\mathcal{B}_{p},ss}(\tilde{\tau})$ we obtain the description of the element of the Hall algebra:
\begin{equation}\label{stratum-decompos}
\bar{\epsilon}^{(\beta,2)}(\tilde{\tau})=\bar{\delta}^{(\beta,2)}_{ss}(\tilde{\tau})-\sum_{\beta_{k}+\beta_{l}=\beta}^{o}\frac{1}{2}\bar{\delta}^{(\beta_{k},1)}_{s}(\tilde{\tau})*\bar{\delta}^{(\beta_{l},1)}_{s}(\tilde{\tau}).
\end{equation}

Now by equations \eqref{chrs} and \eqref{stratum-decompos} we obtain:
\begin{align}\label{stratum-decompos3}
&
\bar{\epsilon}^{(\beta,2)}(\tilde{\tau})=\sum_{\beta_{k}+\beta_{l}=\beta}^{o}\bigg(\bar{\delta}^{(\beta_{k},\beta_{l},2)}_{1}(\tilde{\tau})\bigg)+\bar{\delta}^{(\frac{\beta}{2},\frac{\beta}{2},2)}_{2}(\tilde{\tau})+\sum_{\beta_{k}+\beta_{l}=\beta}^{u-o}\bigg(\bar{\delta}^{(\beta_{k},\beta_{l},2)}_{3}(\tilde{\tau})\bigg)+\bar{\delta}^{(\frac{\beta}{2},\frac{\beta}{2},2)}_{4}(\tilde{\tau})\notag\\
&
-\sum_{\beta_{k}+\beta_{l}=\beta}^{o}\frac{1}{2}\bar{\delta}^{(\beta_{k},1)}_{s}(\tilde{\tau})*\bar{\delta}^{(\beta_{l},1)}_{s}(\tilde{\tau})\notag\\
\end{align}
It is easily seen that: $\sum_{\beta_{k}+\beta_{l}=\beta}^{u-o}\bar{\delta}^{(\beta_{k},\beta_{l},2)}_{3}(\tilde{\tau})=\bar{\delta}^{(\frac{\beta}{2},\frac{\beta}{2},2)}_{3}(\tilde{\tau})+\frac{1}{2}\cdot \sum_{\substack{\beta_{k}+\beta_{l}=\beta\\ \beta_{k}\neq \beta_{l}}}^{o}\bar{\delta}^{(\beta_{k},\beta_{l},2)}_{3}(\tilde{\tau}).$
Hence, we rewrite the right hand side of \eqref{stratum-decompos3} as:
\begin{align}\label{stratum-decompos2}
&
\bar{\epsilon}^{(\beta,2)}(\tilde{\tau})=
\sum_{\beta_{k}+\beta_{l}=\beta}^{o}\bigg(\bar{\delta}^{(\beta_{k},\beta_{l},2)}_{1}(\tilde{\tau})\bigg)+\bar{\delta}^{(\frac{\beta}{2},\frac{\beta}{2},2)}_{2}(\tilde{\tau})+\bar{\delta}^{(\frac{\beta}{2},\frac{\beta}{2},2)}_{3}(\tilde{\tau})+\frac{1}{2}\cdot \sum_{\substack{\beta_{k}+\beta_{l}=\beta\\ \beta_{k}\neq \beta_{l}}}^{o}\bar{\delta}^{(\beta_{k},\beta_{l},2)}_{3}(\tilde{\tau})+\bar{\delta}^{(\frac{\beta}{2},\frac{\beta}{2},2)}_{4}(\tilde{\tau})\notag\\
&
-\sum_{\beta_{k}+\beta_{l}=\beta}^{o}\frac{1}{2}\cdot\bar{\delta}^{(\beta_{k},1)}_{s}(\tilde{\tau})*\bar{\delta}^{(\beta_{l},1)}_{s}(\tilde{\tau}).\notag\\
\end{align}
Next we compute the ordered product $\bar{\delta}^{(\beta_{k},1)}_{s}(\tilde{\tau})*\bar{\delta}^{(\beta_{l},1)}_{s}(\tilde{\tau})$ for a fixed choice of $\beta_{k}$ and $\beta_{l}$. 
\section{Computation of $\bar{\delta}^{(\beta_{k},1)}_{s}(\tilde{\tau})*\bar{\delta}^{(\beta_{l},1)}_{s}(\tilde{\tau})$}\label{point-to-whole}
In this section we describe the computation of the Ringel hall product of the stack functions $\bar{\delta}_{s}^{(\beta_{k},1)}(\tilde{\tau})*\bar{\delta}_{s}^{(\beta_{l},1)}(\tilde{\tau})$ for $\beta_{k}$ and $\beta_{l}$ satisfying the condition in Assumption \ref{big-assumption}. Similar to discussions in Section \ref{compute-example} let $\pi_{i}:\mathfrak{Exact}_{\mathcal{B}_{p}}\rightarrow \mathfrak{M}_{\mathcal{B}_{p}}(\tilde{\tau}) $ for $i=1,2,3$ be the projection map that sends an exact sequence$$0\to E_{1}\to E_{2}\to E_{3}\to 0$$ to its first, second and third objects respectively over moduli stack of objects in $\mathcal{B}_{p}$. We also have the map $\pi_{1}\times \pi_{3}:\mathfrak{Exact}_{\mathcal{B}_{p}}\rightarrow \mathfrak{M}_{\mathcal{B}_{p}}(\tilde{\tau})\times \mathfrak{M}_{\mathcal{B}_{p}}(\tilde{\tau})$. By Joyce's definition in \cite{a30}: 
\begin{equation}\label{javab} 
\delta^{(\beta_{k},1)}_{s}(\tilde{\tau})*\delta^{(\beta_{l},1)}_{s}(\tilde{\tau})=\pi_{2}*((\pi_{1}\times \pi_{3})^{*}(\delta_{s}^{(\beta_{k},1)}(\tilde{\tau})\otimes \delta_{s}^{(\beta_{l},1)}(\tilde{\tau})))
\end{equation}
Suppose that $\delta^{(\beta_{k},1)}_{s}=[\mathcal{M}^{(\beta_{k},1)}(\tilde{\tau})\slash\mathbb{G}_{m},\rho_{1}]$ and $\delta^{(\beta_{l},1)}_{s}=[\mathcal{M}^{(\beta_{k},1)}(\tilde{\tau})\slash\mathbb{G}_{m},\rho_{3}]$ where $\mathcal{M}^{(\beta_{k},1)}(\tilde{\tau})$ and $\mathcal{M}^{(\beta_{k},1)}(\tilde{\tau})$ denote some underlying parameter schemes and $$\rho_{1}:\left[\mathcal{M}^{(\beta_{k},1)}(\tilde{\tau})\slash\mathbb{G}_{m}\right]\rightarrow \mathfrak{M}_{\mathcal{B}_{p}}(\tilde{\tau}),$$and$$\rho_{3}:\left[\mathcal{M}^{(\beta_{l},1)}(\tilde{\tau})\slash\mathbb{G}_{m}\right]\rightarrow \mathfrak{M}_{\mathcal{B}_{p}}(\tilde{\tau}).$$
Let us denote by $\mathcal{Z}'$ the fibered product $$(\left[\mathcal{M}^{(\beta_{k},1)}(\tilde{\tau})\slash\mathbb{G}_{m}\right]\times \left[\mathcal{M}^{(\beta_{l},1)}(\tilde{\tau})\slash\mathbb{G}_{m}\right])\times_{\rho_{1}\times\rho_{3},\mathfrak{M}_{\mathcal{B}_{p}}(\tilde{\tau})\times \mathfrak{M}_{\mathcal{B}_{p}}(\tilde{\tau}),\pi_{1}\times\pi_{3}}\mathfrak{Exact}_{\mathcal{B}_{p}}$$the identity in \eqref{javab} is described by $(\pi_{2}\circ \Phi)_{*}\mathcal{Z}'$ in the following diagram:
\begin{equation}\label{embedding}
\begin{tikzpicture}
back line/.style={densely dotted}, 
cross line/.style={preaction={draw=white, -, 
line width=6pt}}] 
\matrix (m) [matrix of math nodes, 
row sep=3em, column sep=3.25em, 
text height=1.5ex, 
text depth=0.25ex]{ 
\mathcal{Z}'&\mathfrak{Exact}_{\mathcal{B}_{p}}&\mathfrak{M}_{\mathcal{B}_{p}}(\tilde{\tau})\\
\left[\mathcal{M}^{(\beta_{k},1)}(\tilde{\tau})\slash\mathbb{G}_{m}\right]\times \left[\mathcal{M}^{(\beta_{l},1)}(\tilde{\tau})\slash\mathbb{G}_{m}\right]&\mathfrak{M}_{\mathcal{B}_{p}}(\tilde{\tau})\times \mathfrak{M}_{\mathcal{B}_{p}}(\tilde{\tau})&\\};
\path[->]
(m-1-1) edge node [above] {$\Phi$} (m-1-2)
(m-1-2) edge node [above] {$\pi_{2}$} (m-1-3)
(m-1-1) edge (m-2-1)
(m-1-2) edge node [right] {$\pi_{1}\times \pi_{3}$}(m-2-2)
(m-2-1) edge node [above] {$\rho_{1}\times \rho_{1}$} (m-2-2);
\end{tikzpicture}
\end{equation} 
We compute the product of stack functions in \eqref{javab} by computing it over the $\mathbb{C}$-points of $\bar{\delta}_{s}^{(\beta_{k},1)}(\tilde{\tau})$ and $\bar{\delta}_{s}^{(\beta_{l},1)}(\tilde{\tau})$ (these are induced from $\mathbb{C}$-points of $\mathcal{M}^{(\beta_{k},1)}(\tilde{\tau})$ and $\mathcal{M}^{(\beta_{l},1)}(\tilde{\tau})$) and then integrating over all points in $\mathcal{M}^{(\beta_{k},1)}(\tilde{\tau})\times \mathcal{M}^{(\beta_{l},1)}(\tilde{\tau})$.
\subsubsection{Pointwise products}
 Consider the stack function $$\delta_{1}=\bigg(\left[\frac{\operatorname{Spec}(\mathbb{C})}{\mathbb{G}_{m}}\right],\rho_{1}\circ \iota_{1}\bigg),\,\,\,\text{with}\,\,\,\iota_{1}: \left[\frac{\operatorname{Spec}(\mathbb{C})}{\mathbb{G}_{m}}\right]\rightarrow \left[\frac{\mathcal{M}^{(\beta_{k},1)}(\tilde{\tau})}{\mathbb{G}_{m}}\right].$$Moreover let $$\delta_{3}=\bigg(\left[\frac{\operatorname{Spec}(\mathbb{C})}{\mathbb{G}_{m}}\right],\rho_{3}\circ \iota_{3}\bigg),\,\,\,\text{with}\,\,\,\iota_{3}: \left[\frac{\operatorname{Spec}(\mathbb{C})}{\mathbb{G}_{m}}\right]\rightarrow \left[\frac{\mathcal{M}^{(\beta_{l},1)}(\tilde{\tau})}{\mathbb{G}_{m}}\right].$$Note that $\delta_{1}\subset \bar{\delta}_{s}^{(\beta_{k},1)}(\tilde{\tau})$ and $\delta_{3}\subset \bar{\delta}_{s}^{(\beta_{l},1)}(\tilde{\tau})$ are the sub-stack functions, induced by taking the stacky quotients of $\mathbb{C}$-points of $\mathcal{M}^{(\beta_{k},1)}(\tilde{\tau})$ and $\mathcal{M}^{(\beta_{l},1)}(\tilde{\tau})$ respectively. Let $E_{1}\in \mathfrak{M}^{(\beta_{k},1)}_{\mathcal{B}_{p}}(\tilde{\tau})$ and $E_{3}\in \mathfrak{M}^{(\beta_{l},1)}_{\mathcal{B}_{p}}(\tilde{\tau})$ and $E_{2}\in \mathfrak{M}^{(\beta,2)}_{\mathcal{B}_{p}}(\tilde{\tau})$. Consider the exact sequence in $\mathfrak{Exact}_{\mathcal{B}_{p}}$:
\begin{equation}\label{exact11}
0\rightarrow E_{1}\rightarrow E_{2}\rightarrow E_{3}\rightarrow 0 
\end{equation}
The automorphism group of the extension \eqref{exact11} is given by $\operatorname{Hom}(E_{3},E_{1})\rtimes \mathbb{G}_{m}^{2}$. The element $(g_{1},g_{2})\in \mathbb{G}^{2}_{m}$ acts on $\operatorname{Ext}^{1}(E_{3},E_{1})$ by multiplication by $g_{2}^{-1}g_{1}$ and the action of $\operatorname{Hom}(E_{3},E_{1})$ on $\operatorname{Ext}^{1}(E_{3},E_{1})$ is trivial. If the extensions in \eqref{exact11} are non-split, then the parametrizing scheme of such extensions is obtained by $\mathbb{P}(\operatorname{Ext}^{1}(E_{3},E_{1}))$ and for split extensions, it is obtained by $\operatorname{Spec}(\mathbb{C})$. In case of nonsplit extensions, the stabilizer group of the action of $\mathbb{G}_{m}^{2}$ is given by $\mathbb{G}_{m}$ and for split extensions, the stabilizer group of the action of $\mathbb{G}^{2}_{m}$ is $\mathbb{G}^{2}_{m}$ itself, hence: 
\begin{align}\label{deldel}
&
\delta_{1} *\delta_{3}=\left(\left[\frac{\operatorname{Spec}(\mathbb{C})}{\operatorname{Hom}(E_{3},E_{1})\rtimes\mathbb{G}^{2}_{m}}\right],\mu_{1}\right)+\left(\left[\frac{\mathbb{P}(\operatorname{Ext}^{1}(E_{3},E_{1}))}{\operatorname{Hom}(E_{3},E_{1})\rtimes\mathbb{G}_{m}}\right],\mu_{3}\right).\notag\\
\end{align} 
\subsubsection{Motivic integration over points}Now we can integrate the right hand side of \eqref{deldel} over $\mathbb{C}$-points of $\mathcal{M}^{(\beta_{k},1)}(\tilde{\tau})$ and $\mathcal{M}^{(\beta_{l},1)}(\tilde{\tau})$ respectively. Let us define/recall such notion of integration; 
\begin{defn}\label{chichi-migi}
Let $\mathfrak{R}$ be a $\mathbb{C}$-stack given by $\mathfrak{R}=\left[\frac{\mathcal{M}}{G}\right]$. Let $\mathcal{B}G$ denote the quotient stack $\left[\frac{\operatorname{Spec}(\mathbb{C})}{G}\right]$. Now define motivic integration over $\mathfrak{R}$ as an identity in the motivic ring of stack functions:
\begin{equation}
\displaystyle{\int}_{\mathcal{M}}\left[\frac{\operatorname{Spec}(\mathbb{C})}{G}\right]d\mu_{m}:= \left[\frac{\mathfrak{R}}{G}\right].
\end{equation} 
Moreover assume that $\textbf{P}\rightarrow \mathfrak{R}$ is a vector bundle over $\mathfrak{R}$. Then define:
\begin{equation}
\displaystyle{\int}_{\mathcal{M}}\left[\frac{\textbf{P}}{G}\right]d\mu_{m}:=\displaystyle{\int}_{\mathcal{M}}\chi(\textbf{P})\cdot\left[\frac{\operatorname{Spec}(\mathbb{C})}{G}\right]d\mu_{m}:= \chi(\textbf{P})\cdot\left[\frac{\mathfrak{R}}{G}\right],
\end{equation}
where $\chi(\textbf{P})$ denotes the topological Euler characteristic of $\textbf{P}$. Here, the ``\textit{measure}" $\mu_{m}$ is the map sending constructible sets on $\mathcal{M}$ to the their corresponding elements in the Grothendieck group of stacks.
\end{defn}
Now we use Definition \ref{chichi-migi} to ``\textit{motivically}" integrate Equation \eqref{deldel} over the points of $\mathcal{M}^{(\beta_{k},1)}(\tilde{\tau})\times \mathcal{M}^{(\beta_{l},1)}(\tilde{\tau})$:
\begin{align}\label{int0}
&
\bar{\delta}^{(\beta_{k},1)}_{s}(\tilde{\tau})*\bar{\delta}^{(\beta_{l},1)}_{s}(\tilde{\tau})=\displaystyle{\int}_{(E_{1},E_{3})\in\mathcal{M}^{(\beta_{k},1)}(\tilde{\tau})\times \mathcal{M}^{(\beta_{l},1)}(\tilde{\tau})}\delta_{1}*\delta_{3}=\notag\\
&
\displaystyle{\int}_{(E_{1},E_{3})\in\mathcal{M}^{(\beta_{k},1)}(\tilde{\tau})\times \mathcal{M}^{(\beta_{l},1)}(\tilde{\tau})}\left[\frac{\operatorname{Spec}(\mathbb{C})}{\operatorname{Hom}(E_{3},E_{1})\rtimes\mathbb{G}^{2}_{m}}\right]d\mu_{m}\notag\\
&
+\displaystyle{\int}_{(E_{1},E_{3})\in\mathcal{M}^{(\beta_{k},1)}(\tilde{\tau})\times \mathcal{M}^{(\beta_{l},1)}(\tilde{\tau})}\left[\frac{\mathbb{P}(\operatorname{Ext}^{1}(E_{3},E_{1}))}{\operatorname{Hom}(E_{3},E_{1})\rtimes\mathbb{G}_{m}}\right]d\mu_{m}.\notag\\
\end{align}
By the result of Lemma \ref{hom-eval}, if $\iota_{1}=\iota_{3}$, i.e if $E_{1}\cong E_{3}$, then $\operatorname{Hom}(E_{3},E_{1})\cong \mathbb{A}^{1}$ and if $\iota_{1}\neq \iota_{3}$, then $\operatorname{Hom}(E_{3},E_{1})=\operatorname{Spec}(\mathbb{C})$, hence we can evaluate the first summand on the right hand side \eqref{int0} as follows:
\begin{align}\label{int1}
&
\displaystyle{\int}_{(E_{1},E_{3})\in\mathcal{M}^{(\beta_{k},1)}(\tilde{\tau})\times \mathcal{M}^{(\beta_{l},1)}(\tilde{\tau})}\left[\frac{\operatorname{Spec}(\mathbb{C})}{\operatorname{Hom}(E_{3},E_{1})\rtimes\mathbb{G}^{2}_{m}}\right]d\mu_{m}=\notag\\
&
\displaystyle{\int}_{\Delta}\left[\frac{\operatorname{Spec}(\mathbb{C})}{\mathbb{A}^{1}\rtimes\mathbb{G}^{2}_{m}}\right]d\mu_{m}+\displaystyle{\int}_{(E_{1},E_{3})\in\mathcal{M}^{(\beta_{k},1)}(\tilde{\tau})\times \mathcal{M}^{(\beta_{l},1)}(\tilde{\tau})\backslash \Delta}\left[\frac{\operatorname{Spec}(\mathbb{C})}{\mathbb{G}^{2}_{m}}\right]d\mu_{m}\notag\\
&
=\left[\frac{\mathcal{M}^{(\frac{\beta}{2},1)}(\tilde{\tau})}{\mathbb{A}^{1}\rtimes\mathbb{G}^{2}_{m}}\right]+\left[\frac{\mathcal{M}^{(\beta_{k},1)}(\tilde{\tau})\times \mathcal{M}^{(\beta_{l},1)}(\tilde{\tau})\backslash \Delta}{\mathbb{G}^{2}_{m}}\right].\notag\\
\end{align}
Here $\Delta$ is the diagonal in the product $\mathcal{M}^{(\beta_{k},1)}(\tilde{\tau})\times \mathcal{M}^{(\beta_{l},1)}(\tilde{\tau})$. Similarly for the second summand on the right hand side of \eqref{int0} we obtain:
\begin{align}\label{int2}
&
\displaystyle{\int}_{(E_{1},E_{3})\in\mathcal{M}^{(\beta_{k},1)}(\tilde{\tau})\times \mathcal{M}^{(\beta_{l},1)}(\tilde{\tau})}\left[\frac{\mathbb{P}(\operatorname{Ext}^{1}(E_{3},E_{1}))}{\operatorname{Hom}(E_{3},E_{1})\rtimes\mathbb{G}_{m}}\right]d\mu_{m}\notag\\
&
=\displaystyle{\int}_{E_{1}\in\mathcal{M}^{(\frac{\beta}{2},1)}(\tilde{\tau})}\left[\frac{\mathbb{P}(\operatorname{Ext}^{1}(E_{1},E_{1}))}{\mathbb{A}^{1}\rtimes\mathbb{G}_{m}}\right]d\mu_{m}+\displaystyle{\int}_{(E_{1},E_{3})\in\mathcal{M}^{(\beta_{k},1)}(\tilde{\tau})\times \mathcal{M}^{(\beta_{l},1)}(\tilde{\tau})\backslash \Delta}\left[\frac{\mathbb{P}(\operatorname{Ext}^{1}(E_{3},E_{1}))}{\mathbb{G}_{m}}\right]d\mu_{m}.\notag\\
\end{align}
From Equations \eqref{int0}, \eqref{int1} and \eqref{int2} we obtain:
\begin{align}\label{deldel-decompos3}
&
\bar{\delta}^{(\beta_{k},1)}_{s}(\tilde{\tau})*\bar{\delta}^{(\beta_{l},1)}_{s}(\tilde{\tau})=\notag\\
&
\left[\frac{\mathcal{M}^{(\frac{\beta}{2},1)}(\tilde{\tau})}{\mathbb{A}^{1}\rtimes\mathbb{G}^{2}_{m}}\right]+\left[\frac{\mathcal{M}^{(\beta_{k},1)}(\tilde{\tau})\times \mathcal{M}^{(\beta_{l},1)}(\tilde{\tau})\backslash \Delta}{\mathbb{G}^{2}_{m}}\right]+\displaystyle{\int}_{E_{1}\in\mathcal{M}^{(\frac{\beta}{2},1)}}\left[\frac{\mathbb{P}(\operatorname{Ext}^{1}(E_{1},E_{1}))}{\mathbb{A}^{1}\rtimes\mathbb{G}_{m}}\right]d\mu_{m}\notag\\
&
+\displaystyle{\int}_{(E_{1},E_{3})\in\mathcal{M}^{(\beta_{k},1)}(\tilde{\tau})\times \mathcal{M}^{(\beta_{l},1)}(\tilde{\tau})\backslash \Delta}\left[\frac{\mathbb{P}(\operatorname{Ext}^{1}(E_{3},E_{1}))}{\mathbb{G}_{m}}\right]d\mu_{m}.\notag\\
\end{align}
Equation \eqref{deldel-decompos3} above was obtained by fixing the classes $\beta_{k},\beta_{l}$ and adding motivically over the points of the underlying parameterizing schemes, i.e. adding the contribution of all sheaves with the same fixed numerical classes. Now the final step to complete the calculation is to vary the classes $\beta_{k}, \beta_{l}$ as long as $\beta_{k}+\beta_{l}=\beta$:
\begin{align}\label{deldel-decompos4}
&
\sum_{\beta_{k}+\beta_{l}=\beta}^{o}\bar{\delta}^{(\beta_{k},1)}_{s}(\tilde{\tau})*\bar{\delta}^{(\beta_{l},1)}_{s}(\tilde{\tau})=\sum_{\beta_{k}+\beta_{l}=\beta}^{o}\bigg(\displaystyle{\int}_{(E_{1},E_{3})\in\mathcal{M}^{(\beta_{k},1)}(\tilde{\tau})\times \mathcal{M}^{(\beta_{l},1)}(\tilde{\tau})\backslash \Delta}\left[\frac{\mathbb{P}(\operatorname{Ext}^{1}(E_{3},E_{1}))}{\mathbb{G}_{m}}\right]d\mu_{m}\bigg)\notag\\
&
+\displaystyle{\int}_{E_{1}\in\mathcal{M}^{(\frac{\beta}{2},1)}(\tilde{\tau})}\left[\frac{\mathbb{P}(\operatorname{Ext}^{1}(E_{1},E_{1}))}{\mathbb{A}^{1}\rtimes\mathbb{G}_{m}}\right]d\mu_{m}
+\sum_{\beta_{k}+\beta_{l}=\beta}^{o}\left[\frac{\mathcal{M}^{(\beta_{k},1)}(\tilde{\tau})\times \mathcal{M}^{(\beta_{l},1)}(\tilde{\tau})\backslash \Delta}{\mathbb{G}^{2}_{m}}\right]+\left[\frac{\mathcal{M}^{(\frac{\beta}{2},1)}(\tilde{\tau})}{\mathbb{A}^{1}\rtimes\mathbb{G}^{2}_{m}}\right]\notag\\
\end{align}
It is easily seen that
\begin{align}
&
\sum_{\beta_{k}+\beta_{l}=\beta}^{o}\left[\frac{\mathcal{M}^{(\beta_{k},1)}(\tilde{\tau})\times \mathcal{M}^{(\beta_{l},1)}(\tilde{\tau})\backslash \Delta}{\mathbb{G}^{2}_{m}}\right]=\notag\\
&
\,\,\,\,\,\,\,\,\,\,\,\,\,\,\,\,\,\,\,\,\,\,\,\,\,\,\,\,\,\,\,\,\,\,\,\,\,\,\,\,\left[\frac{\mathcal{M}^{(\frac{\beta}{2},1)}(\tilde{\tau})\times \mathcal{M}^{(\frac{\beta}{2},1)}(\tilde{\tau})\backslash \Delta}{\mathbb{G}^{2}_{m}}\right]+\sum_{\substack{\beta_{k}+\beta_{l}=\beta\\ \beta_{k}\neq \beta_{l}}}^{o}\left[\frac{\mathcal{M}^{(\beta_{k},1)}(\tilde{\tau})\times \mathcal{M}^{(\beta_{l},1)}(\tilde{\tau})}{\mathbb{G}^{2}_{m}}\right]\notag\\
\end{align}
Hence
\begin{align}\label{deldel-decompos}
&
\sum_{\beta_{k}+\beta_{l}=\beta}^{o}\bar{\delta}^{(\beta_{k},1)}_{s}(\tilde{\tau})*\bar{\delta}^{(\beta_{l},1)}_{s}(\tilde{\tau})=
\sum_{\beta_{k}+\beta_{l}=\beta}^{o}\bigg(\displaystyle{\int}_{(E_{1},E_{3})\in\mathcal{M}^{(\beta_{k},1)}(\tilde{\tau})\times \mathcal{M}^{(\beta_{l},1)}(\tilde{\tau})\backslash \Delta}\left[\frac{\mathbb{P}(\operatorname{Ext}^{1}(E_{3},E_{1}))}{\mathbb{G}_{m}}\right]d\mu_{m}\bigg)\notag\\
&
+\displaystyle{\int}_{E_{1}\in\mathcal{M}^{(\frac{\beta}{2},1)}(\tilde{\tau})}\left[\frac{\mathbb{P}(\operatorname{Ext}^{1}(E_{1},E_{1}))}{\mathbb{A}^{1}\rtimes\mathbb{G}_{m}}\right]d\mu_{m}
+\left[\frac{\mathcal{M}^{(\frac{\beta}{2},1)}(\tilde{\tau})\times \mathcal{M}^{(\frac{\beta}{2},1)}(\tilde{\tau})\backslash \Delta}{\mathbb{G}^{2}_{m}}\right]\notag\\
&
+\sum_{\substack{\beta_{k}+\beta_{l}=\beta\\ \beta_{k}\neq \beta_{l}}}^{o}\left[\frac{\mathcal{M}^{(\beta_{k},1)}(\tilde{\tau})\times \mathcal{M}^{(\beta_{l},1)}(\tilde{\tau})}{\mathbb{G}^{2}_{m}}\right]+\left[\frac{\mathcal{M}^{(\frac{\beta}{2},1)}(\tilde{\tau})}{\mathbb{A}^{1}\rtimes\mathbb{G}^{2}_{m}}\right].\notag\\
\end{align}
By Equations \eqref{deldel-decompos}, \eqref{stratum-decompos2}, we obtain:
\begin{align}\label{result}
&
\bar{\epsilon}^{(\beta,2)}(\tilde{\tau})=\sum_{\beta_{k}+\beta_{l}=\beta}^{o}\bigg(\bar{\delta}^{(\beta_{k},\beta_{l},2)}_{1}(\tilde{\tau})-\frac{1}{2}\displaystyle{\int}_{\mathcal{M}^{(\beta_{k},1)}(\tilde{\tau})\times \mathcal{M}^{(\beta_{l},1)}(\tilde{\tau})\backslash \Delta}\left[\frac{\mathbb{P}(\operatorname{Ext}^{1}(E_{3},E_{1}))}{\mathbb{G}_{m}}\right]d\mu_{m}\bigg)\notag\\
&
+\bigg(\bar{\delta}^{(\frac{\beta}{2},\frac{\beta}{2},2)}_{2}(\tilde{\tau})-\frac{1}{2}\displaystyle{\int}_{\mathcal{M}^{(\frac{\beta}{2},1)}(\tilde{\tau})}\left[\frac{\mathbb{P}(\operatorname{Ext}^{1}(E_{1},E_{1}))}{\mathbb{A}^{1}\rtimes\mathbb{G}_{m}}\right]d\mu_{m}\bigg)\notag\\
&
+\bigg(\bar{\delta}^{(\frac{\beta}{2},\frac{\beta}{2},2)}_{3}(\tilde{\tau})-\frac{1}{2}\left[\frac{\mathcal{M}^{(\frac{\beta}{2},1)}(\tilde{\tau})\times \mathcal{M}^{(\frac{\beta}{2},1)}(\tilde{\tau})\backslash \Delta}{\mathbb{G}^{2}_{m}}\right]\bigg)\notag\\
&
+\sum_{\substack{\beta_{k}+\beta_{l}=\beta\\ \beta_{k}\neq \beta_{l}}}^{o}\bigg(\frac{1}{2}\cdot\bar{\delta}^{(\beta_{k},\beta_{l},2)}_{3}(\tilde{\tau})-\frac{1}{2}\left[\frac{\mathcal{M}^{(\beta_{k},1)}(\tilde{\tau})\times \mathcal{M}^{(\beta_{l},1)}(\tilde{\tau})}{\mathbb{G}^{2}_{m}}\right]\bigg)+\bigg(\bar{\delta}^{(\frac{\beta}{2},\frac{\beta}{2},2)}_{4}(\tilde{\tau})-\frac{1}{2}\left[\frac{\mathcal{M}^{(\frac{\beta}{2},1)}(\tilde{\tau})}{\mathbb{A}^{1}\rtimes\mathbb{G}^{2}_{m}}\right]\bigg).\notag\\
\end{align}
We introduce a new notation which simplifies the right hand side of Equation \eqref{result}. Let 
\begin{align}\label{list}
&
\epsilon_{1}^{(k,l)}(\tilde{\tau})=\bar{\delta}^{(\beta_{k},\beta_{l},2)}_{1}(\tilde{\tau})-\frac{1}{2}\displaystyle{\int}_{\mathcal{M}^{(\beta_{k},1)}(\tilde{\tau})\times \mathfrak{S}^{(\beta_{l},1)}\backslash \Delta}\left[\frac{\mathbb{P}(\operatorname{Ext}^{1}(E_{3},E_{1}))}{\mathbb{G}_{m}}\right]d\mu_{m}\notag\\
&
\epsilon_{2}^{(\frac{1}{2},\frac{1}{2})}(\tilde{\tau})=\bar{\delta}^{(\frac{\beta}{2},\frac{\beta}{2},2)}_{2}(\tilde{\tau})-\frac{1}{2}\displaystyle{\int}_{\mathcal{M}^{(\frac{\beta}{2},1)}}\left[\frac{\mathbb{P}(\operatorname{Ext}^{1}(E_{1},E_{1}))}{\mathbb{A}^{1}\rtimes\mathbb{G}_{m}}\right]d\mu_{m}\notag\\
&
\epsilon^{(\frac{1}{2},\frac{1}{2})}_{3}(\tilde{\tau})=\bar{\delta}^{(\frac{\beta}{2},\frac{\beta}{2},2)}_{3}(\tilde{\tau})-\frac{1}{2}\left[\frac{\mathcal{M}^{(\frac{\beta}{2},1)}(\tilde{\tau})\times \mathcal{M}^{(\frac{\beta}{2},1)}(\tilde{\tau})\backslash \Delta}{\mathbb{G}^{2}_{m}}\right]\notag\\
&
\epsilon_{3}^{(k,l)}(\tilde{\tau})=\bigg(\frac{1}{2}\cdot\bar{\delta}^{(\beta_{k},\beta_{l},2)}_{3}(\tilde{\tau})-\frac{1}{2}\left[\frac{\mathcal{M}^{(\beta_{k},1)}(\tilde{\tau})\times \mathcal{M}^{(\beta_{l},1)}(\tilde{\tau})}{\mathbb{G}^{2}_{m}}\right]\bigg)\mid_{\beta_{k}\neq \beta_{l}}\notag\\
&
\epsilon_{4}^{(\frac{1}{2},\frac{1}{2})}(\tilde{\tau})=\bar{\delta}^{(\frac{\beta}{2},\frac{\beta}{2},2)}_{4}(\tilde{\tau})-\frac{1}{2}\left[\frac{\mathcal{M}^{(\frac{\beta}{2},1)}(\tilde{\tau})}{\mathbb{A}^{1}\rtimes\mathbb{G}^{2}_{m}}\right].\notag\\
\end{align} 
Using the notation in \eqref{list}, Equation \eqref{result} is rewritten as:
\begin{align}\label{sum-result}
\bar{\epsilon}^{(\beta,2)}(\tilde{\tau})=\sum_{\beta_{k}+\beta_{l}=\beta}^{o}\epsilon_{1}^{(k,l)}(\tilde{\tau})+\epsilon_{2}^{(\frac{1}{2},\frac{1}{2})}(\tilde{\tau})+\epsilon^{(\frac{1}{2},\frac{1}{2})}_{3}(\tilde{\tau})+\sum_{\substack{\beta_{k}+\beta_{l}=\beta\\ \beta_{k}\neq \beta_{l}}}^{o}\epsilon_{3}^{(k,l)}(\tilde{\tau})+\epsilon_{4}^{(\frac{1}{2},\frac{1}{2})}(\tilde{\tau})
\end{align}
Next we show that each summand on the right hand side of \eqref{sum-result} is given by a stack function supported over virtual indecomposables. Note that by construction and Theorem \ref{theorem81}, the first summands on the right hand side of equations \eqref{list}, are each characteristic stack functions which are given as (stacky) quotients of some associated parameterizing scheme by the action of $GL_{2}(\mathbb{C})$, together with the corresponding embedding map, in other words:$$\bar{\delta}^{(\beta_{k},\beta_{l},2)}_{i}(\tilde{\tau})=\left(\left[\frac{\mathcal{M}^{(\beta_{k},\beta_{l},2)}_{i}(\tilde{\tau})}{GL_{2}(\mathbb{C})}\right], \rho_{i}\right).$$For more detail on how to obtain the schemes $\mathcal{M}^{(\beta_{k},\beta_{l},2)}_{i}(\tilde{\tau})$ look at the appendix, Definition \ref{st-ss-components}. Joyce in \cite[Section 6.2]{a33} has shown that given $\left[\left(\left[\frac{\mathcal{U}}{GL_{2}(\mathbb{C})}\right],\nu\right)\right]$ one has the following identity of stack functions:
\begin{align}\label{FGTG}
&
\left[\left(\left[\frac{\mathcal{U}}{GL_{2}(\mathbb{C})}\right],a\right)\right]=
F(GL_{2}(\mathbb{C}),\mathbb{G}^{2}_{m},\mathbb{G}^{2}_{m})\left[\left(\left[\frac{\mathcal{U}}{\mathbb{G}^{2}_{m}}\right],\mu\circ i_{1}\right)\right]\notag\\
&
+F(GL_{2}(\mathbb{C}),\mathbb{G}^{2}_{m},\mathbb{G}_{m})\left[\left(\left[\frac{\mathcal{U}}{\mathbb{G}_{m}}\right],\mu\circ i_{2}\right)\right],\notag\\
\end{align}
where 
\begin{align}\label{ration-value}
&
F(GL_{2}(\mathbb{C}),\mathbb{G}^{2}_{m},\mathbb{G}^{2}_{m})=\frac{1}{2}, \,\,\,\,
F(GL_{2}(\mathbb{C}),\mathbb{G}^{2}_{m},\mathbb{G}_{m})=-\frac{3}{4},
\end{align}
and $\mu\circ i_{1}$ and $\mu\circ i_{2}$ are the obvious embeddings. Now apply the result of Joyce \cite[Section 6.2]{a33} and obtain a decomposition of $\bar{\delta}^{(\beta_{k},\beta_{l},2)}_{i}(\tilde{\tau})$ of the following form:
\begin{align}\label{delta_i}
&
\bar{\delta}^{(\beta_{k},\beta_{l},2)}_{i}(\tilde{\tau})=\left[\left(\left[\frac{\mathcal{M}^{(\beta_{k},\beta_{l},2)}_{i}(\tilde{\tau})}{GL_{2}(\mathbb{C})}\right], \rho_{i}\right)\right]=\notag\\
&
\frac{1}{2}\left[\left(\left[\frac{\mathcal{M}^{(\beta_{k},\beta_{l},2)}_{i}(\tilde{\tau})}{\mathbb{G}^{2}_{m}}\right],\mu'_{i}\circ i_{1}\right)\right]-\frac{3}{4}\left[\left(\left[\frac{\mathcal{M}^{(\beta_{k},\beta_{l},2)}_{i}(\tilde{\tau})}{\mathbb{G}_{m}}\right],\mu'_{i}\circ i_{2}\right)\right],\notag\\
\end{align}
where $\mu'_{i}\circ i_{1}$ and $\mu'_{i}\circ i_{2}$ are the obvious embeddings.
\subsection{Computation of $\tilde{\Psi}^{\mathcal{B}_{p}}(\epsilon_{1}^{(k,l)}(\tilde{\tau}))$}
By definition, for any fixed choice of $k,l$, the objects $E_{2}\in \mathcal{M}_{1}^{(\beta_{k},\beta_{l},2)}(\tilde{\tau})$ fit into non-split exact sequences
\begin{equation}\label{non-split}
0\rightarrow E_{1}\rightarrow E_{2}\rightarrow E_{3}\rightarrow 0
\end{equation} 
where $E_{1}\ncong E_{3}$. In this case, the automorphism group of the extension \eqref{non-split} is given by $\mathbb{G}_{m}$. Since over $\mathcal{M}_{1}^{(\beta_{k},\beta_{l},2)}(\tilde{\tau})$, there exists an action of $GL_{2}(\mathbb{C})$ and the stabilizer group of each point in $\mathcal{M}_{1}^{(\beta_{k},\beta_{l},2)}(\tilde{\tau})$ is $\mathbb{G}_{m}$, then the $GL_{2}(\mathbb{C})$ action reduces to a free action of $PGL_{2}(\mathbb{C})$ on $\mathcal{M}_{1}^{(\beta_{k},\beta_{l},2)}(\tilde{\tau})$. Hence it is easy to see that there exists a map$$\pi_{1}:\mathcal{M}_{1}^{(\beta_{k},\beta_{l},2)}(\tilde{\tau})\slash PGL_{2}(\mathbb{C})\rightarrow \mathcal{M}^{(\beta_{k},1)}(\tilde{\tau})\times \mathcal{M}^{(\beta_{l},1)}(\tilde{\tau})\backslash \Delta$$which sends $E_{2}$ to $(E_{1},E_{3})$.

Consider points $(E_{1},E_{3})\in  \mathcal{M}^{(\beta_{k},1)}(\tilde{\tau})\times \mathcal{M}^{(\beta_{l},1)}(\tilde{\tau})\backslash \Delta$. Consider the fiber $\pi_{1}\mid_{(E_{1},E_{3})}$ which is given by the set of points $E_{2}\in \mathfrak{S}_{1}^{(\beta_{k},\beta_{l},2)}(\tilde{\tau})$ which fit in the exact sequence \eqref{non-split}. Dividing by the automorphism group of extension \eqref{non-split}, $\mathbb{G}_{m}$, would provide a bijective correspondence between the set of isomorphism classes of such $E_{2}$ (which still undergo a free action of $PGL_{2}(\mathbb{C})$) and the set of tuples $(E_{1},E_{3})$. Therefore there exists a bijective map between the closed points of  the fiber of $\pi_{1}$ over $(E_{1},E_{3})\in \mathcal{M}^{(\beta_{k},1)}(\tilde{\tau})\times \mathcal{M}^{(\beta_{l},1)}(\tilde{\tau})\backslash \Delta$ and the closed points of $\mathcal{M}_{1}^{(\beta_{k},\beta_{l},2)}(\tilde{\tau})\slash PGL_{2}(\mathbb{C})$ which fit into the exact sequence \eqref{non-split}, i.e the closed points of $\mathbb{P}(\operatorname{Ext}^{1}(E_{3},E_{1}))$ over $\mathcal{M}^{(\beta_{k},1)}(\tilde{\tau})\times \mathcal{M}^{(\beta_{l},1)}(\tilde{\tau})\backslash \Delta$. Now rewrite $\bar{\delta}^{(\beta_{k},\beta_{l},2)}_{1}$ as:
\begin{align}\label{re-written}
&
\bar{\delta}^{(\beta_{k},\beta_{l},2)}_{1}=\left[\frac{\mathcal{M}_{1}^{(\beta_{k},\beta_{l},2)}(\tilde{\tau})}{GL_{2}(\mathbb{C})}\right]=\left[\frac{\mathcal{M}_{1}^{(\beta_{k},\beta_{l},2)}(\tilde{\tau})\slash PGL_{2}(\mathbb{C})}{\mathbb{G}_{m}}\right]=\notag\\
&
\displaystyle{\int}_{\mathcal{M}_{1}^{(\beta_{k},\beta_{l},2)}(\tilde{\tau})\slash PGL_{2}(\mathbb{C})}\left[\frac{\operatorname{Spec}(\mathbb{C})}{\mathbb{G}_{m}}\right]d\mu_{m}=\displaystyle{\int}_{(E_{1},E_{3})\in\mathcal{M}^{(\beta_{k},1)}(\tilde{\tau})\times \mathcal{M}^{(\beta_{l},1)}(\tilde{\tau})\backslash \Delta}\left[\frac{\mathbb{P}(\operatorname{Ext}^{1}(E_{3},E_{1}))}{\mathbb{G}_{m}}\right]d\mu_{m},\notag\\
\end{align}
by Equations \eqref{re-written} and \eqref{result} the equation for $\epsilon_{1}^{(k,l)}(\tilde{\tau})$ is obtained as:
\begin{align}
&
\epsilon_{1}^{(k,l)}(\tilde{\tau})=\displaystyle{\int}_{(E_{1},E_{3})\in\mathcal{M}^{(\beta_{k},1)}(\tilde{\tau})\times \mathcal{M}^{(\beta_{l},1)}(\tilde{\tau})\backslash \Delta}\left[\frac{\mathbb{P}(\operatorname{Ext}^{1}(E_{3},E_{1}))}{\mathbb{G}_{m}}\right]d\mu_{m}\notag\\
&
-\frac{1}{2}\displaystyle{\int}_{(E_{1},E_{3})\in\mathcal{M}^{(\beta_{k},1)}(\tilde{\tau})\times \mathcal{M}^{(\beta_{l},1)}(\tilde{\tau})\backslash \Delta}\left[\frac{\mathbb{P}(\operatorname{Ext}^{1}(E_{3},E_{1}))}{\mathbb{G}_{m}}\right]d\mu_{m}\notag\\
&
=\frac{1}{2}\displaystyle{\int}_{(E_{1},E_{3})\in\mathcal{M}^{(\beta_{k},1)}(\tilde{\tau})\times \mathcal{M}^{(\beta_{l},1)}(\tilde{\tau})\backslash \Delta}\left[\frac{\mathbb{P}(\operatorname{Ext}^{1}(E_{3},E_{1}))}{\mathbb{G}_{m}}\right]d\mu_{m}.\notag\\
\end{align}
Now apply the Lie algebra morphism $\tilde{\Psi}^{\mathcal{B}_{p}}$ to $\epsilon_{1}^{(k,l)}(\tilde{\tau})$ and obtain:
\begin{align}\label{psi-e1}
&
\tilde{\Psi}^{\mathcal{B}_{p}}(\epsilon_{1}^{(k,l)}(\tilde{\tau}))=\notag\\
&
\frac{1}{2}\cdot\chi^{na}\bigg(\displaystyle{\int}_{(E_{1},E_{3})\in\mathcal{M}^{(\beta_{k},1)}(\tilde{\tau})\times \mathcal{M}^{(\beta_{l},1)}(\tilde{\tau})\backslash \Delta}\left[\frac{\mathbb{P}(\operatorname{Ext}^{1}(E_{3},E_{1}))}{\mathbb{G}_{m}}\right]d\mu_{m},\mu^{*}\nu_{\mathfrak{M}_{1}^{(\beta,2)}}\bigg)\cdot \tilde{\lambda}^{(\beta,2)}\notag\\
&
\frac{1}{2}\cdot\bigg(\displaystyle{\int}_{(E_{1},E_{3})\in\mathcal{M}^{(\beta_{k},1)}(\tilde{\tau})\times \mathcal{M}^{(\beta_{l},1)}(\tilde{\tau})\backslash \Delta}\chi(\mathbb{P}(\operatorname{Ext}^{1}(E_{3},E_{1})))\cdot \left[\frac{\operatorname{Spec}(\mathbb{C})}{\mathbb{G}_{m}}\right]d\chi\bigg)\cdot \tilde{\lambda}^{(\beta,2)}\notag\\
&
=(-1)^{1}\cdot(-1)^{\operatorname{dim} (\mathfrak{M}^{(\beta_{k},\beta_{l},2)}_{1}(\tilde{\tau}))-1}\cdot \frac{1}{2}\cdot\bigg(\displaystyle{\int}_{(E_{1},E_{3})\in\mathcal{M}^{(\beta_{k},1)}(\tilde{\tau})\times \mathcal{M}^{(\beta_{l},1)}(\tilde{\tau})\backslash \Delta}\chi(\mathbb{P}(\operatorname{Ext}^{1}(E_{3},E_{1})))d\chi\bigg)\cdot \tilde{\lambda}^{(\beta,2)}\notag\\
\end{align}
The factor of $(-1)^{\operatorname{dim} (\mathfrak{M}^{(\beta,2)}_{st-ss,\mathcal{B}_{p}}-1)}$ is due to the fact that a stacky point given as $\left[\frac{\operatorname{Spec}(\mathbb{C})}{\mathbb{G}_{m}}\right]$ has relative dimension $\operatorname{dim} (\mathfrak{M}^{(\beta_{k},\beta_{l},2)}_{1}(\tilde{\tau}))-1$ with respect to the ambient stack. Moreover, the factor of $(-1)^{1}$ is due the fact that Behrend's function over $\left[\frac{\operatorname{Spec}(\mathbb{C})}{\mathbb{G}_{m}}\right]$ detects a singular point of multiplicity $1$.
\subsection{Computation of $\tilde{\Psi}^{\mathcal{B}_{p}}(\epsilon_{2}^{(\frac{1}{2},\frac{1}{2})}(\tilde{\tau}))$}
By Equation \eqref{delta_i}
\begin{align}
&
\bar{\delta}^{(\frac{\beta}{2},\frac{\beta}{2},2)}_{2}(\tilde{\tau})=\left[\frac{\mathcal{M}_{2}^{(\frac{\beta}{2},\frac{\beta}{2},2)}(\tilde{\tau})}{GL_{2}(\mathbb{C})}\right]=\notag\\
&
\frac{1}{2}\left[\left(\left[\frac{\mathcal{M}_{2}^{(\frac{\beta}{2},\frac{\beta}{2},2)}(\tilde{\tau})}{\mathbb{G}^{2}_{m}}\right],\mu'_{2}\circ i_{1}\right)\right]-\frac{3}{4}\left[\left(\left[\frac{\mathcal{M}_{2}^{(\frac{\beta}{2},\frac{\beta}{2},2)}(\tilde{\tau})}{\mathbb{G}_{m}}\right],\mu'_{2}\circ i_{2}\right)\right].\notag\\
\end{align}
The stabilizer group of points in $\mathcal{M}^{(\frac{\beta}{2},\frac{\beta}{2},2)}_{2}(\tilde{\tau})$ is given by the automorphism group of extensions associated to those points. Since a point $p\in \mathcal{M}^{(\frac{\beta}{2},\frac{\beta}{2},2)}_{2}(\tilde{\tau})(\mathbb{C})$ is represented by an object $E_{2}$ which fits into a non-split exact sequence: \begin{equation}\label{non-split-2}
0\rightarrow E_{1}\rightarrow E_{2}\rightarrow E_{1}\rightarrow 0,
\end{equation} 
then the automorphism group of extension \eqref{non-split-2} is given by $\mathbb{A}^{1}\rtimes \mathbb{G}_{m}$. The stabilizer group (at point $p$)  of the action of $GL_{2}(\mathbb{C})$ on $\mathcal{M}^{(\frac{\beta}{2},\frac{\beta}{2},2)}_{2}(\tilde{\tau})$ is obtained by $\mathbb{A}^{1}\rtimes \mathbb{G}_{m}$. Since $E_{1}\cong E_{3}$ then for diagonal matrices (given by $\mathbb{G}_{m}^{2}=T^{GL_{2}(\mathbb{C})}$) we have that $$\mathbb{G}_{m}^{2}\cap \operatorname{Stab}_{p}(GL_{2}(\mathbb{C}))=\mathbb{G}_{m}\subset \mathbb{G}_{m}^{2},$$hence the action of $\mathbb{G}_{m}^{2}$ descends to a free action of $\mathbb{G}^{2}_{m}\slash \mathbb{G}_{m}\cong \mathbb{G}_{m}$. Hence:
\begin{equation}
\left[\left(\left[\frac{\mathcal{M}_{2}^{(\frac{\beta}{2},\frac{\beta}{2},2)}(\tilde{\tau})}{\mathbb{G}^{2}_{m}}\right],\mu'_{2}\circ i_{1}\right)\right]=\left[\left(\left[\frac{\mathcal{M}_{2}^{(\frac{\beta}{2},\frac{\beta}{2},2)}(\tilde{\tau})\slash (\mathbb{G}_{m}^{2}\slash \mathbb{G}_{m})}{\mathbb{G}_{m}}\right],\mu'_{2}\circ i_{1}\right)\right].
\end{equation}
Hence
\begin{align}\label{re-written2} 
&
\bar{\delta}^{(\frac{\beta}{2},\frac{\beta}{2},2)}_{2}(\tilde{\tau})=\frac{1}{2}\left[\left(\left[\frac{\mathcal{M}_{2}^{(\frac{\beta}{2},\frac{\beta}{2},2)}(\tilde{\tau})\slash (\mathbb{G}_{m}^{2}\slash \mathbb{G}_{m})}{\mathbb{G}_{m}}\right],\mu'_{2}\circ i_{1}\right)\right]-\frac{3}{4}\left[\left(\left[\frac{\mathcal{M}_{2}^{(\frac{\beta}{2},\frac{\beta}{2},2)}(\tilde{\tau})}{\mathbb{G}_{m}}\right],\mu'_{2}\circ i_{2}\right)\right].\notag\\
\end{align}
By equations \eqref{result} and \eqref{re-written2}:
\begin{align}\label{e2}
&
\epsilon_{2}^{(\frac{1}{2},\frac{1}{2})}(\tilde{\tau})=\frac{1}{2}\left[\left(\left[\frac{\mathcal{M}_{2}^{(\frac{\beta}{2},\frac{\beta}{2},2)}(\tilde{\tau})\slash (\mathbb{G}_{m}^{2}\slash \mathbb{G}_{m})}{\mathbb{G}_{m}}\right],\mu'_{2}\circ i_{1}\right)\right]-\frac{3}{4}\left[\left(\left[\frac{\mathcal{M}_{2}^{(\frac{\beta}{2},\frac{\beta}{2},2)}(\tilde{\tau})}{\mathbb{G}_{m}}\right],\mu'_{2}\circ i_{2}\right)\right]\notag\\
&
-\frac{1}{2}\displaystyle{\int}_{\mathcal{M}^{(\frac{\beta}{2},1)}(\tilde{\tau})}\left[\frac{\mathbb{P}(\operatorname{Ext}^{1}(E_{1},E_{1}))}{\mathbb{G}_{m}}\right].\notag\\
\end{align}

\subsubsection{Calclation of $\displaystyle{\int}_{\mathcal{M}^{(\frac{\beta}{2},1)}(\tilde{\tau})}\left[\frac{\mathbb{P}(\operatorname{Ext}^{1}(E_{1},E_{1}))}{\mathbb{G}_{m}}\right]$} Note that there exists a map from $\rho: \mathcal{M}_{2}^{(\frac{\beta}{2},\frac{\beta}{2},2)}\rightarrow \Delta\cong \mathcal{M}^{(\frac{\beta}{2},1)}(\tilde{\tau})$ which sends an object $E_{2}$ fitting in the exact sequence \eqref{non-split-2} to $E_{1}$. Moreover, given $E_{1}\in \mathcal{M}^{(\frac{\beta}{2},1)}(\tilde{\tau})$, the fiber of $\rho$ over $E_{1}$ is the set of points $p\in \mathcal{M}_{2}^{(\frac{\beta}{2},\frac{\beta}{2},2)}(\tilde{\tau})$ represented by exact sequence \eqref{non-split-2}. Now for any point $E_{1}\in \mathcal{M}^{(\frac{\beta}{2},1)}(\tilde{\tau})$, there exists a map $\pi_{Ext}:\rho^{-1} (\Delta)\rightarrow \operatorname{Ext}^{1}(E_{1},E_{1})$ which projects the points $p\in \rho^{-1}(E_{1})$ to their extension classes. Now consider an element of the extension class $\alpha\in \operatorname{Ext}^{1}(E_{1},E_{1})\backslash 0$. The pre-image, $\pi_{Ext}^{-1}(\alpha)$, consists of points $p\in \rho^{-1}(E_{1})$ which fall into the class, $\alpha$. There exists a surjective morphism:$$GL_{2}(\mathbb{C})\twoheadrightarrow \pi_{Ext}^{-1}(\alpha)$$ which is induced by the $GL_{2}(\mathbb{C})$-action on $\mathcal{M}_{2}^{(\frac{\beta}{2},\frac{\beta}{2},2)}(\tilde{\tau})$. The stabilizer group of the action of $GL_{2}(\mathbb{C})$ at point $p\in \mathcal{M}_{2}^{(\frac{\beta}{2},\frac{\beta}{2},2)}(\tilde{\tau})$ represented by $E_{2}$ sitting inside the exact sequence \ref{non-split-2} is given by $\mathbb{A}^{1}\rtimes \mathbb{G}_{m}$ hence $\operatorname{Stab}(GL_{2}(\mathbb{C}))_{p}$ for $p\in \rho^{-1}(E_{1})$ is given by $\mathbb{A}^{1}\rtimes \mathbb{G}_{m}$, then the fibers of the map $GL_{2}(\mathbb{C})\twoheadrightarrow \pi_{Ext}^{-1}(\alpha)$ are given by $\mathbb{A}^{1}\rtimes \mathbb{G}_{m}$. Now it is easy to relate the virtual poincare polynomial of $ \pi_{Ext}^{-1}(\alpha)$ to the virtual Poincar{\'e} polynomial of $GL_{2}(\mathbb{C})$.\\
In general given two algebraic spaces $X$ and $Y$ and a fibration $X\rightarrow Y$ with fibers $Z$ one has the following identity for their corresponding virtual poincare polynomials:$$\operatorname{P}_{t}(X)=\operatorname{P}_{t}(Y)\cdot \operatorname{P}_{t}(Z).$$Therefore, we obtain
\begin{equation}\label{first} 
\operatorname{P}_{t}(GL_{2}(\mathbb{C}))=\operatorname{P}_{t}(\mathbb{A}^{1}\rtimes \mathbb{G}_{m})\cdot\operatorname{P}_{t}(\pi_{Ext}^{-1}(\alpha)).
\end{equation}
On the other hand, for each $\alpha$, the free action of $\mathbb{G}^{2}_{m}\slash \mathbb{G}_{m}\cong \mathbb{G}_{m}$ on $\mathcal{M}_{2}^{(\frac{\beta}{2},\frac{\beta}{2},2)}(\tilde{\tau})$ induces a free $\mathbb{G}_{m}$-action on $\pi_{Ext}^{-1}(\alpha)$. Passing to the quotients via this action, we obtain a map$$\pi_{Ext}^{-1}(\alpha)\rightarrow \pi_{Ext}^{-1}(\alpha)\slash (\mathbb{G}^{2}_{m}\slash \mathbb{G}_{m}),$$whose fibers are given by $\mathbb{G}^{2}_{m}\slash \mathbb{G}_{m}\cong \mathbb{G}_{m}$. Hence we obtain the following relation for the virtual Poincar{\'e} polynomials:
\begin{align}\label{second}
&
 \operatorname{P}_{t}(\pi_{Ext}^{-1}(\alpha))=\operatorname{P}_{t}(\mathbb{G}^{2}_{m}\slash \mathbb{G}_{m})\cdot\operatorname{P}_{t}(\pi_{Ext}^{-1}(\alpha)\slash (\mathbb{G}^{2}_{m}\slash \mathbb{G}_{m})).\notag\\
\end{align}
By equations \eqref{first} and \eqref{second} and according to calculations in \cite[page 4]{a65} we obtain:
\begin{align}\label{first-second}
&
\operatorname{P}_{t}(\pi_{Ext}^{-1}(\alpha)\slash (\mathbb{G}^{2}_{m}\slash \mathbb{G}_{m}))=\frac{\operatorname{P}_{t}(GL_{2}(\mathbb{C}))}{\operatorname{P}_{t}(\mathbb{A}^{1}\rtimes \mathbb{G}_{m})\cdot \operatorname{P}_{t}(\mathbb{G}^{2}_{m}\slash \mathbb{G}_{m})}\notag\\
&
=\frac{(t^{4}-1)\cdot (t^{2}-1)\cdot t^{2}}{t^{2}\cdot(t^{2}-1)\cdot(t^{2}-1)}=t^{2}+1.\notag\\
\end{align}
The computation in Equation \eqref{first-second} means that for each $E_{1}\in \mathcal{M}^{(\frac{\beta}{2},1)}(\tilde{\tau})$ by passing to the quotients via the action of $\mathbb{G}^{2}_{m}\slash \mathbb{G}_{m}$, the map $\pi_{Ext}$ induces a map $$\pi_{(\mathbb{G}^{2}_{m}\slash \mathbb{G}_{m})}\mid_{E_{1}}: \rho^{-1}(E_{1})\slash (\mathbb{G}^{2}_{m}\slash \mathbb{G}_{m})\rightarrow \mathbb{P}(\operatorname{Ext}^{1}(E_{1},E_{1}))$$ whose fibers have virtual Poincar{\'e} polynomial as obtained in Equation \eqref{first-second}. Moreover, the Euler characteristic of the fibers is computed by evaluating their virtual Poincar{\'e} polynomial at $t=1$, hence for every such $E_{1}$ the fiber of the map $\pi_{(\mathbb{G}^{2}_{m}\slash \mathbb{G}_{m})}\mid_{E_{1}}$ has Euler characteristic equal to 2. We will use this factor of $2$ in Equation \eqref{simplify} below.

 Now back to Equation \eqref{e2}; Since $\epsilon_{2}^{(\frac{1}{2},\frac{1}{2})}(\tilde{\tau})$ is supported over virtual indecomposables, we can apply the Lie algebra homomorphism, $\tilde{\Psi}^{\mathcal{B}_{p}}$:
\begin{align}\label{psi-e2}
&
\tilde{\Psi}^{\mathcal{B}_{p}}(\epsilon_{2}^{(\frac{1}{2},\frac{1}{2})}(\tilde{\tau}))=\frac{1}{2}\cdot\chi^{na}\bigg(\left[\frac{\mathcal{M}_{2}^{(\frac{\beta}{2},\frac{\beta}{2},2)}(\tilde{\tau})\slash (\mathbb{G}_{m}^{2}\slash \mathbb{G}_{m})}{\mathbb{G}_{m}}\right], (\mu'_{2}\circ i_{1})^{*}\nu_{\mathfrak{M}^{(\beta,2)}_{2}}\bigg)\cdot\tilde{\lambda}^{(\beta,2)}\notag\\
&
-\frac{3}{4}\cdot\chi^{na}\bigg(\left[\frac{\mathcal{M}_{2}^{(\frac{\beta}{2},\frac{\beta}{2},2)}(\tilde{\tau})}{\mathbb{G}_{m}}\right], (\mu'_{2}\circ i_{2})^{*}\nu_{\mathfrak{M}^{(\beta,2)}_{2}}\bigg)\cdot\tilde{\lambda}^{(\beta,2)}\notag\\
&
-\frac{1}{2}\cdot\chi^{na}\bigg(\displaystyle{\int}_{\mathcal{M}^{(\frac{\beta}{2},1)}(\tilde{\tau})}\left[\frac{\mathbb{P}(\operatorname{Ext}^{1}(E_{1},E_{1}))}{\mathbb{G}_{m}}\right]d\mu_{m}, \mu^{*}\nu_{\mathfrak{M}^{(\beta,2)}_{2}}\bigg)\cdot\tilde{\lambda}^{(\beta,2)}.\notag\\
\end{align} 
Now we compute each term on the right hand side of Equation \eqref{psi-e2} separately. For the first term we obtain 
\begin{align}\label{simplify}
&
\frac{1}{2}\cdot\chi^{na}\bigg(\left[\frac{\mathcal{M}_{2}^{(\frac{\beta}{2},\frac{\beta}{2},2)}(\tilde{\tau})\slash (\mathbb{G}_{m}^{2}\slash \mathbb{G}_{m})}{\mathbb{G}_{m}}\right], (\mu'_{2}\circ i_{1})^{*}\nu_{\mathfrak{M}^{(\beta,2)}_{2}}\bigg)\cdot\tilde{\lambda}^{(\beta,2)}=\notag\\
&
\Bigg((-1)^{\operatorname{dim} (\mathfrak{M}^{(\beta,2)}_{st-ss,\mathcal{B}_{p}})-1}\cdot \frac{1}{2}\cdot \displaystyle{\int}_{E_{2}\in \mathcal{M}_{2}^{(\frac{\beta}{2},\frac{\beta}{2},2)}(\tilde{\tau})\slash (\mathbb{G}_{m}^{2}\slash \mathbb{G}_{m})}1 \left[\frac{\operatorname{Spec(\mathbb{C})}}{\mathbb{G}_{m}}\right]d\chi\Bigg)\cdot\tilde{\lambda}^{(\beta,2)}=\notag\\
&
\Bigg((-1)^{1}\cdot(-1)^{\operatorname{dim} (\mathfrak{M}^{(\beta_{k},\beta_{l},2)}_{2}(\tilde{\tau}))-1}\cdot \frac{1}{2}\cdot 2\cdot \displaystyle{\int}_{E_{1}\in \mathcal{M}^{(\frac{\beta}{2},1)}(\tilde{\tau})}\chi\bigg(\mathbb{P}(\operatorname{Ext}^{1}(E_{1},E_{1}))\bigg) d\chi\Bigg)\cdot\tilde{\lambda}^{(\beta,2)},\notag\\
\end{align}
For the second term on the right hand side of \eqref{psi-e2} we use the property (2) of the stack functions in \eqref{property2}:$$\left[\frac{\mathcal{M}_{2}^{(\frac{\beta}{2},\frac{\beta}{2},2)}(\tilde{\tau})}{\mathbb{G}_{m}}\right]=\chi(\mathcal{M}_{2}^{(\frac{\beta}{2},\frac{\beta}{2},2)}(\tilde{\tau}))\cdot\left[\frac{\operatorname{Spec}(\mathbb{C})}{\mathbb{G}_{m}}\right]$$ 
The action of $\mathbb{G}_{m}\cong \mathbb{G}^{2}_{m}\slash \mathbb{G}_{m}$ is free on $\mathcal{M}^{(\frac{\beta}{2},\frac{\beta}{2},2)}_{2}(\tilde{\tau})$ and since the $\mathbb{G}_{m}$-fixed locus via this action is empty we have that$$\chi(\mathcal{M}_{2}^{(\frac{\beta}{2},\frac{\beta}{2},2)}(\tilde{\tau}))=0.$$Therefore, the second term on the right hand side of Equation \eqref{psi-e2} vanishes:
\begin{align}
&
-\frac{3}{4}\cdot\chi^{na}\bigg(\left[\frac{\mathcal{M}_{2}^{(\frac{\beta}{2},\frac{\beta}{2},2)}(\tilde{\tau})}{\mathbb{G}_{m}}\right], (\mu'_{2}\circ i_{2})^{*}\nu_{\mathfrak{M}^{(\beta,2)}_{2}}\bigg)\cdot\tilde{\lambda}^{(\beta,2)}=0.\notag\\
\end{align}
Finally, by equations \eqref{simplify} and \eqref{psi-e2} we obtain:
\begin{align}\label{psi-e2-2}
\tilde{\Psi}^{\mathcal{B}_{p}}(\epsilon_{2}^{(\frac{1}{2},\frac{1}{2})}(\tilde{\tau}))=&
-\Bigg(\displaystyle{\int}_{E_{1}\in \mathcal{M}^{(\frac{\beta}{2},1)}(\tilde{\tau})}\chi\bigg(\mathbb{P}(\operatorname{Ext}^{1}(E_{1},E_{1}))\bigg) d\chi\Bigg)\cdot\tilde{\lambda}^{(\beta,2)}\notag\\
&
-\frac{1}{2}\cdot\Bigg(\displaystyle{\int}_{E_{1}\in \mathcal{M}^{(\frac{\beta}{2},1)}(\tilde{\tau})}\chi\bigg(\mathbb{P}(\operatorname{Ext}^{1}(E_{1},E_{1}))\bigg) d\chi\Bigg)\cdot\tilde{\lambda}^{(\beta,2)}\notag\\
&
=-\frac{3}{2}\cdot(-1)^{\operatorname{dim} (\mathfrak{M}^{(\beta_{k},\beta_{l},2)}_{2}(\tilde{\tau}))-1}\cdot \Bigg(\displaystyle{\int}_{E_{1}\in \mathcal{M}^{(\frac{\beta}{2},1)}(\tilde{\tau})}\chi\bigg(\mathbb{P}(\operatorname{Ext}^{1}(E_{1},E_{1}))\bigg) d\chi\Bigg)\cdot\tilde{\lambda}^{(\beta,2)}.\notag\\
\end{align}
\subsection{Computation of $\tilde{\Psi}^{\mathcal{B}_{p}}(\epsilon_{3}^{(\frac{1}{2},\frac{1}{2})}(\tilde{\tau}))$}
By equations \eqref{result} and \eqref{delta_i}, we re-write $\epsilon_{3}^{(\frac{1}{2},\frac{1}{2})}(\tilde{\tau})$ as follows:
\begin{align}
&
\epsilon_{3}^{(\frac{1}{2},\frac{1}{2})}(\tilde{\tau})=\frac{1}{2}\left[\left(\left[\frac{\mathcal{M}_{3}^{(\frac{\beta}{2},\frac{\beta}{2},2)}(\tilde{\tau})}{\mathbb{G}^{2}_{m}}\right],\mu'_{3}\circ i_{1}\right)\right]-\frac{3}{4}\left[\left(\left[\frac{\mathcal{M}_{3}^{(\frac{\beta}{2},\frac{\beta}{2},2)}(\tilde{\tau})}{\mathbb{G}_{m}}\right],\mu'_{3}\circ i_{2}\right)\right]\notag\\
&
-\frac{1}{2}\left[\frac{\mathcal{M}^{(\frac{\beta}{2},1)}(\tilde{\tau})\times \mathcal{M}^{(\frac{\beta}{2},1)}(\tilde{\tau})\backslash \Delta}{\mathbb{G}^{2}_{m}}\right].\notag\\
\end{align}
The set of $T^{GL_{2}(\mathbb{C})}/\mathbb{G}_{m}$-fixed points is given by the image of the map $f: \mathcal{M}^{(\frac{\beta}{2},1)}(\tilde{\tau})\times \mathcal{M}^{(\frac{\beta}{2},1)}(\tilde{\tau})\backslash \Delta\rightarrow \mathcal{M}_{3}^{(\frac{\beta}{2},\frac{\beta}{2},2)}(\tilde{\tau})$ such that for every $E_{1}\ncong E_{3}$, sends $(E_{1},E_{3})$ to $E_{1}\oplus E_{3}$. Rewrite
\begin{align}\label{darghooz}
&
\left[\frac{\mathcal{M}_{3}^{(\frac{\beta}{2},\frac{\beta}{2},2)}(\tilde{\tau})}{\mathbb{G}^{2}_{m}}\right]=\left[\frac{\operatorname{Im}(f)}{\mathbb{G}^{2}_{m}}\right]+\left[\frac{\mathcal{M}_{3}^{(\frac{\beta}{2},\frac{\beta}{2},2)}(\tilde{\tau})\backslash \operatorname{Im}(f)}{\mathbb{G}^{2}_{m}}\right],\notag\\
\end{align}
Where \eqref{darghooz} is obtained by property (1) of stack functions in \eqref{splitt-motiv}. Now by injectivity of the map $f$: 
\begin{align}
&
\left[\frac{\operatorname{Im}(f)}{\mathbb{G}^{2}_{m}}\right]=\left[\frac{\mathcal{M}^{(\frac{\beta}{2},1)}(\tilde{\tau})\times \mathcal{M}^{(\frac{\beta}{2},1)}(\tilde{\tau})\backslash \Delta}{\mathbb{G}^{2}_{m}}\right]\notag\\
\end{align}
hence
\begin{align}
&
\epsilon_{3}^{(\frac{1}{2},\frac{1}{2})}(\tilde{\tau})=\frac{1}{2}\left[\left(\left[\frac{\mathcal{M}^{(\frac{\beta}{2},1)}(\tilde{\tau})\times \mathcal{M}^{(\frac{\beta}{2},1)}(\tilde{\tau})\backslash \Delta}{\mathbb{G}^{2}_{m}}\right],\mu'_{3}\circ i_{1}\right)\right]+\frac{1}{2}\left[\left(\left[\frac{\mathcal{M}_{3}^{(\frac{\beta}{2},\frac{\beta}{2},2)}(\tilde{\tau})\backslash \operatorname{Im}(f)}{\mathbb{G}^{2}_{m}}\right],\mu'_{3}\circ i_{1}\right)\right]\notag\\
&
-\frac{3}{4}\left[\left(\left[\frac{\mathcal{M}_{3}^{(\frac{\beta}{2},\frac{\beta}{2},2)}(\tilde{\tau})}{\mathbb{G}_{m}}\right],\mu'_{3}\circ i_{2}\right)\right]-\frac{1}{2}\left[\left(\left[\frac{\mathcal{M}^{(\frac{\beta}{2},1)}(\tilde{\tau})\times \mathcal{M}^{(\frac{\beta}{2},1)}(\tilde{\tau})\backslash \Delta}{\mathbb{G}^{2}_{m}}\right],\mu'\right)\right]\notag\\
&
=\frac{1}{2}\left[\left(\left[\frac{\mathcal{M}_{3}^{(\frac{\beta}{2},\frac{\beta}{2},2)}(\tilde{\tau})\backslash \operatorname{Im}(f)}{\mathbb{G}^{2}_{m}}\right],\mu'_{3}\circ i_{1}\right)\right]-\frac{3}{4}\left[\left(\left[\frac{\mathcal{M}_{3}^{(\frac{\beta}{2},\frac{\beta}{2},2)}(\tilde{\tau})}{\mathbb{G}_{m}}\right],\mu'_{3}\circ i_{2}\right)\right]\notag\\
\end{align}
There exists a free action of $\mathbb{G}^{2}_{m}$ on $\mathcal{M}_{3}^{(\frac{\beta}{2},\frac{\beta}{2},2)}(\tilde{\tau})$. Moreover, there exists a trivial action of $\mathbb{G}_{m}\subset \mathbb{G}^{2}_{m}$ on $\mathcal{M}_{3}^{(\frac{\beta}{2},\frac{\beta}{2},2)}(\tilde{\tau})$ where $\mathbb{G}_{m}$ is given by the corresponding subgroup of diagonal matrices. Hence the free $\mathbb{G}^{2}_{m}$ action on $ \mathcal{M}_{3}^{(\frac{\beta}{2},\frac{\beta}{2},2)}(\tilde{\tau})$ is reduced to a free action of $\mathbb{G}^{2}_{m}\slash \mathbb{G}_{m}\cong \mathbb{G}_{m}$. Therefore:
\begin{align}
&
\left[\left(\left[\frac{\mathcal{M}_{3}^{(\frac{\beta}{2},\frac{\beta}{2},2)}(\tilde{\tau})\backslash \operatorname{Im}(f)}{\mathbb{G}^{2}_{m}}\right],\mu'_{3}\circ i_{1}\right)\right]=
\left[\left(\left[\frac{\left(\mathcal{M}_{3}^{(\frac{\beta}{2},\frac{\beta}{2},2)}(\tilde{\tau})\backslash \operatorname{Im}(f)\right)\slash (\mathbb{G}^{2}_{m}\slash \mathbb{G}_{m})}{\mathbb{G}_{m}}\right],\mu'_{3}\circ i_{1}\right)\right].\notag\\
\end{align}
Denote \footnote{The scheme $\mathcal{A}$ is the analog of $\tilde{Q}^{(2,2m)}_{2}$ in \cite[Lemma 5.6]{a37} when the sheaves $F$ composing the objects $E:=(F,V,\phi)$ have zero dimensional support with length $2m$. In that situation, roughly speaking, the scheme $\mathcal{M}^{(\beta,2)}_{3}(\tilde{\tau})$ is replaced by a subscheme of a product of Quot schemes. Here we are following an almost identical strategy.}
$$\mathcal{A}:=\left(\mathcal{M}_{3}^{(\frac{\beta}{2},\frac{\beta}{2},2)}(\tilde{\tau})\backslash \operatorname{Im}(f)\right)\slash (\mathbb{G}^{2}_{m}\slash \mathbb{G}_{m}).$$
With this notation
\begin{align}
&
\epsilon_{3}^{(\frac{1}{2},\frac{1}{2})}(\tilde{\tau})=\frac{1}{2}\left[\left(\left[\frac{\mathcal{A}}{\mathbb{G}_{m}}\right],\mu'_{3}\circ i_{1}\right)\right]-\frac{3}{4}\left[\left(\left[\frac{\mathcal{M}_{3}^{(\frac{\beta}{2},\frac{\beta}{2},2)}(\tilde{\tau})}{\mathbb{G}_{m}}\right],\mu'_{3}\circ i_{2}\right)\right].\notag\\
\end{align} 
Since $\epsilon_{3}^{(\frac{1}{2},\frac{1}{2})}(\tilde{\tau})$ is supported over the virtual indecomposables, one can apply the Lie algebra morphism $\tilde{\Psi}^{\mathcal{B}_{p}}$:
\begin{align}\label{psi-e3}
&
\tilde{\Psi}^{\mathcal{B}_{p}}(\epsilon_{3}^{(\frac{1}{2},\frac{1}{2})}(\tilde{\tau}))=\notag\\
&
\frac{1}{2}\cdot\chi^{na}\bigg(\left[\frac{\mathcal{A}}{\mathbb{G}_{m}}\right],(\mu'_{3}\circ i_{1})^{*}\nu_{\mathfrak{M}^{(\beta,3)}_{3}}\bigg)\cdot\tilde{\lambda}^{(\beta,2)}-\frac{3}{4}\cdot\chi^{na}\bigg(\left[\frac{\mathcal{M}_{3}^{(\frac{\beta}{2},\frac{\beta}{2},2)}(\tilde{\tau})}{\mathbb{G}_{m}}\right],(\mu'_{3}\circ i_{2})^{*}\nu_{\mathfrak{M}^{(\beta,3)}_{3}}\bigg)\cdot\tilde{\lambda}^{(\beta,2)}\notag\\
&
=\bigg((-1)^{1}\cdot\frac{1}{2}\cdot \chi(\mathcal{A})-(-1)^{1}\cdot\frac{3}{4}\cdot \chi(\mathcal{M}_{3}^{(\frac{\beta}{2},\frac{\beta}{2},2)}(\tilde{\tau}))\bigg)\cdot\tilde{\lambda}^{(\beta,2)}.\notag\\
\end{align}
\subsubsection{Calculation of $ \chi(\mathcal{A})$ and $\chi(\mathcal{M}_{3}^{(\frac{\beta}{2},\frac{\beta}{2},2)}(\tilde{\tau}))$}
There exists a map $\pi_{13}: \mathcal{M}_{3}^{(\frac{\beta}{2},\frac{\beta}{2},2)}(\tilde{\tau})\rightarrow \mathcal{M}^{(\frac{\beta}{2},1)}(\tilde{\tau})\times \mathcal{M}^{(\frac{\beta}{2},1)}(\tilde{\tau})\backslash \Delta$ which sends a point $E_{2}\cong E_{1}\oplus E_{3}$ to $(E_{1},E_{3})$. Fix such  $(E_{1},E_{3})\in \mathcal{M}^{(\frac{\beta}{2},1)}(\tilde{\tau})\times \mathcal{M}^{(\frac{\beta}{2},1)}(\tilde{\tau})\backslash \Delta$. The set of points in $\pi_{13}^{-1}(E_{1},E_{3})$ are those $E_{2}$ which can be written as $E_{2}\cong E_{1}\oplus E_{3}$. Moreover, every point in $\pi_{13}^{-1}(E_{1},E_{3})\mid_{\operatorname{Im}(f)}$ can be determined by $E_{2}\cong E_{1}\oplus E_{3}$ or $E_{2}\cong E_{3}\oplus E_{1}$, i.e by a permutation and an isomorphism, hence there exists an induced map $\pi'_{13}:\mathcal{M}_{3}^{(\frac{\beta}{2},\frac{\beta}{2},2)}(\tilde{\tau})\rightarrow \operatorname{Sym}(\mathcal{M}^{(\frac{\beta}{2},1)}(\tilde{\tau})\times \mathcal{M}^{(\frac{\beta}{2},1)}(\tilde{\tau}))$. The action of $GL_{2}(\mathbb{C})$ on $\mathcal{M}_{3}^{(\frac{\beta}{2},\frac{\beta}{2},2)}(\tilde{\tau})$ is restricted to an action on the fiber of $\pi'_{13}$ over $(E_{1},E_{3})$.

The stabilizer group of the points in $(\pi'_{13})^{-1}(E_{1},E_{3})$ is given by $\mathbb{G}^{2}_{m}$. Hence, one concludes that for every $(E_{1},E_{3})\in \operatorname{Sym}(\mathcal{M}^{(\frac{\beta}{2},1)}(\tilde{\tau})\times \mathcal{M}^{(\frac{\beta}{2},1)}(\tilde{\tau}))$ there exists a map $GL_{2}(\mathbb{C})\rightarrow \pi_{13}^{-1}(E_{1},E_{3})$ whose fiber over each point $p\in (\pi'_{13})^{-1}(E_{1},E_{3})$ is given by $\mathbb{G}^{2}_{m}$. Hence, over $(\pi'_{13})^{-1}(E_{1},E_{3})\backslash \operatorname{Im}(f)$, the action of $GL_{2}(\mathbb{C})$ is restricted to $GL_{2}(\mathbb{C})\backslash (\mathbb{G}^{2}_{m}\bigcup(\mathbb{G}^{2}_{m})^{*})$ where $\mathbb{G}^{2}_{m}$ is given by the diagonal matrices of the form:$$\mathbb{G}^{2}_{m}=\bigg\{\left(\begin{array}{cc}  g_{1}&0\\
 0&g_{2}
 \end{array}\right)
\mid g_{1}, g_{2}\in \mathbb{C}^{*}\bigg\},$$and $(\mathbb{G}^{2}_{m})^{*}$ is given by the anti-diagonal matrices of the form:$$(\mathbb{G}^{2}_{m})^{*}=\bigg\{\left(\begin{array}{cc}  0&g_{1}\\
 g_{2}&0
 \end{array}\right)\mid g_{1},g_{2}\in \mathbb{C}^{*}\bigg\}
.$$
Therefore, there exists a map $GL_{2}(\mathbb{C})\backslash (\mathbb{G}^{2}_{m}\bigcup(\mathbb{G}^{2}_{m})^{*})\rightarrow (\pi'_{13})^{-1}(E_{1},E_{3})\backslash \operatorname{Im}(f)$ whose fiber over each point $p\in (\pi'_{13})^{-1}(E_{1},E_{3})\backslash \operatorname{Im}(f)$ is given by $\mathbb{G}^{2}_{m}$. Now compute the virtual Poincar{\'e} polynomial of $(\pi'_{13})^{-1}(E_{1},E_{3})\backslash \operatorname{Im}(f)$.
\begin{align}
&
\operatorname{P}_{t}((\pi'_{13})^{-1}(E_{1},E_{3})\backslash \operatorname{Im}(f))=\frac{\operatorname{P}_{t}(GL_{2}(\mathbb{C})\backslash (\mathbb{G}^{2}_{m}\bigcup(\mathbb{G}^{2}_{m})^{*}))}{\operatorname{P}_{t}(\mathbb{G}^{2}_{m})}\notag\\
&
=\frac{(t^{4}-1)\cdot (t^{2}-1)\cdot t^{2}-(t^{2}-1)^{2}}{(t^{2}-1)^{2}}=t^{4}+t^{2}-1.\notag\\
\end{align} 
The occurrence of the term $(t^{2}-1)^{2}$ in the numerator is due to the free action of $\mathbb{G}^{2}_{m}\bigcup(\mathbb{G}^{2}_{m})^{*}$ on $\operatorname{Im}(f)$. Note that $(\pi'_{13})^{-1}(E_{1},E_{3})\backslash \operatorname{Im}(f)$ is a $\mathbb{G}_{m}$ bundle over $(\pi'_{13})^{-1}(E_{1},E_{3})\backslash \operatorname{Im}(f)\slash \mathbb{G}_{m}$, hence: 
\begin{align}\label{fibers}
&
\operatorname{P}_{t}((\pi'_{13})^{-1}(E_{1},E_{3})\backslash \operatorname{Im}(f)\slash \mathbb{G}_{m})=\frac{\operatorname{P}_{t}((\pi'_{13})^{-1}(E_{1},E_{3})\backslash \operatorname{Im}(f))}{\operatorname{P}_{t}(\mathbb{G}_{m})}\notag\\
&
=\frac{t^{4}+t^{2}-1}{(t^{2}-1)}=t^{2}+2.\notag\\
\end{align}
Therefore, for every $(E_{1},E_{3})\in \operatorname{Sym}(\mathcal{M}^{(\frac{\beta}{2},1)}(\tilde{\tau})\times \mathcal{M}^{(\frac{\beta}{2},1)}(\tilde{\tau}))$, by evaluating \eqref{fibers} at $t=1$, the fibers $(\pi'_{13})^{-1}(E_{1},E_{3})\backslash\operatorname{Im}(f)$ have Euler characteristic 3.
On the other hand, there exists a bijective map between the set of points in the $\mathbb{G}^{2}_{m}\slash \mathbb{G}_{m}$-fixed locus of $\mathcal{M}_{3}^{(\frac{\beta}{2},\frac{\beta}{2},2)}(\tilde{\tau})$ and the set of points in $\operatorname{Im}(f)\cong \mathcal{M}^{(\frac{\beta}{2},1)}(\tilde{\tau})\times \mathcal{M}^{(\frac{\beta}{2},1)}(\tilde{\tau})\backslash \Delta$. Denote by $d_{3}=\operatorname{dim} (\mathfrak{M}^{(\beta_{k},\beta_{l},3)}_{3}(\tilde{\tau}))-1$. Now rewrite Equation \eqref{psi-e3}:
\begin{align}\label{psi-e3'}
\tilde{\Psi}^{\mathcal{B}_{p}}(\epsilon_{3}^{(\frac{1}{2},\frac{1}{2})}(\tilde{\tau}))&=\bigg((-1)^{1}\cdot\frac{1}{2}\cdot \chi(\mathcal{A})-(-1)^{1}\cdot\frac{3}{4}\cdot \chi(\mathcal{M}_{3}^{(\frac{\beta}{2},\frac{\beta}{2},2)}(\tilde{\tau}))\bigg)\cdot(-1)^{d_{3}}\cdot\tilde{\lambda}^{(\beta,2)}\notag\\
&
=\bigg(-\frac{1}{2}\cdot\displaystyle{\int}_{\mathcal{A}}1d\chi_{\mathcal{A}}+\frac{3}{4}\displaystyle{\int}_{\mathcal{M}_{3}^{(\frac{\beta}{2},\frac{\beta}{2},2)}(\tilde{\tau})}1d\chi_{\mathcal{M}_{3}^{(\frac{\beta}{2},\frac{\beta}{2},2)}(\tilde{\tau})}\bigg)\cdot(-1)^{d_{3}}\cdot\tilde{\lambda}^{(\beta,2)}=\notag\\
&
-\frac{1}{2}\cdot 3\cdot \bigg(\displaystyle{\int}_{\operatorname{Sym}(\mathcal{M}^{(\frac{\beta}{2},1)}(\tilde{\tau})\times \mathcal{M}^{(\frac{\beta}{2},1)}(\tilde{\tau}))}1d\chi_{\operatorname{Sym}(\mathcal{M}^{(\frac{\beta}{2},1)}(\tilde{\tau})\times \mathcal{M}^{(\frac{\beta}{2},1)}(\tilde{\tau}))}\bigg)\cdot(-1)^{d_{3}}\cdot\tilde{\lambda}^{(\beta,2)}\notag\\
&
+\frac{3}{4}\cdot\bigg(\displaystyle{\int}_{\mathcal{M}^{(\beta_{k},1)}(\tilde{\tau})\times \mathcal{M}^{(\beta_{l},1)}(\tilde{\tau})\backslash \Delta}1d\chi_{\mathfrak{S}^{(\beta_{k},1)}\times \mathfrak{S}^{(\beta_{l},1)}\backslash \Delta}\bigg)\cdot(-1)^{d_{3}}\cdot\tilde{\lambda}^{(\beta,2)}=\notag\\
&
-\frac{1}{2}\cdot 3\cdot \frac{1}{2}\bigg(\displaystyle{\int}_{\mathcal{M}^{(\beta_{k},1)}(\tilde{\tau})\times \mathcal{M}^{(\beta_{l},1)}(\tilde{\tau})\backslash \Delta}1d\chi_{\mathcal{M}^{(\beta_{k},1)}(\tilde{\tau})\times \mathcal{M}^{(\beta_{l},1)}(\tilde{\tau})\backslash \Delta}\bigg)\cdot(-1)^{d_{3}}\cdot\tilde{\lambda}^{(\beta,2)}\notag\\
&
+\frac{3}{4}\cdot\bigg(\displaystyle{\int}_{\mathcal{M}^{(\beta_{k},1)}(\tilde{\tau})\times \mathcal{M}^{(\beta_{l},1)}(\tilde{\tau})\backslash \Delta}1d\chi_{\mathcal{M}^{(\beta_{k},1)}(\tilde{\tau})\times \mathcal{M}^{(\beta_{l},1)}(\tilde{\tau})\backslash \Delta}\bigg)\cdot(-1)^{d_{3}}\cdot\tilde{\lambda}^{(\beta,2)}=0.\notag\\
\end{align}
\subsection{Computation of $\tilde{\Psi}^{\mathcal{B}_{p}}(\epsilon_{3}^{(k,l)}(\tilde{\tau}))$}
Throughout this subsection by earlier construction $\beta_{k}\neq\beta_{l}$ and $\beta_{k}+\beta_{l}=\beta$. The action of $GL_{2}(\mathbb{C})$ restricts to an action of $\mathbb{G}^{2}_{m}$ on $\mathcal{M}_{3}^{(\beta_{k},\beta_{l},2)}(\tilde{\tau})$ hence we obtain
\begin{equation}
\bar{\delta}^{(\beta_{k},\beta_{l})}_{3}(\tilde{\tau})=\left[\frac{\mathcal{M}_{3}^{(\beta_{k},\beta_{l},2)}(\tilde{\tau})}{\mathbb{G}^{2}_{m}}\right].
\end{equation}
On the other hand, it is easily seen that there exists a bijective map $f:\mathcal{M}_{3}^{(\beta_{k},\beta_{l},2)}(\tilde{\tau})\rightarrow \mathcal{M}^{(\beta_{k},1)}(\tilde{\tau})\times \mathcal{M}^{(\beta_{l},1)}(\tilde{\tau})$ which takes $E_{2}\cong E_{1}\oplus E_{3}$  to $(E_{1},E_{3})$ where $E_{1}\ncong E_{3}$. Hence, we obtain:
\begin{equation}
\left[\frac{\mathcal{M}_{3}^{(\beta_{k},\beta_{l},2)}(\tilde{\tau})}{\mathbb{G}^{2}_{m}}\right]=\left[\frac{\mathcal{M}^{(\beta_{k},1)}(\tilde{\tau})\times \mathcal{M}^{(\beta_{l},1)}(\tilde{\tau})}{\mathbb{G}^{2}_{m}}\right].
\end{equation} 
Therefore, for every pair $(\beta_{k},\beta_{l})$ by definition of $\epsilon_{3}^{(k,l)}(\tilde{\tau})$ one obtains:
\begin{align}
&
\epsilon_{3}^{(k,l)}(\tilde{\tau})=\frac{1}{2}\cdot\bigg(\left[\frac{\mathcal{M}_{3}^{(\beta_{k},\beta_{l},2)}(\tilde{\tau})}{\mathbb{G}^{2}_{m}}\right]-\left[\frac{\mathcal{M}^{(\beta_{k},1)}(\tilde{\tau})\times \mathcal{M}^{(\beta_{l},1)}(\tilde{\tau})}{\mathbb{G}^{2}_{m}}\right]\bigg)=0,
\end{align}
hence:
\begin{equation}
\sum_{\substack{\beta_{k}+\beta_{l}=\beta \\ \beta_{k}\neq \beta_{l}}}^{o}\tilde{\Psi}^{\mathcal{B}_{p}}(\epsilon_{3}^{(k,l)}(\tilde{\tau}))=0.
\end{equation} 
\subsection{Computation of $\tilde{\Psi}^{\mathcal{B}_{p}}(\epsilon_{4}^{(\frac{1}{2},\frac{1}{2})}(\tilde{\tau}))$} By construction, there exists a bijective map between the set of points in $\mathcal{M}_{4}^{(\frac{\beta}{2},\frac{\beta}{2},2)}(\tilde{\tau})$ and $\mathcal{M}^{(\frac{\beta}{2},1)}(\tilde{\tau})$, hence one rewrites $\epsilon_{4}^{(\frac{1}{2},\frac{1}{2})}(\tilde{\tau})$ directly as follows:
\begin{align}
&
\epsilon_{4}^{(\frac{1}{2},\frac{1}{2})}(\tilde{\tau})=\frac{1}{2}\cdot\left[\left(\left[\frac{\mathcal{M}_{4}^{(\frac{\beta}{2},\frac{\beta}{2},2)}(\tilde{\tau})}{\mathbb{G}^{2}_{m}}\right],\mu'_{4}\circ i_{1}\right)\right]-\frac{3}{4}\cdot\left[\left(\left[\frac{\mathcal{M}_{4}^{(\frac{\beta}{2},\frac{\beta}{2},2)}(\tilde{\tau})}{\mathbb{G}_{m}}\right],\mu'_{4}\circ i_{2}\right)\right]\notag\\
&
-\frac{1}{2}\cdot\left[\left(\left[\frac{\mathcal{M}^{(\frac{\beta}{2},1)}(\tilde{\tau})}{\mathbb{A}^{1}\rtimes\mathbb{G}^{2}_{m}}\right],\mu'\right)\right]=
\frac{1}{2}\cdot\left[\left(\left[\frac{\mathcal{M}^{(\frac{\beta}{2},1)}(\tilde{\tau})}{\mathbb{G}^{2}_{m}}\right],\mu'_{4}\circ i_{1}\right)\right]-\frac{3}{4}\cdot\left[\left(\left[\frac{\mathcal{M}^{(\frac{\beta}{2},1)}(\tilde{\tau})}{\mathbb{G}_{m}}\right],\mu'_{4}\circ i_{2}\right)\right]\notag\\
&
-\frac{1}{2}\cdot\left[\left(\left[\frac{\mathcal{M}^{(\frac{\beta}{2},1)}(\tilde{\tau})}{\mathbb{G}^{2}_{m}}\right],\mu'\circ i_{1}\right)\right]+\frac{1}{2}\left[\left(\left[\frac{\mathcal{M}^{(\frac{\beta}{2},1)}(\tilde{\tau})}{\mathbb{G}_{m}}\right],\mu'\circ i_{2}\right)\right]\notag\\
\end{align}
which simplifies to:
\begin{align}
&
\epsilon_{4}^{(\frac{1}{2},\frac{1}{2})}(\tilde{\tau})=-\frac{1}{4}\left[\left(\left[\frac{\mathcal{M}^{(\frac{\beta}{2},1)}(\tilde{\tau})}{\mathbb{G}_{m}}\right],\mu'\circ i_{2}\right)\right].\notag\\
\end{align}
Now apply the Lie algebra homomorphism and obtain
\begin{align}\label{psi-e4}
&
\tilde{\Psi}^{\mathcal{B}_{p}}(\epsilon_{4}^{(\frac{1}{2},\frac{1}{2})}(\tilde{\tau}))=-\frac{1}{4}\cdot\chi^{na}\bigg(\left[\frac{\mathcal{M}^{(\frac{\beta}{2},1)}(\tilde{\tau})}{\mathbb{G}_{m}}\right], (\mu'\circ i_{2})^{*}\nu_{\mathfrak{M}^{(\beta,2)}_{4}}\bigg)\cdot\tilde{\lambda}^{(\beta,2)}\notag\\
&
=(-1)^{1}\cdot (-1)^{\operatorname{dim} (\mathfrak{M}^{(\beta_{k},\beta_{l},3)}_{4}(\tilde{\tau}))-1}\cdot(-\frac{1}{4})\cdot \chi(\mathcal{M}^{(\frac{\beta}{2},1)}(\tilde{\tau}))\cdot\tilde{\lambda}^{(\beta,2)}\notag\\
\end{align} 
\section{Computation of $\tilde{\Psi}^{\mathcal{B}_{p}}(\bar{\epsilon}^{(\beta,2)}(\tilde{\tau}))$}
Finally in order to compute $\tilde{\Psi}^{\mathcal{B}_{p}}(\bar{\epsilon}^{(\beta,2)}(\tilde{\tau}))$ in Equation \eqref{sum-result}, we add the contributions coming from $\epsilon_{i}^{(k,l)}$ (for all possible choices of $\beta_{k}$ and $\beta_{l}$) to $\tilde{\Psi}^{\mathcal{B}_{p}}(\bar{\delta}^{(\beta,2)}_{s}(\tilde{\tau}))$, i.e:
\begin{align}
&
\tilde{\Psi}^{\mathcal{B}_{p}}(\bar{\epsilon}^{(\beta,2)})(\tilde{\tau})=\sum_{\beta_{k}+\beta_{l}=\beta}^{o}\bigg(\tilde{\Psi}^{\mathcal{B}_{p}}(\epsilon_{1}^{(k,l)}(\tilde{\tau}))\bigg)+\tilde{\Psi}^{\mathcal{B}_{p}}(\epsilon_{2}^{(\frac{1}{2},\frac{1}{2})}(\tilde{\tau}))+\tilde{\Psi}^{\mathcal{B}_{p}}(\epsilon_{3}^{(\frac{1}{2},\frac{1}{2})}(\tilde{\tau}))\notag\\
&
+\sum_{\substack{\beta_{k}+\beta_{l}=\beta\\ \beta_{k}\neq \beta_{l}}}^{o}\bigg(\tilde{\Psi}^{\mathcal{B}_{p}}(\epsilon_{3}^{(k,l)}(\tilde{\tau}))\bigg)+\tilde{\Psi}^{\mathcal{B}_{p}}(\epsilon_{4}^{(\frac{1}{2},\frac{1}{2})}(\tilde{\tau})).\notag\\ 
\end{align}
Let $d_{i}=\operatorname{dim} (\mathfrak{M}^{(\beta_{k},\beta_{l},3)}_{i}(\tilde{\tau}))-1$ for $i=1,\cdots,4$. By equations \eqref{psi-e1}, \eqref{psi-e2-2},\eqref{psi-e3'},\eqref{psi-e4} we obtain:
\begin{align}\label{total}
&
\tilde{\Psi}^{\mathcal{B}_{p}}(\bar{\epsilon}^{(\beta,2)})(\tilde{\tau})=\chi^{na}(\left[\frac{\mathcal{M}^{(\beta,2)}_{s}(\tilde{\tau})}{GL_{2}(\mathbb{C})}\right],\nu_{\mathfrak{M}^{(\beta,2)}_{s}})\cdot\tilde{\lambda}^{(\beta,2)}\notag\\
&
+\sum_{\beta_{k}+\beta_{l}=\beta}^{o}\bigg(\frac{-1}{2}\cdot (-1)^{d_{1}}\bigg(\displaystyle{\int}_{(E_{1},E_{3})\in\mathcal{M}^{(\beta_{k},1)}(\tilde{\tau})\times \mathcal{M}^{(\beta_{l},1)}(\tilde{\tau})\backslash \Delta}\chi(\mathbb{P}(\operatorname{Ext}^{1}(E_{3},E_{1})))d\chi\bigg)\cdot \tilde{\lambda}^{(\beta,2)}\bigg)\notag\\
&
-\frac{3}{2}\cdot(-1)^{d_{2}}\cdot\Bigg(\displaystyle{\int}_{E_{1}\in \mathcal{M}^{(\frac{\beta}{2},1)}(\tilde{\tau})}\chi\bigg(\mathbb{P}(\operatorname{Ext}^{1}(E_{1},E_{1}))\bigg) d\chi\Bigg)\cdot\tilde{\lambda}^{(\beta,2)}
+\frac{1}{4}\cdot(-1)^{d_{4}}\cdot \chi(\mathcal{M}^{(\frac{\beta}{2},1)}_{s}(\tilde{\tau}))\cdot\tilde{\lambda}^{(\beta,2)}.
\notag\\
\end{align}
\section{Final computation of the invariants}
By the wall crossing computation of Joyce and Song in \cite{a30} (Equation 13.31):
\begin{align}\label{sub2}
&
\chi(\mathcal{M}^{(\frac{\beta}{2},1)}_{s}(\tilde{\tau}))=\sum_{1\leq l,\beta_{1}+\cdots+\beta_{l}=\frac{\beta}{2}}\Bigg[\frac{(1)}{l!}\cdot\bar{\chi}_{\mathcal{B}_{p}}((0,2),(\beta_{1},0))\cdot\notag\\
&
 \prod_{i=1}^{l}\bigg(\overline{DT}^{\beta_{i}}(\tau)\cdot \bar{\chi}_{\mathcal{B}_{p}}((\beta_{1}+\cdots+\beta_{i-1},2),(\beta_{i},0))\cdot (-1)^{\bar{\chi}_{\mathcal{B}_{p}}((0,2),(\beta_{1},0))+\sum_{i=1}^{l}\bar{\chi}_{\mathcal{B}_{p}}((\beta_{1}+\cdots \beta_{i-1},2),(\beta_{i},0))}\bigg)\Bigg].\notag\\
\end{align}
By Definition \ref{Bp-ss-defn} and equations \eqref{sub2} and \eqref{total} we obtain:
\begin{align}\label{total2}
&
\textbf{B}^{ss}_{p}(X,\beta,2,\tilde{\tau})=
\sum_{\beta_{k}+\beta_{l}=\beta}^{o}\Bigg[(-1)^{d_{1}+1}\cdot \frac{1}{2}\cdot\displaystyle{\int}_{(E_{1},E_{3})\in\mathfrak{S}^{(\beta_{k},1)}\times \mathfrak{S}^{(\beta_{l},1)}\backslash \Delta}\chi\bigg(\mathbb{P}(\operatorname{Ext}^{1}(E_{3},E_{1}))\bigg)d\chi\Bigg]\notag\\
&
+\frac{3}{2}\cdot(-1)^{d_{2}+1}\cdot\displaystyle{\int}_{E_{1}\in \mathcal{M}^{(\frac{\beta}{2},1)}(\tilde{\tau})}\chi\bigg(\mathbb{P}(\operatorname{Ext}^{1}(E_{1},E_{1}))\bigg) d\chi+\frac{1}{4}\cdot(-1)^{d_{4}}\cdot\notag\\
&
\sum_{1\leq l,\beta_{1}+\cdots+\beta_{l}=\frac{\beta}{2}}\Bigg[\frac{(1)}{l!}\cdot\bar{\chi}_{\mathcal{B}_{p}}((0,2),(\beta_{1},0))\cdot \prod_{i=1}^{l}\bigg(\overline{DT}^{\beta_{i}}(\tau)\cdot \bar{\chi}_{\mathcal{B}_{p}}((\beta_{1}+\cdots+\beta_{i-1},2),(\beta_{i},0))\notag\\
&
\,\,\,\,\,\,\,\,\,\,\,\,\,\,\,\,\,\,\,\,\,\,\,\,\,\,\,\,\,\,\,\,\,\,\,\,\,\,\,\,\,\,\,\,\,\,\,\,\,\,\,\,\,\,\,\,\,\,\,\,\,\,\,\,\,\,\,\,\,\,\,\,\,\,\,\,\,\,\,\,\,\,\,\,\,\,\,\,\,\,\,\,\,\,\,\,\,\,\,\,\,\,\,\,\,\,\,\,\,\,\cdot (-1)^{\bar{\chi}_{\mathcal{B}_{p}}((0,2),(\beta_{1},0))+\sum_{i=1}^{l}\bar{\chi}_{\mathcal{B}_{p}}((\beta_{1}+\cdots \beta_{i-1},2),(\beta_{i},0))}\bigg)\Bigg],\notag\\
\end{align}
where the first and second summands on the right hand side of \eqref{total2} can be calculated easily based on the geometry given, and using Grothendieck-Riemann-Roch along the fibers. 
\begin{remark}
It is easy to see that by substituting $([\mathbb{P}^{1}],2)$ in Equation \eqref{total2} one immediately obtains the result obtained in \eqref{first-answer}. As another example one may try to compute the right hand side of Equation \eqref{total2} by substituting $(\beta,2)=(2[\mathbb{P}^{1}],2)$, for instance when the base variety $X$ is given by the total space of $\mathcal{O}_{\mathbb{P}^{1}}^{\oplus 2}(-1)\rightarrow \mathbb{P}^{1}$. Assume $\chi(F)=k$. Now if $k=2q+1$, then semistability implies stability and if $k=2q$ then $F$ is given as a strictly semistable sheaf. Based on computations in  \cite{a52} and \cite{a53} for $k=2q+1$ there exist no stable sheaves with $\beta=2[\mathbb{P}^{1}]$. Now assume $k=2q$. In this case the semistable sheaves are given by $F=\mathcal{O}_{\mathbb{P}^{1}}(q-1)\oplus \mathcal{O}_{\mathbb{P}^{1}}(q-1)$\footnote{Note that by obvious reasons, a sheaf $F=\mathcal{O}_{\mathbb{P}^{1}}(a)\oplus \mathcal{O}_{\mathbb{P}^{1}}(b)$ where $a\neq b$ is unstable.} Therefore by substituting $(2[\mathbb{P}^{1}],2)$ in \eqref{total2} we see that:
\begin{equation}
\textbf{B}^{ss}_{p}(X,2[\mathbb{P}^{1}],2,\tilde{\tau})=-\frac{1}{2}(n+q)^{2}-(n+q),
\end{equation}  
The computations in this case involve arguments similar to the ones given in Section \ref{compute-example} hence we have omitted the explicit calculations here.
\end{remark}
\begin{remark}
In \cite{a37} Toda has exploited similar stratification strategy for the moduli stack of objects composed of a zero dimensional sheaf $F$ given as the quotient $\mathcal{O}_{X}^{2}\twoheadrightarrow F$ where the objects are assumed to be semistable with respect to a stability condition in the sense of Bridgeland \cite{a73}. Moreover the author has given an evidence of the integrality conjecture \cite{a42} (conjecture 6) for the corresponding partition functions associated to the moduli stack of these objects. The identities in \cite[Equation 100, 99 and 88]{a37} play an important role in the auhtor's proof of the ``\textit{integrality conjecture}". Our stratification strategy and calculation of the invariants share many similarities with those in \cite{a37}. Equations \eqref{psi-e1} and \eqref{psi-e2-2}  are analog of equations 88 and 99 in \cite{a37} respectively, and we suspect that in some cases they can be used to prove the integrality property of the corresponding partition functions for the invariants of the moduli stack of objects in $\mathcal{B}_{p}$. Note that the stability condition used in our approach is the weak stability condition (Definition \ref{weak}) used by Joyce-Song in \cite{a30} which, despite having a much simpler definition than the Bridgeland stability conditions used by Toda in \cite{a37}, shares many strong properties with them. 
\end{remark}
\appendix

\section{Locally closedness of the strata}\label{appA}
In this appendix we discuss the reason behind the assumption on locally closedness property of the strata obtained on the right hand side of Equation \eqref{st1-4}. As was mentioned in Theorem \ref{theorem81}, the moduli stack of semistable objects in $\mathcal{B}_{p}$ is obtained as a two-fold (stacky) quotient (by $GL_{r}(\mathbb{C})\times GL(V)$) of the bundle $\mathcal{P}^{\oplus r}$, parameterizing the maps $\mathcal{O}_{X}^{\oplus r}(-n)\to F$, defined over the Quot scheme, parameterizing $F$. Our strategy in this section is to investigate the locally closedness property of the strata appearing on the right hand side of \eqref{st1-4}, via analyzing the stabilizer groups of objects in their corresponding parameterizing schemes, given as subschemes of $\mathcal{P}^{\oplus r}$; 
\begin{lemma}
Fix $(\beta_{k},\beta_{l})$ such that $\beta_{k}+\beta_{l}=\beta$, then the $(k,l)$'th summand on the right hand side of Equation  \eqref{sym-kl34} is composed of locally closed strata. 
\end{lemma} 
\begin{proof} By Theorem \ref{theorem81} there exists a projection map$$\mathcal{P}^{\oplus2} \xrightarrow{\pi} \mathfrak{M}^{(\beta,2)}_{\mathcal{B}_{p}}.$$Let us use the notation for the underlying parametrizing scheme $\mathfrak{S}^{(\beta,2)}_{st-ss}(\tilde{\tau}):=\pi^{-1} \mathfrak{M}^{(\beta,2)}_{st-ss}(\tilde{\tau})\subset \mathcal{P}^{\oplus2} $. By construction in Section \ref{scheme-beta} and Remark \ref{equal-schemes} there exists an action of $GL_{2}(\mathbb{C})\times GL(V)$ on $\mathfrak{S}^{(\beta,2)}_{st-ss}(\tilde{\tau})$. This action induces an action of the corresponding Lie algebra on the tangent space of $\mathfrak{S}^{(\beta,2)}_{st-ss}(\tilde{\tau})$ given by the map:
\begin{align}\label{lie-tangent}
\mathcal{O}_{\mathfrak{S}^{(\beta,2)}_{st-ss}(\tilde{\tau})}\otimes \left(\mathfrak{g}l_{2}(\mathbb{C})\times \mathfrak{g}l(V)\right)\rightarrow T_{\mathfrak{S}^{(\beta,2)}_{st-ss}(\tilde{\tau})},
\end{align}
The dimension of the automorphism group of objects representing the elements of $\mathfrak{S}^{(\beta,2)}_{st-ss}(\tilde{\tau})$ is given by the dimension of their stabilizer group (in $GL_{2}(\mathbb{C})\times GL(V)$), which is given by the dimension of the kernel of the map in \eqref{lie-tangent}, which itself is an upper-semicontinious function. Now let us denote the $(k,l)$'th summand on the right hand side of \eqref{sym-kl34} by:
\begin{equation}\label{kl-sum}
R=\mathfrak{M}^{(\beta_{k},\beta_{l},2)}_{1}(\tilde{\tau})\bigsqcup\mathfrak{M}^{(\beta_{k},\beta_{l},2)}_{2}(\tilde{\tau})\bigsqcup\mathfrak{M}^{(\beta_{k},\beta_{l},2)}_{O-\Delta}(\tilde{\tau})\bigsqcup \mathfrak{D}^{(\beta_{k},\beta_{l},2)}(\tilde{\tau}).
\end{equation}
Then take the pre-image of $R$ under $\pi$ and denote
\begin{equation}\label{pre-image}
\pi^{-1}R:=\mathfrak{S}^{(\beta_{k},\beta_{l},2)}_{1}(\tilde{\tau})\bigsqcup\mathfrak{S}^{(\beta_{k},\beta_{l},2)}_{2}(\tilde{\tau})\bigsqcup\mathfrak{S}^{(\beta_{k},\beta_{l},2)}_{O-\Delta}(\tilde{\tau})\bigsqcup \mathfrak{D}'^{(\beta_{k},\beta_{l},2)}(\tilde{\tau})\subset\mathfrak{S}^{(\beta,2)}_{st-ss}(\tilde{\tau}).
\end{equation}
and observe that the dimension of the stabilizer group of points in each summand of of $\pi^{-1}R$ remains constant as we vary over points inside that stratum.

Now we would like to re-package the data on the right hand side of \eqref{pre-image} and write $\pi^{-1}R$ as sum of three summands, based on the dimension of stabilizer groups involved; First note that the stabilizer group of points in $\mathfrak{S}^{(\beta_{k},\beta_{l},2)}_{1}(\tilde{\tau})$ is given by the automorphism group of their corresponding objects (i.e given by extensions of stable non-isomorphic objects with classes $(\beta_{k},1)$ and $(\beta_{l},1)$) which is given by $\mathbb{G}_{m}$ (Lemma \ref{nonsp-ext}). So lets us denote this stratum (with one dimensional stabilizer group) as $R'_{1}=\mathfrak{S}^{(\beta_{k},\beta_{l},2)}_{1}(\tilde{\tau}).$ Now for the second and third summands, notice that the stabilizer groups of elements in $\mathfrak{S}^{(\beta_{k},\beta_{l},2)}_{2}(\tilde{\tau})$ and $\mathfrak{S}^{(\beta_{k},\beta_{l},2)}_{O-\Delta}(\tilde{\tau})$ are given by $\mathbb{A}_{1}\rtimes \mathbb{G}_{m}$ and $\mathbb{G}_{m}^{2}$ respectively. Therefore, let us denote the union of tho two strata (both of which have 2 dimensional stabilizer groups) as $$R'_{2}=\mathfrak{S}^{(\beta_{k},\beta_{l},2)}_{2}(\tilde{\tau})\bigsqcup \mathfrak{S}^{(\beta_{k},\beta_{l},2)}_{O-\Delta}(\tilde{\tau}).$$Finally by construction, the stabilizer group of elements in $\mathfrak{D}'^{(\beta_{k},\beta_{l},2)}(\tilde{\tau})$ is given by $\mathbb{A}_{1}\rtimes \mathbb{G}_{m}^{2}$ and so let $$R'_{3}=\mathfrak{D}'^{(\beta_{k},\beta_{l},2)}(\tilde{\tau}).$$Hence, the right hand side of \eqref{kl-sum} is written as $$\pi^{-1}R=R'_{1}\bigsqcup R'_{2}\bigsqcup R'_{3},$$ where the points over $R'_{1}$, $R'_{2}$ and $R'_{3}$ have 1-dimensional, 2-dimensional and 3-dimensional stabilizer groups respectively. It is true that $R'_{1}$, $R'_{2}$ and $R'_{3}$ are locally closed in $\mathfrak{S}^{(\beta,2)}_{ss}$ (defined in Lemma \ref{non-isom}).  

Now let us discuss the locally closedness property of these strata; First, consider $R'_{2}$; Let $\overline{\mathfrak{S}^{(\beta_{k},\beta_{l},2)}_{O-\Delta}(\tilde{\tau})}$ denote the closure of $\mathfrak{S}^{(\beta_{k},\beta_{l},2)}_{O-\Delta}(\tilde{\tau})$ in $\mathfrak{S}^{(\beta,2)}_{ss}$. Recall that by definition, the objects parametrized by elements of $\mathfrak{S}^{(\beta_{k},\beta_{l},2)}_{O-\Delta}(\tilde{\tau})$ are given by split extensions, i.e $E_{2}\cong E_{1}\oplus E_{3}$ where $E_{1}\ncong E_{3}$. The automorphism group of these objects is given by $\mathbb{G}_{m}^{2}$. Taking the closure of $\mathfrak{S}^{(\beta_{k},\beta_{l},2)}_{O-\Delta}(\tilde{\tau})$, we immediately see that the objects parametrized by $\overline{\mathfrak{S}^{(\beta_{k},\beta_{l},2)}_{O-\Delta}(\tilde{\tau})}$ are given by all split extensions of $E_{1}$ by $E_{3}$, i.e one has:
\begin{equation}
\overline{\mathfrak{S}^{(\beta_{k},\beta_{l},2)}_{O-\Delta}(\tilde{\tau})}\subset \mathfrak{S}^{(\beta_{k},\beta_{l},2)}_{O-\Delta}(\tilde{\tau})\bigcup R'_{3}.
\end{equation}    
Now take the closure of $\mathfrak{S}^{(\beta_{k},\beta_{l},2)}_{2}(\tilde{\tau})$ and obtain $\overline{\mathfrak{S}^{(\beta_{k},\beta_{l},2)}_{2}(\tilde{\tau})}$. By definition, the objects representing elements of $\mathfrak{S}^{(\beta_{k},\beta_{l},2)}_{2}(\tilde{\tau})$ are given by non-split exact sequences $$0\rightarrow E_{1}\rightarrow E_{2}\rightarrow E_{3}\rightarrow 0,$$where $E_{1}\cong E_{3}$. So it is seen that taking the closure, the objects representing the elements in the boundary of $\overline{\mathfrak{S}^{(\beta_{k},\beta_{l},2)}_{2}(\tilde{\tau})}$ are given by $E_{2}\cong E_{1}\oplus E_{3}$ where $E_{1}\cong E_{3}$, i.e:
\begin{equation}
\overline{\mathfrak{S}^{(\beta_{k},\beta_{l},2)}_{2}(\tilde{\tau})}\subset \mathfrak{S}^{(\beta_{k},\beta_{l},2)}_{2}(\tilde{\tau})\bigcup R'_{3}
\end{equation}
Since $R'_{3}\bigcap \mathfrak{S}^{(\beta_{k},\beta_{l},2)}_{O-\Delta}(\tilde{\tau})=\varnothing$ and $R'_{3}\bigcap \mathfrak{S}^{(\beta_{k},\beta_{l},2)}_{2}(\tilde{\tau})=\varnothing$, then it is seen that $\overline{\mathfrak{S}^{(\beta_{k},\beta_{l},2)}_{2}(\tilde{\tau})}$ and $\overline{\mathfrak{S}^{(\beta_{k},\beta_{l},2)}_{O-\Delta}(\tilde{\tau})}$ have empty intersections in $R'_{2}$ but non-empty intersections in $R'_{3}$ , in other words their boundary is given by a subset of $R'_{3}$ which itself is locally closed in $\mathfrak{S}^{(\beta,2)}_{ss}$. Hence $\mathfrak{S}^{(\beta_{k},\beta_{l},2)}_{O-\Delta}(\tilde{\tau})$ and $\mathfrak{S}^{(\beta_{k},\beta_{l},2)}_{2}(\tilde{\tau})$ are locally closed in $\mathfrak{S}^{(\beta,2)}_{ss}$.
\end{proof}
\begin{remark}
\emph{Here, for completeness, we discuss a second approach to the proof of the fact that $\mathfrak{S}^{(\beta_{k},\beta_{l},2)}_{2}(\tilde{\tau})$ and $\mathfrak{S}^{(\beta_{k},\beta_{l},2)}_{O-\Delta}(\tilde{\tau})$ are locally closed in $R'_{2}$}.
\end{remark}
\begin{proof}: First we state a theorem from SGA3:
\begin{theorem}(SGA3 Exp, X. Theorem 8.8)\label{SGA}
Let $\mathcal{T}$ be a commutative flat group scheme, separated of finite type over a noetherian scheme $S$ , with connected affine fibers. Let $s\in S$ and $\overline{s}$ a geometric point over $s$ and suppose: 
\begin{enumerate}
\item the reduced subscheme $(\mathcal{T}_{\overline{s}})_{red}$ of the geometric fiber $\mathcal{T}_{\overline{s}}$ is a torus and; 
\item there exists a generization $t$ of $s$ (i.e. the closure of $\{t\}$ contains $s$ ) such that $\mathcal{T}_{t}$ is smooth over $k(t) $, the residue field of $t$;
\end{enumerate} 
Then there exists an open neighborhood $\mathcal{U}$ of $s$ such that $\mathcal{T}\mid_{\mathcal{U}}$ is a torus over $\mathcal{U}$. 
\end{theorem}
Note that , by construction, $R'_{2}$ is given as a stratum of $\mathfrak{S}^{(\beta,2)}_{ss}$ over which the stabilizer groups of associated points are two dimensional. As we vary over $R'_{2}$, the stabilizer groups of points in $R'_{2}$ make a group scheme $\mathcal{G}$ over $\mathbb{C}$. According to the above theorem, for every point $p\in R'_{2}$ such that $(\mathcal{G}_{p})_{red}$ is given by the two dimensional torus ($\mathbb{G}_{m}^{2}$), there exists an open neighborhood $\mathcal{U}$ such that $\mathcal{G}\mid_{\mathcal{U}}$ is given by $\mathbb{G}_{m}^{2}$ over $\mathcal{U}$. Let $\mathcal{K}$ denote the union of all such $\mathcal{U}$. It is easily seen that every geometric point $k\in \mathcal{K}$ is given as a two dimensional torus $\mathbb{G}_{m}^{2}$ which corresponds to the stabilizer group of a point $p\in \mathfrak{S}^{(\beta_{k},\beta_{l},2)}_{O-\Delta}(\tilde{\tau})$ and there exists a bijective correspondence between such $k$ and $p$. Hence according to  Theorem \ref{SGA}, the locus of points $p\in R'_{2}$ with torus stabilizers, i.e $\mathfrak{S}^{(\beta_{k},\beta_{l},2)}_{O-\Delta}(\tilde{\tau})$ is open in $R'_{2}$. Since $R'_{2}$ is locally closed in $\mathfrak{S}^{(\beta,2)}_{ss}$ then $\mathfrak{S}^{(\beta_{k},\beta_{l},2)}_{O-\Delta}(\tilde{\tau})$ is locally closed in $\mathfrak{S}^{(\beta,2)}_{ss}$. On the other hand the complement  $(\mathfrak{S}^{(\beta_{k},\beta_{l},2)}_{O-\Delta}(\tilde{\tau}))^{c}$ is closed in $R'_{2}$. Since $\mathfrak{S}^{(\beta_{k},\beta_{l},2)}_{O-\Delta}(\tilde{\tau})\bigcap \mathfrak{S}^{(\beta_{k},\beta_{l},2)}_{2}(\tilde{\tau})=\varnothing$ then $\mathfrak{S}^{(\beta_{k},\beta_{l},2)}_{2}(\tilde{\tau})\subset (\mathfrak{S}^{(\beta_{k},\beta_{l},2)}_{O-\Delta}(\tilde{\tau}))^{c}$ which is closed in $R'_{2}$, hence $\mathfrak{S}^{(\beta_{k},\beta_{l},2)}_{2}(\tilde{\tau})$ is locally closed in $\mathfrak{S}^{(\beta,2)}_{ss}$.
\end{proof}
Now consider the $(k,l)$'th summand on the right hand side of \eqref{sym-kl34}. The action of $GL(V)$ on $ \mathfrak{S}^{(\beta,2)}_{ss}$ induces an action on each stratum. Take the quotient of each stratum by $GL(V)$ and obtain locally closed quotient stacks (disjoint from one another) as follows: 
\begin{defn}\label{st-ss-components}
Define
\begin{enumerate}
 \item $ \mathcal{M}^{(\beta_{k},\beta_{l},2)}_{1}(\tilde{\tau})=\bigg[\frac{ \mathfrak{S}^{(\beta_{k},\beta_{l},2)}_{1}(\tilde{\tau})}{GL(V)}\bigg]$.
 \item $ \mathcal{M}^{(\beta_{k},\beta_{l},2)}_{2}(\tilde{\tau})=\bigg[\frac{ \mathfrak{S}^{(\beta_{k},\beta_{l},2)}_{2}(\tilde{\tau})}{GL(V)}\bigg]$.
 \item $ \mathcal{M}^{(\beta_{k},\beta_{l},2)}_{O-\Delta}(\tilde{\tau})=\bigg[\frac{ \mathfrak{S}^{(\beta_{k},\beta_{l},2)}_{O-\Delta}(\tilde{\tau})}{GL(V)}\bigg]$.
 \item $ \mathcal{D}^{(\beta_{k},\beta_{l},2)}(\tilde{\tau})=\bigg[\frac{ \mathfrak{D}'^{(\beta_{k},\beta_{l},2)}(\tilde{\tau})}{GL(V)}\bigg]$.
 \end{enumerate}
 \end{defn}
\begin{defn}\label{BpR-st-ss}
Define $\mathcal{M}^{(\beta,2)}_{st-ss}(\tilde{\tau})=\left[\frac{\mathfrak{S}^{(\beta,2)}_{st-ss}(\tilde{\tau})}{GL(V)}\right].$
\end{defn}
By definitions \ref{st-ss-components} and \ref{BpR-st-ss} and Equation \eqref{sym-kl34} one obtains:
\begin{align}\label{sym-kl3}
 &
 \mathcal{M}^{(\beta,2)}_{st-ss}(\tilde{\tau})=
\bigsqcup^{o}_{\beta_{k}+\beta_{l}=\beta}\bigg(  \mathcal{M}^{(\beta_{k},\beta_{l},2)}_{1}(\tilde{\tau})\bigsqcup \mathcal{M}^{(\beta_{k},\beta_{l},2)}_{2}(\tilde{\tau})\bigg)\bigsqcup
\bigsqcup^{u-o}_{\beta_{k}+\beta_{l}=\beta}\bigg( \mathcal{M}^{(\beta_{k},\beta_{l},2)}_{O-\Delta}(\tilde{\tau})\bigsqcup  \mathcal{D}^{(\beta_{k},\beta_{l},2)}(\tilde{\tau})\bigg)
 \end{align} 
  Now by our construction taking the quotient of $\mathcal{M}^{(\beta,2)}_{st-ss}(\tilde{\tau})$ one more time by the action of $GL_{2}(\mathbb{C})$, will naturally produce back the quotient stack $\mathfrak{M}^{(\beta,2)}_{\mathcal{B}_{p},ss}(\tilde{\tau})$ and the right hand side of \eqref{sym-kl34}, in other words, we can  easily see that:
 \begin{enumerate}
 \item $ \mathfrak{M}^{(\beta_{k},\beta_{l},2)}_{1}(\tilde{\tau})=\bigg[\frac{ \mathcal{M}^{(\beta_{k},\beta_{l},2)}_{1}(\tilde{\tau})}{GL_{2}(\mathbb{C})}\bigg]$.
 \item $ \mathfrak{M}^{(\beta_{k},\beta_{l},2)}_{2}(\tilde{\tau})=\bigg[\frac{ \mathcal{M}^{(\beta_{k},\beta_{l},2)}_{2}(\tilde{\tau})}{GL_{2}(\mathbb{C})}\bigg]$.
 \item $ \mathfrak{M}^{(\beta_{k},\beta_{l},2)}_{O-\Delta}(\tilde{\tau})=\bigg[\frac{ \mathcal{M}^{(\beta_{k},\beta_{l},2)}_{O-\Delta}(\tilde{\tau})}{GL_{2}(\mathbb{C})}\bigg]$.
 \item $ \mathfrak{D}^{(\beta_{k},\beta_{l},2)}(\tilde{\tau})=\bigg[\frac{ \mathcal{D}^{(\beta_{k},\beta_{l},2)}(\tilde{\tau})}{GL_{2}(\mathbb{C})}\bigg]$.
 \end{enumerate}

%\backmatter
\bibliographystyle{plain}
\bibliography{ref1}

\noindent {\tt{artan@mit.edu}} \\
\noindent {\tt{artan.sheshmani@gmail.com}} \\
\noindent {\tt{Massachusetts Institute of Technology (MIT), Department of Mathematics}} \\
\noindent {\tt{Building 2, Room 2-247, 77 Massachusetts Avenue, Cambridge, MA, USA 02139-4307}} \\
\noindent {\tt{Ohio State University, Department of Mathematics}} \\
\noindent {\tt{231 West 18 Avenue, Columbus, Ohio, USA, 43210}} \\
\end{document}